\documentclass[11pt,reqno]{amsart}
\usepackage{color, amsmath,amssymb, amsfonts, amstext,amsthm, latexsym}
\allowdisplaybreaks
\newcommand{\R}{{\mathbb R}}

\newcommand{\RR}{{\mathbb R}}

\newcommand{\EX}{{\mathbb E}}
\newcommand{\EE}{{\mathbb E}}

  \newcommand{\PX}{{\mathbb P}}
\newcommand{\PP}{{\mathbb P}}

\newcommand{\s}{\sigma}
\newcommand{\ts}{\tilde\sigma}
\newcommand{\e}{\varepsilon}
\newcommand{\ro}{\varrho}
\newcommand{\tro}{\tilde\varrho}
\renewcommand{\a}{\alpha}
\renewcommand{\b}{\beta}
\newcommand{\om}{\omega}
\newcommand{\Om}{\Omega}

\newcommand{\de}{\delta}
\newcommand{\la}{\lambda}

\newcommand{\hf}{\frac{1}{2}}

\newcommand{\cF}{{\cal F}}
\renewcommand{\cF}{\mathcal F}
\newcommand{\cL}{{\mathcal L}}
\newcommand{\cA}{{\mathcal A}}
\newcommand{\cB}{{\mathcal B}}
\newcommand{\cE}{{\mathcal E}}
\newcommand{\cI}{{\mathcal I}}
\newcommand{\cN}{{\mathcal N}}

\newcommand{\cR}{{\mathcal R}}
\newcommand{\HH}{{\mathcal H}}

\newcommand{\tW}{{\widetilde{W}}}

\numberwithin{equation}{section}
\newtheorem{theorem}{Theorem}[section]

\newtheorem{lemma}[theorem]{Lemma}
\newtheorem{remark}[theorem]{Remark}
\newtheorem{prop}[theorem]{Proposition}

\begin{document}

\title[ Stochastic 2D hydrodynamical systems]
{Stochastic 2D hydrodynamical  systems: \\ Wong-Zakai approximation
and Support theorem}

\author[I. Chueshov and A. Millet]
{Igor Chueshov and Annie Millet  }

\address[I.~Chueshov]
{Department of  Mechanics and Mathematics\\
Kharkov National University\\
4~Svobody Square\\
61077, Kharkov, Ukraine } 
%tel/fax: +38 057 707 5202} 
\email[I. Chueshov]{chueshov@univer.kharkov.ua}

\address[A.~Millet]
{ SAMM,  EA 4543  Universit\'{e} Paris 1,
 Centre Pierre Mend\`{e}s France,
90 rue de Tolbiac, F- 75634 Paris Cedex 13, France {\it and}
Laboratoire  PMA (UMR 7599)%,
   %    Universit\'es Paris~6-Paris~7, Bo\^{\i}te Courrier 188,
    %      4 place Jussieu, 75252 Paris Cedex 05, France
   } \email[A.
~Millet]{amillet@univ-paris1.fr {\it and} annie.millet@upmc.fr}



\begin{abstract}
We deal with a class of abstract nonlinear stochastic models
with multiplicative noise, which
covers many 2D hydrodynamical models including the  2D Navier-Stokes
equations, 2D MHD  models and 2D magnetic B\'{e}nard problems  as well as
some shell models of turbulence.
 Our main result  describes
the support of the  distribution of  solutions. Both inclusions are
proved by means of a general Wong--Zakai type result
 of convergence in probability for nonlinear stochastic PDEs driven
by a Hilbert-valued Brownian motion and some adapted finite dimensional
approximation of this process.
\end{abstract}

\maketitle

%%%%%%%%%%%%%%%%%%%%%%%%%%%%%%%%%%%%%%%%%%%%%%%%%%%%%%%%
%%%%%%%%%%%%%%%%%%%%%%%%%%%%%%%%%%%%%%%%%%%%%%%%%%%%%%%%
\section{Introduction}\label{s1}
Our goal in this paper is to continue the unified  investigation of
statistical properties of some stochastic   2D hydrodynamical models
which  started in our previous paper \cite{ch-mi}. The model
introduced there  covers a wide class of mathematical coupled models
from 2D fluid dynamics. This class includes the  2D Navier-Stokes
equations and also  some other classes of two dimensional
hydrodynamical models  such as the  magneto-hydrodynamic  equation,
the Boussinesq model for the  B\'{e}nard convection and the 2D
magnetic B\'{e}nard problem. We also cover the case of  regular
higher dimensional problems such as the 3D Leray $\alpha$-model for
the Navier-Stokes equations  and some shell models
 of turbulence. For details we refer to \cite[Sect.2.1]{ch-mi}.
\par
 Our unified
approach is based on an abstract stochastic evolution equation in
some Hilbert space of the form
\begin{equation} \label{abstr-1}
\partial_t u  + A u + B(u, u) +
R(u)=  \Xi(u)\, \dot{W},\quad u|_{t=0}=\xi,
\end{equation}
where $ \dot{W}$ is a multiplicative  noise white in
time with spatial correlation. The hypotheses concerning
the linear operator $A$, the bilinear mapping $B$
and the  operators $R$ and $\Xi$   are stated  below.
These hypotheses guarantee unique solvability of
problem \eqref{abstr-1}.
\par
 For general abstract stochastic
evolution equations in infinite dimensional spaces we  refer to
\cite{PZ92}. However the hypotheses in \cite{PZ92} do not cover our
hydrodynamical type model. We also note that  stochastic Navier-Stokes
equations were studied by many authors (see, e.g.,
\cite{CG94,FG95,MS02,VKF} and the references therein). In
\cite{ch-mi} we prove existence,   uniqueness and provide a priori
estimates
  for a  weak (variational) solution
to the  abstract problem of the form \eqref{abstr-1},
where the forcing term may also include a stochastic
control term with a multiplicative coefficient.
In all the concrete hydrodynamical examples mentioned
above, the diffusion coefficient may contain a small multiple of the
gradient of the solution. This result contains the
corresponding existence and uniqueness theorems and a priori bounds
for 2D Navier-Stokes equations,
the Boussinesq model of  the  B\'{e}nard convection, and
also for the GOY and Sabra shell models  of turbulence.
 Theorem~2.4 \cite{ch-mi} generalizes the existence result
for Boussinesq or  MHD equations  given in \cite{Ferrario} or \cite{BaDP}  to the case of
multiplicative noise (see also \cite{DM}) and also covers  new situations such as the 2D
magnetic B\'{e}nard problem   or  the 3D Leray $\alpha$-model.
Our main result  in \cite{ch-mi} is
a Wentzell-Freidlin type large
deviation principle (LDP) in an appropriate Polish space $X$
for stochastic equations of the form
\eqref{abstr-1} with   $\Xi^\e:=\sqrt{\e}\s$   as $\e \to 0$, which describes the exponential
rate
of convergence  of the  solution $u:=u^\e$ to the deterministic solution  $u^0$.
One of the key arguments is a time increment control which provides the
weak convergence needed in order to prove the large deviations principle.
 We refer to  \cite{ch-mi} for detailed discussion  and references.
\par
Another classical problem is that of approximation of solutions in terms of a simpler
model, where the stochastic integral is changed into  a "deterministic" one, replacing
the noise by  a random element of its reproducing kernel Hilbert space, such as a finite dimensional
space approximation of its piecewise linear interpolation on a time grid.  This is the celebrated
Wong-Zakai approximation of the solution and once more the lack of continuity of the
solution as a function of the noise has to be dealt with. This requires to make a drift correction
coming from the fact that  the It\^o integral is replaced by a
Stratonovich one. For finite-dimensional diffusion processes, 
this kind of
approximation  is well-known (see, e.g., \cite{IW},
\cite{protter}, \cite{SV},   \cite{WZ} and also the survey
\cite{twardo} and the references therein).
There is a substantial number of publications devoted to
 Wong--Zakai  approximations of infinite-dimensional
stochastic equations. 
For instance, in 
  \cite{gyongy}, and  \cite{gyongy-1}  I.~Gy\"ongy 
established   Wong-Zakai approximations of linear  parabolic evolution equations
satisfying a coercivity and stochastic parabolicity condition 
and subject to a random finite-dimensional
perturbation driven by a continuous martingale; 
some applications to filtering and
some stochastic dynamo models are given. Z.~Brzezniak, M.~Capi\'nsky and F.~Flandoli \cite{brzecapinfland}
 studied a similar
problem for a linear parabolic equation subject 
to an perturbation driven by an infinite dimensional
Gaussian noise. 
In \cite{brzefland} Z.~Brzezniak and F.~Flandoli and in \cite{gyongy-2} I. Gy\"ongy and A.~Shmatkov 
obtained  some more refined convergence,
either a.s. or with some rate of convergence.
  Let us also mention  the reference \cite{ChVui2004}  by I.~Chueshov and P.~Vuillermot 
which deals with  semilinear non-autonomous parabolic PDE systems 
perturbed by multiplicative noise and considers some  applications
to invariance of deterministic sets    
with respect to the corresponding evolutions
and the reference \cite{tessitorezab} by G.~Tessitore and J.~Zabczyk 
which studies Wong-Zakai type approximations of mild solutions 
to abstract semilinear  parabolic type  equations. 
Similar Wong Zakai approximations were proven in H\"older spaces 
 by A.~Millet and M.~Sanz-Sol\'e
in   \cite{M-SS1} 
for a semi-linear stochastic hyperbolic equation and 
 by V.~Bally, A.~Millet and M.~Sanz-Sol\'e  in \cite{BMSS}
for  the  one-dimensional heat equation with a multiplicative stochastic perturbation driven
by a space-time white noise. Similarly, in  \cite{CW-M}   C.~Cardon-Weber and A.~Millet  proved 
similar Wong-Zakai  approximation  results    in various topologies for
 the stochastic one-dimensional Burgers equation 
with a multiplicative perturbation driven by space-time white noise. In references
\cite{gyongy}, \cite{gyongy-1}, \cite{M-SS1}, \cite{BMSS} and \cite{CW-M}, 
these approximations were the main step to characterize the support of the distribution
of these stochastic evolution equations. 
\par 
Except for \cite{CW-M} which studies a toy model of turbulence and has a truly non-linear
feature, all the above papers require linear or Lipschitz assumptions on the coefficients 
which do not cover non-linear models such as the 
Navier-Stokes equations or general hydrodynamical models.
Some result on the  Wong-Zakai approximation for the  2D Navier--Stokes system
is proved by W.~Grecksch and B.~Schmalfuss  in \cite{GreSch95}, but for a  linear finite dimensional noise, 
which is a particular case of the framework used in this paper.
We also mention the paper of K.~Twardowska 
\cite{Twa-padova} which claims  the convergence
of Wong-Zakai  approximations for 2D Navier--Stokes system
with a rather general diffusion part.  However, the argument 
used in \cite{Twa-padova} is incomplete and we were not able 
to fill the gaps.
\par
 In the same spirit, a slightly more general  approximation result,
using adapted linear interpolation of the Gaussian noise,
 provides  a  description of the  support  for the abstract system
\eqref{abstr-1} and thus covers a wide class of hydrodynamical
models. This is the well-known  Stroock-Varadhan characterization of the
support of the distribution of the solution in the Polish space $X$
where it lives.
Note that the approach  used   in the present paper   
is different from the original one of D.W.~Stroock and S.R.S.~Varadhan \cite{SV},  and is 
 similar to that introduced in
\cite{M-SS1} and \cite{M-SS2}. Indeed, only one result of
convergence in probability provides both inclusions needed to
describe the support of the distribution. See also V.~Mackevi\v{c}ius \cite{Ma} as well as S.~Aida, S.~Kusuoka and
 D.~Stroock \cite{AKS}, where related results were obtained for
diffusion processes using a non-adapted linear time interpolation of
the noise. The technique proposed in \cite{Ma} was used in \cite{gyongy} and \cite{gyongy-1}
to characterize the support of the solution of stochastic quasi-linear parabolic evolution equations
in (weighted) Sobolev spaces. 
 The references \cite{BMSS},
\cite{CW-M} and the paper by T.~Nakayama  \cite{Nak}
 establish  a support theorem for the one-dimensional heat or Burgers
 equation,  and  for general mild solutions to  semi-linear abstract
parabolic equations  along the same line.  However, unlike these  references in a parabolic setting,
the argument used in this paper does not rely on the Green  
 function 
 associated with the second order differential operator   and deals with a nonlinear physical model.
 As in \cite{ch-mi}, the control equation is needed with some  control
defined in terms of both an element of the Reproducing Kernel
Hilbert Space of the driving Brownian motion
and an adapted linear finite-dimensional approximation of this Brownian.
For this class of control equations we first establish a  Wong-Zakai type approximation
theorem (see Theorem~\ref{th:WZ-appr}), which is the main step of our proof and, as we believe,
has an independent interest.
Note that all previous works were using intensively a time H\"older regularity of the solution of
either the diffusion or the evolution equation. Such a time regularity is out of reach
for the Navier-Stokes equations and the general hydrodynamical models we cover.
\par
A key ingredient of the proof of the main convergence  theorem
 is some "time integrated"  time increment which can be  obtained    with a better
speed of convergence to zero
than that needed in \cite{ch-mi}.
 To our best knowledge,   there is only one publication   related
to the support characterization  of solutions to 
2D hydrodynamical models.  
The short note \cite{Twa-gyor},  which states a characterization of the support
for 2D Navier-Stokes equations  with the Dirichlet boundary conditions,
 does not provide a detailed  proof  and refers to
\cite{Twa-padova} where the  argument is
incomplete.
\medskip 
\par The paper is
organized as follows. In Section \ref{s2} we recall the mathematical
model introduced in \cite{ch-mi}. In this section we also formulate
our abstract hypotheses. Our main results are stated
 in Section \ref{sec:support} under some additional integrability
property on the solution. We first  formulate the Wong-Zakai type
approximation Theorem~\ref{th:WZ-appr}, which is the main  tool to characterize
the support of the distribution of the solution to the stochastic
hydrodynamical equations. This characterization is given in Theorem~\ref{th:support}  and we show how
the support characterization can be deduced. In
Section~\ref{sect:prelim} we provide some preliminary step where the noise is truncated.
Section  \ref{pr-conv} contains the proof of  Theorem~\ref{th:WZ-appr}.
It heavily depends on the time increment
speed of convergence,  which is proved in Section \ref{time-in}. In
the appendix (see Section~\ref{appendix}) we discuss with details
the way our result can be applied  to different classes of
hydrodynamical models and give conditions which  ensure that the
solution fulfills
 the extra integrability assumption we have imposed 
 (see \eqref{crit-bound}).  

%%%%%%%%%%%%%%%%%%%%%%%%%%%%%%%%%%%%%%%%%%%%%%%%%%%%%%
%%%%%%%%%%%%%%%%%%%%%%%%%%%%%%%%%%%%%%%%%%%%%%%%%%%%%%
\section{Description of the model} \label{s2}

\subsection{Deterministic analog}
Let $(H, |.|)$ denote a separable Hilbert space, $A$ be an (unbounded) self-adjoint
positive linear operator on $H$. Set $V=Dom(A^{\frac{1}{2}})$.
For $v\in V$  set   $\|v\|= |A^{\frac{1}{2}} v|$.
Let $V'$ denote the dual of $V$ (with respect to the
inner product $(.,.)$ of $H$).
Thus we have the Gelfand  triple $V\subset H\subset V'$.
Let $\langle u,v\rangle $ denote the
duality between $u\in V$ and $v\in V'$
such that  $\langle u,v\rangle =(u,v)$ for $u\in V$, $v\in H$,
 and let $B : V\times V \to
V'$ be a  mapping satisfying the condition (\textbf{B})
given below.
\par
The goal of this paper is to study stochastic perturbations of the following
abstract model in $H$
\begin{equation} \label{u0}
\partial_t u(t)  +   A u(t) + B\big(u(t),u(t) \big) + R u(t) =  f,
\end{equation}
where   $R$ is a  continuous  operator in $H$.
We assume that the mapping $B : V\times V \to
V'$  satisfies the following  antisymmetry and bound conditions:
\medskip\par
\noindent \textbf{Condition (B):}{\it
\begin{enumerate}
  \item  $B : V\times V \to V'$ is a  bilinear continuous  mapping.
  \item For $u_i\in V$, $i=1,2,3$,
\begin{equation} \label{as}
 \langle B(u_1, u_2)\, ,\, u_3\rangle  = - \,\langle  B(u_1, u_3)\, ,\, u_2\rangle.
\end{equation}
  \item
 There exists a Banach (interpolation) space
${\mathcal H} $ possessing the properties\\
 (i) $V\subset
{\mathcal H}\subset H;$\\
 (ii) there exists a constant $a_0>0$ such that
\begin{equation} \label{interpol}
\|v\|_\HH^2 \leq a_0 |v|\, \|v\|\quad \mbox{for any $v\in V$};
\end{equation}
(iii)
for every $\eta >0$ there exists $C_\eta >0$
 such that
\begin{align} \label{boundB}
| \langle B(u_1, u_2)\,, \, u_3\rangle | &\leq \eta\,  \|u_3\|^2 + C_\eta \, \|u_1\|_\HH^2 \,
 \|u_2\|_\HH^2, \quad for \;  u_i\in V, \; i=1,2,3.
\end{align}
\end{enumerate}
}
Note (see \cite[Remark 2.1]{ch-mi}) that the upper estimate  in \eqref{boundB}
can also be  written in the following two equivalent forms:
\par
{\em (iii-a)} there exist positive constants $C_1$ and $C_2$ such that
\begin{align} \label{boundB-eq1}
| \langle B(u_1, u_2)\,, \, u_3\rangle | &\leq C_1  \|u_3\|^2 + C_2 \, \|u_1\|_\HH^2 \,
 \|u_2\|_\HH^2, \quad \mbox{\rm for }\; u_i\in V,\; i=1,2,3;
\end{align}
\par
{\em (iii-b)} there exists a constant $C>0$ such that  for $u_i\in V,\; i=1,2,3$ we have:
\begin{equation}\label{preB}
| \langle B(u_1, u_2)\,, \, u_3\rangle | = | \langle B(u_1, u_3)\,, \, u_2\rangle |
 \leq C  \, \|u_1\|_\HH \, \|u_2\| \,
 \|u_3\|_\HH .%, \quad \mbox{\rm for } \;  u_i\in V,\; i=1,2,3.
\end{equation}
\par
For $u\in V$ set   $B(u):=B(u,u)$; with this notation,
relations \eqref{as}, \eqref{interpol} and
\eqref{preB} yield  for every $\eta >0$ the existence of $C_\eta >0$
such that for $u_1, u_2\in V$,
\begin{equation} \label{boundB1}
| \langle B(u_1)\, , \, u_2\rangle | \leq \eta\,  \|u_1\|^2 + C_\eta \, |u_1|^2 \,
 \|u_2\|_\HH^4.
\end{equation}
Relations \eqref{as} and   \eqref{boundB1}   imply that for any $\eta>0$ there exists $C_\eta>0$
such that
\begin{equation}
|\langle B(u_1)-B(u_2)\, ,\, u_1-u_2\rangle| =
 |\langle B(u_1-u_2), u_2\rangle | \leq \eta \|u_1-u_2\|^2 +
C_\eta\,  |u_1 - u_2|^2\, \|u_2\|_\HH^4 \label{diffB1}
\end{equation}
for all $u_1,u_2\in V$. As it was explained in \cite{ch-mi}
the main motivation for  condition ({\bf B}) is that it  covers a wide class
of 2D hydrodynamical models including
 Navier-Stokes equations,
magneto-hydrodynamic  equations,
 Boussinesq model for the  B\'{e}nard convection,
 magnetic B\'{e}nard problem,
3D Leray $\alpha$-model for Navier-Stokes equations,
Shell models of turbulence (GOY, Sabra, and dyadic  models).

\subsection{Noise}
 We will consider a
stochastic   external random force   $f$
in   equation  \eqref{u0},  driven by a Wiener process $W$
and whose intensity may depend on the solution $u$. More precisely,
let $Q$ be a linear
positive  operator in the Hilbert space $H$ which belongs to the trace class,
and hence is  compact. Let $H_0 = Q^{\frac12} H$. Then $H_0$ is a
Hilbert space with the scalar product
$$
(\phi, \psi)_0 = (Q^{-\frac12}\phi, Q^{-\frac12}\psi),\; \forall
\phi, \psi \in H_0,
$$
together with the induced norm $|\cdot|_0=\sqrt{(\cdot,
\cdot)_0}$. The embedding $i: H_0 \to  H$ is Hilbert-Schmidt and
hence compact, and moreover, $i \; i^* =Q$.
Let $L_Q\equiv L_Q(H_0,H) $ denote the space of linear operators $S:H_0\mapsto H$ such that
$SQ^{\frac12}$ is a Hilbert-Schmidt operator  from $H$ to $H$. The norm on the space $L_Q$ is
  defined by  $|S|_{L_Q}^2 =tr (SQS^*)$,  where $S^*$ is the adjoint operator of
$S$. The $L_Q$-norm can be also written in the form:
\begin{equation}\label{LQ-norm}
 |S|_{L_Q}^2=tr ([SQ^{1/2}][SQ^{1/2}]^*)=\sum_{k\geq 1} |SQ^{1/2}\psi_k|^2=
 \sum_{k\geq 1} |[SQ^{1/2}]^*\psi_k|^2
\end{equation}
for any orthonormal basis  $\{\psi_k\}$ in $H$.
\par
Let   $W(t)$ be a   Wiener process  defined   on a
probability space $(\Om, \cF,  \PX)$, taking values in $H$
and with covariance operator $Q$. This means that $W$ is Gaussian, has independent
time increments and that for $s,t\geq 0$, $f,g\in H$,
\[
\EE  (W(s),f)=0\quad\mbox{and}\quad
\EE  (W(s),f) (W(t),g) = \big(s\wedge t)\, (Qf,g).
\]
We also have the representation
\begin{equation}\label{W-n}
W(t)=\lim_{n\to\infty} W_n(t)\;\mbox{ in }\; L^2(\Om; H)\; \mbox{ with }
W_n(t)=\sum_{1\leq j\leq n} q^{1/2}_j \beta_j(t) e_j,
\end{equation}
where  $\beta_j$ are  standard (scalar) mutually independent Wiener processes,
$\{ e_j\}$ is an  orthonormal basis in $H$ consisting of eigen-elements of $Q$, with
$Qe_j=q_je_j$. For details concerning this Wiener process
we refer  to \cite{PZ92}, for instance.
Let $(\cF_t, t\geq 0)$ denote the Brownian filtration, that is the smallest
right-continuous complete filtration with respect to which $(W(t), t\geq 0)$ is adapted.
\par
We now define  some  adapted approximations of the processes $\beta_j$
and $W$.
For all integers  $n\ge 1$ and $k=0,1,\ldots,2^n$,  set $t_k=kT2^{-n}$ and
define step functions $\underline{s}_n, s_n, \bar{s}_n  : [0,T]\to [0,T]$
by the formulas
\begin{equation}\label{s-funct}
\underline{s}_n=t_k, \; s_n=t_{k-1}\vee 0, \; \bar{s}_n=t_{k+1}
\quad\mbox{for}\quad s\in [t_k,t_{k+1}[.
\end{equation}
It is clear that  $ s_n<\underline{s}_n<\bar{s}_n$. Now we set
$\dot{\tilde{\beta}}_j^n(s)=
T^{-1}2^n\left(\beta_j( \underline{s}_n)-\beta_j(s_n) \right)$,
for every $s\in [0,T]$, and thus  we obtain
an adapted approximation for $\dot{\beta}_j(s)$  given by the formula
\begin{equation}\label{beta-appr}
\dot{\tilde{\beta}}_j^n(s)=T^{-1}2^n\left[\beta_j(t_k)-\beta_j(t_{k-1}\vee 0) \right],
\quad\mbox{for}\quad s\in [t_k,t_{k+1}[.
\end{equation}
Clearly $\dot{\tilde{\beta}}_j^n(s)\equiv 0$ for $s\in [0,t_{1}[$
and $\dot{\tilde{\beta}}_j^n(s)=T^{-1}2^n\beta_j(t_1)$ for $s\in [t_1,t_{2}[$.
We also let
\begin{equation}\label{W-appr}
\dot{\widetilde{W}}^n(s)= \sum_{1\leq j\leq n}\dot{\tilde{\beta}}_j^n(s)q_j^{1/2}e_j
\end{equation}
denote the  corresponding finite-dimensional adapted approximation of $\dot{W}$.
\begin{lemma}\label{le:grwth-noise}
There exists an absolute constant  $\alpha_0>0$ such that for every
$\alpha > \a_0/\sqrt{T}$ and $t\in [0,T]$ we have as $n\to \infty$
\[ %begin{equation}
%\label{beta-ge-al}
\lim_{n\to\infty} \PP \Bigg(  \sup_{1\leq j\le n}\sup_{s\le t}\left|\dot{\tilde{\beta}}_j^n(s)\right|
 \ge \alpha n^{1/2} 2^{n/2}\Bigg) =  %\to 0 \, , \quad
%\end{equation}
%and
%\begin{equation}\label{W-ge-al}
\lim_{n\to\infty} \PP \Bigg(  \sup_{s\le t}\Big|\dot{\widetilde{W}}^n(s)\Big|_{H_0} \ge
\alpha n 2^{\frac{n}{2}} \Bigg) = 0.
\] %end{equation}
\end{lemma}
\begin{proof}
One can see that
\[
\tilde{\Om}_n(t) =\Big\{  \sup_{1\leq j\le n}\sup_{s\le t}\left|\dot{\tilde{\beta}}_j^n(s)\right|
 \ge
\alpha n^{1/2} 2^{n/2}\Big\}\subset\bigcup_{1\leq j\leq n}\; \bigcup_{0\leq k<2^n}
\left\{|\gamma_j^k|\ge \alpha T^{1/2} n^{1/2}\right\},
\]
where $\gamma_j^k =T^{-1/2} 2^{n/2}[\beta_j(t_{k+1})-\beta_j(t_{k})]$
are independent standard normal Gaussian random  variables.
Therefore, if
$\alpha>\sqrt{T^{-1}2\ln 2}$ we have
\begin{eqnarray*}
 \PP  \big(\tilde{\Om}_n(t)\big)
 &  \leq &  n2^n  P( \vert \gamma_1^0 \vert \geq \alpha \sqrt{T n})
 =\frac{n2^{n+1}}{\sqrt{2\pi}}\int_{\alpha \sqrt{Tn}}^\infty
e^{-z^2/2}dz
\\
& \leq &
\frac{n^{1/2}2^{n+1}}{\alpha \sqrt{2\pi T}}\int_{\alpha \sqrt{Tn}}^\infty
ze^{-z^2/2}dz =
\frac{2 n^{1/2}}{\alpha \sqrt{2\pi T}}\exp\Big[n
\Big(-{\alpha^2T}/{2}+\ln{2}\Big)\Big].
\end{eqnarray*}
This proves the first convergence result. %\eqref{beta-ge-al}.  Relation \eqref{W-ge-al} follows
The second one follows immediately from the estimate
\[
 \sup_{s\le t}\Big|\dot{\widetilde{W}}^n(s)\Big|^2_{H_0}
= \sup_{s\le t}\sum_{1\leq j\leq n}\left|\dot{\tilde{\beta}}_j^n(s)\right|^2
\le n  \sup_{1\leq j\leq n} \sup_{s\le t}\left|\dot{\tilde{\beta}}_j^n(s)\right|^2.
\]
\end{proof}
In the sequel, we will localize the processes using the following set:
\begin{equation}\label{om-n-t}
\Om_n(t)=\Bigg\{ \sup_{j\le n}\sup_{s\le t}
\Big|\dot{\tilde{\beta}}_j^n(s)\Big|
 \le \alpha n^{1/2} 2^{n/2}\Bigg\}\cap
\left\{ \sup_{s\le t}\Big|\dot{\widetilde{W}}^n(s)\Big|_{H_0}
 \le
\alpha n 2^{n/2}\right\}.
\end{equation}
It is clear that $\Om_n(t)\subset \Om_n(s)$ for $t>s$ and $\Om_n(t)\in\cF_t$.
Furthermore,  Lemma~\ref{le:grwth-noise} implies that  $\PP\left( \Om_n(T)^c\right) \to 0$
as $n\to\infty$.
\par
For any $n\geq 1$, we introduce the following localized processes $\dot {\tilde{\beta}}_j^n$ and $\dot{\tilde{W}}^n$:
\begin{equation}\label{bw-cut}
    \dot{\beta}_j^n(t)=\dot{\tilde{\beta}}_j^n(t)1_{\Om_n(t)},\;j\le n, \quad
  \dot{W}^n(t)=\dot{\tW}^n(t)1_{\Om_n(t)}.
\end{equation}
For all integers $n$ and  $j=1, \cdots, n$,
 $(\dot{\beta}_j^n(t), 0\leq t\leq T)$ (resp. $(\dot{W}^n(t), 0\leq t\leq T)$)  are $(\cF_t)$-adapted
$\R$ (resp.\ $H_0$) valued processes.

\subsection{Diffusion coefficients}
We need below two diffusion coefficients $\s$ and $\ts$ which
map $H$ into $L_Q(H_0, H)$. They are
assumed to satisfy the following growth and Lipschitz conditions:
\medskip\par
\noindent \textbf{Condition (S):} {\it
The maps   $\s,\ts $ belong to $ {\mathcal C}\big(H; L_Q(H_0, H)\big)$ and satisfy:
\begin{enumerate}
\item
 There exist non-negative  constants $K_i$ and $L$
such that for every  $u,v\in H$:
\begin{eqnarray}\label{s-bnd}
|\s(u)|^2_{L_Q}+|\ts(u)|^2_{L_Q} & \leq &K_0+ K_1 |u|^2, \\
\label{s-lip}
|\s(u)-\s(v)|^2_{L_Q}+|\ts(u)-\ts(v)|^2_{L_Q}
&\leq& L |u-v|^2 .
\end{eqnarray}
\item Moreover, for every $N>0$,
\begin{equation}\label{sn-conv}
\lim_{n\to\infty}\sup_{|u|\le N} |\ts(u)-\ts(u)\circ\Pi_n|_{L_Q}
=0,
\end{equation}
where $\Pi_n : H_0\to H_0$ denote the projector defined by
$\Pi_n u =\sum_{k=1}^n \big( u \, ,\, e_k\big) \, e_k$, where
 $\{e_k, k\geq 1\}$ is the orthonormal basis of $H$  made by
 eigen-elements of the covariance operator
 $Q$ and used in  \eqref{W-n}.
\end{enumerate}
{\bf Condition (DS):}
For every integer $j\geq 1$ let  $\s_j,\ts_j\, :\; H \mapsto H$
be defined  by
\begin{equation}\label{sj-def}
\s_j(u) =q_j^{1/2}\s(u)e_j\, , \quad \ts_j(u) =q_j^{1/2}\ts(u)e_j\, ,\quad \forall
u\in H.
\end{equation}
We  assume that the maps   $\ts_j$
are twice Fr\'echet differentiable and satisfy
\begin{enumerate}
\item For every integer $N\geq 1$ there exist positive constants $C_i(N)$, $i=1,2,3$ such that:
\begin{align}\label{Ds-bnd}
& C_1(N):=\sup_{j}\sup_{|u|\le N}
|D\ts_j(u)|_{L(H,H)}  <\infty, \\   %\quad\mbox{for every}\quad N=1,2,\ldots,
%\end{equation}
%\begin{equation}
\label{D2s-bnd}
& C_2(N):=\sup_{j}\sup_{|u|\le N}
|D^2\ts_j(u)|_{L(H\times H,H)}  <\infty , \\
%\quad\mbox{for every}\quad N=1,2,\ldots,
%\end{equation}
%and
%\begin{equation}
\label{Ds*-bnd}
&\sup_{j}\sup_{|u|\le N}
\|[D\ts_j(u)]^*v\|  \leq C_3(N) \|v\|\quad\mbox{for every}\quad v\in V. %,\; N=1,2,\ldots,
\end{align}
\item For every integer $n\geq 1$,  let  %$\s_j$ and $\ts_j$ given by \eqref{sj-def},
 the functions $\varrho_n$,  $\tro_n : H\mapsto H $ be defined by
\begin{equation}\label{rn-def}
\varrho_n(u) = \sum_{1\leq j\leq n} D\tilde{\s}_j(u)\, \s_j(u)\, , \quad
\tro_n(u) = \sum_{1\leq j\leq n} D\ts_j(u)\ts_j(u), \quad \forall  u \in H,
\end{equation}
where $\s_j$ and $\ts_j$ are  given by \eqref{sj-def}.
%Suppose that $\varrho_n, \tro_n$  satisfy the following  conditions:
For every integer $N\geq 1$ there exist positive
 constants   $\bar{K}_N, \bar{C}_N$
such that:
\begin{eqnarray}\label{rn-1}
\sup_{|u|\le N}\sup_{n}\left\{ |\varrho_n(u)| +|\tro_n(u)|\right\}
&\le& {\bar K}_N, \\
%\end{equation}
%\begin{equation}
\label{rn-2}
\sup_{|u|, |v|\le N}\sup_{n}\left\{ |\varrho_n(u)-\ro_n(v)| +|\tro_n(u)-\tro_n(v)|\right\}
&\le &{\bar C}_N |u-v|.
\end{eqnarray}
\item Furthermore, there  exist  mappings $\ro$,   $\tro : H\mapsto H $  such that every integer $N\geq 1$
\begin{equation}\label{rn-3}
\lim_{n\to\infty}\sup_{|u|\le N}\left\{ |\varrho_n(u)-\ro(u)|
+|\tro_n(u)-\tro(u)|\right\} =0 .
\end{equation}
\end{enumerate}
}

\begin{remark}\label{re:as-s}
 {\rm
As  a simple (non-trivial) example of diffusion coefficient $\s$ and $\ts$
satisfying Conditions  {\bf (S)}  and {\bf (DS)}, we can consider the case when
  $\ts(u)$ is proportional to $\s(u)$, i.e. $\ts(u)=c_0\s(u)$
  for some constant $c_0$ and $\s(u)$ is an  affine function of $u$ of the form:
\[
\s(u)f=\sum_{j \geq 1} f_j \s_j(u)\quad\mbox{for}\quad
f=\sum_{j\geq 1} f_j \sqrt{q_j}e_j\in H_0,
\]
where $\s_j(u)=g_j+S_ju$, $j=1,2,\ldots$. Here  $g_j\in H$  satisfy
$\sum_{j\geq 1} |g_j|^2<\infty$  and $S_j: H\mapsto H$
are linear operators  such that $S_j^*: V\mapsto V$ and
$ \sum_{j\geq 1} |S_j|_{L(H,H)}^2 +
\sup_{j\geq 1} |S^*_j|_{L(V,V)}^2<\infty.
$
For instance, in the case $H=L_2(D)$ and $V=H^1(D)$, where $D$ is a bounded domain in
$\RR^d$, our framework includes diffusion terms of the form
\[
\s(u)dW(t)=\sum_{1\leq j\leq N}( g_j(x)+\phi_j(x) u(x))d\beta_j(t),
\]
where $g_j\in L_2(D)$ and $\phi_j\in C^1(\bar{D})$, $j=1,2,\ldots,N$
are arbitrary functions. In this situation,  another possibility to satisfy
Conditions {\bf (S)} and {\bf (DS)} is
\[
\s(u)dW(t)=\sum_{1\leq j\leq N} s_j([\cR_ju](x))d\beta_j(t),
\]
where $s_j:\RR\mapsto\RR$ are ${\mathcal C}^2$-functions such that
$s_j'$ and $s_j''$ are bounded,  and $[\cR_ju](x)=\int_Dr_j(x,y)u(y)dy$
with sufficiently smooth kernels $r_j(x,y)$.
}
\end{remark}
\medskip\par
In order to define the sequence of processes $u^n$ converging to $u$ in the Wong-Zakai approximation, we
need a control term, that is a coefficient $G$ of the process acting on an element of the RKHS of $W$. We impose that
$G$ and $R$ satisfy the following:\\
\noindent \textbf{Condition (GR):} {\it
Let   $G\, :\, H\mapsto L (H_0, H)$  and  $R\, :\, H\mapsto  H$ satisfy the following growth and Lipschitz conditions:
\begin{eqnarray}\label{G-bnd-lip}
|G(u)|^2_{L (H_0, H)} &\leq& K_0+ K_1 |u|^2,
\quad
|G(u)-G(v)|^2_{L (H_0, H)}
\leq L |u-v|^2\, ,\\
%\end{equation}
%\begin{equation}
\label{R-bnd-lip}
|R(u)| &\leq& R_0(1+|u|),
\quad
|R(u)-R(v)| \leq R_1 |u-v|,
\end{eqnarray}
for some nonnegative constants $K_i$, $R_i$, $i=0,1$,  $L$ and for every  $u,v\in H$
(we can assume that $K_i$ and $L$ are the same constants as in \eqref{s-bnd} and
\eqref{s-lip}).
}
\subsection{Basic problem}

Let $ X: = {\mathcal C}\big([0, T]; H\big) \cap L^2\big(0, T;V\big) $
denote the Banach space endowed  with the norm defined by
\begin{equation} \label{norm}
 \|u\|_X = \Big\{\sup_{0\leq s\leq T}|u(s)|^2+ \int_0^T \|
u(s)\|^2 ds\Big\}^\frac12 .
\end{equation}
The class  $\mathcal{A}$ of  admissible random shifts  is the
   set of $H_0-$valued
$(\cF_t)-$predictable stochastic processes $h$ such that
$\int_0^T |h(s)|^2_0 ds < \infty, \; $ a.s.
For any $M>0$, let
\begin{equation} S_M=\Big\{h \in L^2(0, T; H_0): \int_0^T \!\! |h(s)|^2_0 ds \leq M\Big\},
\;  %\quad \mbox{\rm and}\quad %\]
%The set $S_M$ endowed with the following weak topology is a
%  Polish space (complete separable metric space)
%$ d_1(h, k)=\sum_{i=1}^{\infty} \frac1{2^i} \big|\int_0^T \big(h(s)-k(s),
%\tilde{e}_i(s)\big)_0 ds \big|,$
%where $
%\{\tilde{e}_i(s)\}_{i=1}^{\infty}$ is an  orthonormal basis
%for $L^2(0, T; H_0)$.
%Define
%\begin{equation}
\label{AM}
 \mathcal{A}_M=\{h\in \mathcal{A}: h(\om) \in
 S_M, \; a.s.\}.
\end{equation}

Assume that $h\in\cA_M$ and $\xi \in H$ is $\cF_0$-measurable random element such that
$\EX |\xi|^4 < \infty$.
Then under the conditions {\bf (B)}, {\bf (GR)}, \eqref{s-bnd} and \eqref{s-lip} in {\bf (S)},
 Theorem~2.4~\cite{ch-mi} implies that there exists a  unique  $(\cF_t)$-predictable
solution $u\in X$ to the stochastic problem:
\begin{eqnarray}\label{main-eq}
u(t)&  = & \xi - \int_0^t\left[Au(s)+  B(u(s))+ R(u(s))\right]ds  \\
&&  +  \int_0^t \big(\s+\ts\big)(u(s)) dW(s) +
  \int_0^t G(u(s)) h(s) ds,\;\; \mbox{ a.s. for  all $t \in [0,T]$.}
\nonumber
\end{eqnarray}
This solution  is  weak in the PDE sense and  strong in the probabilistic meaning.
Moreover, for this solution there exists a constant
 $C:=C(K_i,L,R_i,T,M)$
 such that for $h\in {\mathcal A}_M$,
\begin{equation} \label{eq3.1}
 \EX\Big( \sup_{0\leq t\leq T}
 |u(t)|^4
+ \int_0^T \|u(t)\|^2\, dt+\int_0^T \|u(t)\|_\HH^4\, dt \Big) \leq C\, \big( 1+\EX|\xi|^4\big).
\end{equation}
\begin{remark}\label{re:serrin} 
{\rm In the case of 2D Navier--Stokes  equations in a domain $D\subset\R^2$, 
we can choose  $\HH$ as the space of divergent free 2D vector fields from 
$[ L^4(D)]^2$  (see \cite{ch-mi}). Therefore,    the finiteness
of the integral  $\int_0^T \|u(t)\|_\HH^4 dt$ stated in (\ref{eq3.1})
is a Serrin's type condition. In the case of deterministic Navier--Stokes
equations this condition implies additional regularity of weak solutions
(see, e.g., \cite{Constantin}). For instance, they become strong solutions
for an  appropriate choice of the initial data. However we do not know whether 
similar regularity properties can be established for our abstract model
without additional requirements concerning the diffusion part of the equation.  
}
\end{remark}

\subsection{Approximate problem}
We also consider the  evolution equation on the time interval $[0,T]$:
\begin{eqnarray}\label{apprx-eq}
u^n(t)&  = & \xi - \int_0^t\left[Au^n(s)+  B(u^n(s))+ R(u^n(s))\right]ds
+  \int_0^t \s((u^n(s)) dW(s)
\\
&&
+  \int_0^t \ts(u^n(s)) \dot{\widetilde{W}}^n(s)ds
- \int_0^t \big(\ro+\hf\tro\big) (u^n(s))ds
 +  \int_0^t G(u^n(s)) h(s) ds,\;\; \mbox{ a.s.},
\nonumber
\end{eqnarray}
 where $\dot{\widetilde{W}}^n(t)$ is defined in \eqref{W-appr}.
Let again $h\in {\mathcal A}_M$, $\xi$ be $\cF_0$ measurable such that $\EX |\xi|^4<+\infty$.
\par
First,  since
 $(\dot{\widetilde{W}}^n_t, t\in [0,T])$ is $H_0$-valued
and $({\mathcal F}_t)$-adapted, we check that the following infinite
dimensional version of the Benes criterion holds: for some $\delta>0$
we have that $\sup_{0\leq s\leq T} \EX\big(\big(\exp(\delta
|\dot{\widetilde{W}}^n(s)|_0^2\big)\big)<+\infty$. This is a
straightforward consequence of the inequality for some standard
Gaussian random variable $Z$:
\begin{align*} \sup_{0\leq s\leq T}
\EX\big(\exp\big(\delta |\dot{\widetilde{W}}^n(s)|_0^2\big)\big)& \leq
\prod_{1\leq j\leq n} \sup_{0\leq s\leq T}
\EX\Big(\exp(\delta T^{-2}2^{2n} |\beta_j(\underline{s}_n)-\beta_j(s)|^2\big)\Big)\\
& \leq
 \Big(\EX \big(\exp(\delta T^{-1} \, 2^{n} \, |Z|^2)\big)\Big)^n<+\infty
\end{align*}
for $\delta>0$ small enough. Therefore, for any constant $\gamma$, the measure with density
$L^\gamma_T=\exp\Big(\gamma \int_0^T \dot{\widetilde{W}}^n(s) dW(s)-\frac{\gamma^2}{2}
\int_0^T |\dot{\widetilde{W}}^n(s)|^2_0\, ds\Big)$ with respect to
${\mathbb P}$ is a probability ${\mathbb Q}_\gamma<< {\mathbb P}$, such
that the process
\begin{equation} \label{Wgamma}
 \big( W_\gamma(s):= W(s)-  \gamma {\widetilde{W}}^n(s), 0\leq s\leq T\big)\; \mbox{\rm
 is a } {\mathbb Q}_\gamma \; \mbox{\rm  Brownian motion}
\end{equation}
 with values in $H$, and the same covariance operator $Q$.
Then  Theorem~3.1~\cite{ch-mi} shows that
\begin{eqnarray*} u^n(t)&  = & \xi - \int_0^t\left[Au^n(s)+  B(u^n(s))+ R(u^n(s))\right]ds
+  \int_0^t \s((u^n(s)) dW_\gamma(s)
\\
&&
- \int_0^t \big(\ro+\hf\tro\big) (u^n(s))ds
 +  \int_0^t G(u^n(s)) h(s) ds,\;\; \mbox{ a.s.},
\end{eqnarray*}
has a unique $(\cF_t)$-predictable solution $u^n\in X$ which satisfies  \eqref{eq3.1} where ${\mathbb P}$
is replaced by ${\mathbb Q}_\gamma$
Therefore, $u^n\in X$ is the unique solution to  problem \eqref{apprx-eq} when $\tilde{\sigma}= - \gamma \sigma$, and
$u^n$ also satisfies  \eqref{eq3.1} with expected values under the given probability ${\mathbb P}$,
  but the  constant $C$ in the right hand side  depends on the constants
$K_i$, $L$, $R_i$, $T$, $M$, and
in addition,  it may also depend on $n$.
\par
To keep the convergence result as general as possible,  %(which will be necessary to deduce the support characterization),
 in the sequel we only  suppose that $\tilde{\sigma}$ satisfies
Conditions {\bf (DS)} and  that problem  \eqref{apprx-eq} is well-posed in $X$.
%and that its solution $u^n$ satisfies \eqref{eq3.1} with some constant which may depend upon $n$.

\section{Main results}\label{sec:support}
In this section we first state Wong--Zakai approximation results
(see Theorem \ref{th:WZ-appr}) and then show in Theorem
\ref{th:support} how the description support can be derived from
Theorem \ref{th:WZ-appr}.
\par
More precisely,  let $u$ and $u^n$ be solutions to \eqref{main-eq}
and \eqref{apprx-eq} respectively. Our first main result  proves
that the $X$ norm of the difference $u^n-u$ converges to 0 in
probability. This Wong--Zakai type result is the key point of the
support characterization stated in Theorem \ref{th:support} below.
\begin{theorem}\label{th:WZ-appr}
Let conditions {\bf (B)}, {\bf (S)}, {\bf (DS)} and {\bf (GR)} hold,
 $h\in {\mathcal A}_M$ for some $M>0$
and $\xi$ be $\cF_0$ measurable such that $\EX |\xi|^4<+\infty$. Let $u$ be the  solution to
\eqref{main-eq} such  that:\\
  (i) $t\mapsto \|u(t)\|_\HH$ is continuous on $[0,T]$ almost surely,\\
 (ii) there exists $q>0$ such that
for any constant $C>0$ we have
\begin{equation} \label{crit-bound}
\EX \Big( \sup_{[0,\tau_C]}\| u(t)\|^q_\HH \Big)<\infty,
\end{equation}
where
 $\tau_C:=\inf\{t : \sup_{s\le t}|u(s)|^2 + \int_0^t \|u(s)\|^2 ds \geq C\}\wedge T$
 is a stopping time.
\smallskip\par
 Suppose that for every $n\geq 1$  problem \eqref{apprx-eq} is well posed and let $u^n$ denote its solution.
%and  its  solution $u^n$  satisfies
%\eqref{eq3.1} (with a constant $C$ which may depend on $n$).
  Then for every $\la>0$ we have
\begin{equation} \label{WZ-cnvrg}
\lim_{n\to\infty}\PP \Big( \sup_{t\in [0,T]}|u(t)-u^n(t)|^2 +
\int_0^T \|u(s)-u^n(s)\|^2 ds   \geq \la\Big)=0 .
\end{equation}
\end{theorem}
The convergence in  \eqref{WZ-cnvrg} allows to deduce both
inclusions characterizing the support of the distribution of  the
solution  $U$  to the following stochastic perturbation of the
evolution equation \eqref{u0}:
\begin{equation} \label{u-supp}
d U(t)  + \big[  A U(t) + B\big(U(t) \big) + R (U(t)) \big]\, dt
=  \Xi(U(t))\, dW(t), \quad U(0)=\xi\in H,
\end{equation}
where $\Xi\in {\mathcal C}\big(H; L_Q(H_0, H)\big)$ is such that
conditions  {\bf (S)} and {\bf (DS)} hold with $\ts\equiv\s\equiv \Xi$.
Thus problem \eqref{u-supp} is a special case of problem \eqref{main-eq}.
\par
Let $\phi\in L^2(0,T; H_0)$ and
$\ro_\Xi\equiv \ro $ be defined by \eqref{rn-def} and \eqref{rn-3} with $\ts=\s=\Xi$.
We also consider
the following (deterministic) nonlinear PDE
\begin{equation}  \label{vh-supp}
\partial_t v_\phi(t)  +   A v_\phi(t) + B\big(v_\phi(t) \big) +  R (v_\phi(t)) +
\frac{1}{2} \ro_\Xi (v_\phi(t))
 =  \Xi(v_\phi(t))\phi(t), \quad v_\phi(0)=\xi\in H.
\end{equation}
If $B(u)$ satisfies condition {\bf (B)} and $R:H\mapsto H$
possesses property \eqref{R-bnd-lip} we can use
Theorem 3.1 in \cite{ch-mi} to obtain the existence (and uniqueness)
of the  solution $v_\phi$ to \eqref{vh-supp} in the space
$X = {\mathcal C}\big([0, T]; H\big) \cap L^2\big(0, T;V\big)$.
Let
\[
\cL= \left\{ v_\phi\, : \; \phi\in L^2(0,T; H_0)\right\}\subset X.
\]
Our second main result  is the following  consequence of the
approximation Theorem~\ref{th:WZ-appr}.
\begin{theorem}\label{th:support}
Let conditions {\bf (B)} and {\bf (S)}, {\bf (DS)} with
$\ts\equiv\s\equiv \Xi$ be in force. Assume that $R:H\mapsto H$
satisfies \eqref{R-bnd-lip}. Let $(U(t), t\in [0,T])$ denote the  solution to the
stochastic evolution  equation
 \eqref{u-supp} with deterministic initial data $\xi\in H$.
Suppose that
 conditions (i) and (ii) of Theorem~\ref{th:WZ-appr} hold for this solution $U$.
Then
$
{\rm supp}\, U(\cdot)= \bar{\cL},
$
where $\bar{\cL}$ is the closure of $\cL$ in $X$ and ${\rm supp}\, U(\cdot)$
denotes the support of the distribution $\PX\circ U^{-1}$, i.e.,  the support
of the Borel measure on $X$ defined by
$\mu(\cB)=\PX\left\{\om\, :\; U(\cdot)\in \cB\right\}$
for any Borel subset $\cB$ of  $X$.
\end{theorem}
\begin{remark}\label{re:appl}
 {\rm
  Both Theorems \ref{th:WZ-appr} and \ref{th:support} are
{\em conditional} in the sense that they provide   an approximation of the solution $u$
 or  the description of its
support when $u$   satisfy some additional conditions
concerning its properties in the space $\HH$  (see
the requirements (i) and (ii) in Theorem~\ref{th:WZ-appr}).
      We do not know
whether these conditions can be derived from the basic requirements
which we already have imposed on the model. 
However  they can be established under 
additional conditions concerning operators in (\ref{main-eq}).
For instance, we can assume that  the bilinear operator $B$ possesses
the property 
\begin{equation}\label{Ns-per}
(B(u,u), Au)=0~~~ \mbox{ for }~~ u\in Dom (A)
\end{equation}
(this is the case of 2D Navier-Stokes equations in a periodic domain,
 see, e.g., \cite{Constantin}) and the diffusion coefficient 
$\sigma+ \ts$ satisfies the estimate
  \[
|A^{\frac{1}{2}}(
\sigma(u)+ \ts(u))|^2_{L_Q(H_0,H)} \leq 
K(1+\|A^{1-\delta} u\|^2)~~~\mbox{ for all }~~~ u\in Dom(A)
\] 
for some $K>0$ and $\delta>0$. Under these conditions we can apply 
 It\^o's  formula to the norm $\| u(t)\|^2=| A^{1/2} u(t)|^2$
and obtain that 
\begin{align*} 
& \|u(t\wedge \tau_N)\|^2+
2\int_0^{t\wedge \tau_N}\!\! |Au(s)|^2 ds = \| \xi\|^2
+ 2 \int_0^{t\wedge \tau_N}\!\!
 \big( (\sigma+\ts)(u(s))  dW(s) , A u(s)\big)
 \nonumber \\
& \; -2\int_0^{t\wedge \tau_N}\!\! \big( R  (u(s))+ G(u(s)) h(s), 
Au(s)\big) \, ds
 +  \int_0^{t\wedge \tau_N} \!\! |A^{1/2}(\sigma+\ts)(u(s))|_{L_Q}^2\, ds
\end{align*}  
for an appropriate sequence of stopping times $\{ \tau_N\}$
(see \cite{ch-mi} for similar calculations).  In the standard way 
(see \cite{PZ92}) this implies that 
$u(t)\in C(0,T; V)$ a.s. and 
  \[
\EE \left\{\sup_{t\in [0,T]}\|u(t)\|^2+
2\int_0^{T}\!\! |Au(s)|^2 ds \right\}\le C(1+ \EE\| \xi\|^2)
\]   
for some constant $C>0$.
Since $V\subset \HH$, this implies 
the requirements (i) and (ii) in Theorem~\ref{th:WZ-appr}.
Unfortunately the assumption in (\ref{Ns-per}) is rather restrictive.
 To our best knowledge, in our 2D hydrodynamical framework 
it is only valid for  2D  Navier-Stokes equations   with the periodic boundary conditions.
\par 
Other simple examples where we can apply Theorem~\ref{th:WZ-appr} and 
\ref{th:support} are the shell models of turbulence.
We can consider either the GOY model
  or the Sabra model,    or else
the so-called dyadic model. Indeed,  
(see \cite[Sect.2.1.6]{ch-mi}) in all these models
we have that
$
\left|\langle B(u,v), w\rangle\right|\le C |u||A^{1/2} v| |w|$,
$\forall  u,w\in H, \forall v\in Dom(A^{1/2})$.
Thus condition ({\bf B}) holds with $\HH= Dom(A^{s})$ for any choice
of $s\in [0,1/4]$. In particular we can choose $\HH=H$.
In this case conditions (i) and (ii) in Theorems~\ref{th:WZ-appr} and \ref{th:support}
trivially hold.
\par 
In
the Appendix (see Section \ref{appendix})  we  show that
conditions (i) and (ii) in the statement of Theorem~\ref{th:WZ-appr}
can be also established under another set of hypotheses
which
hold for several important cases of hydrodynamical models
 such as (non-periodic) 
 2D Navier-Stokes equations and 2D MHD equations. }
\end{remark}
\par

\noindent{\bf Proof of Theorem \ref{th:support}.}
The argument  is similar to that introduced in \cite{M-SS2}.
\par
In the definition of the evolution equation \eqref{apprx-eq} let $\ts =\Xi$,
$\s=0$, $G=0$. Then if $v_h$ is the solution
to \eqref{vh-supp}, where   $\phi=\dot {\widetilde{W}}_n$   is the (random) element of $L^2(0,T;H_0)$
defined by \eqref{W-appr}, then we have that
 $u^n=v_{\dot{\widetilde{W}}_n}$ for $u^n$ defined by \eqref{apprx-eq}.
Note that is this case, well posedeness of \eqref{apprx-eq}  in $X$
is easy to prove on $\omega$ by $\omega$.
Under this choice of  $\ts$, $\s$ and $G$ for the solution  $u$ to \eqref{main-eq},
%$U$ to \eqref{u-supp}
we obviously have that $U(t)=u(t)$, where $U$ solves \eqref{u-supp}.
Therefore the Wong-Zakai approximation  stated in Theorem~\ref{th:WZ-appr}
implies that
\[
\lim_n \PP(\|v_{\dot{\widetilde{W}}_n} - U\|_X \geq \lambda)=0
\quad\mbox{for any $\lambda >0$},
\]
where $\|\cdot \|_X$ is the norm defined by \eqref{norm}. Therefore
$ \mbox{\rm Support} (\PP \circ U^{-1}) \subset \overline{\mathcal L}$.
\smallskip\par
 Conversely, fix $h\in L^2(0, T;H_0)$, let $n\geq 1$ be an integer, $\ts =-\s$ and $G=\s :=  \Xi$.
 Let $T_n^h:\Omega \to \Omega$ be defined by
\begin{equation}\label{Th-tr}
T^h_n( \omega) = W(\omega) - \widetilde{W}^n(\omega)+\int_0^. h(s)ds.
\end{equation}
 Then for every fixed integer $n\geq 1$, by Girsanov's theorem
 there exists a probability ${\mathbb Q}_n^h % :=\PP \circ (T_n^h)^{-1}
 << \PP$ such that $T_n^h$ is a ${\mathbb Q}_n^h$-Brownian motion with values in $H$
and the same covariance operator $Q$.
 Indeed, the proof is easily decomposed in two steps, using Theorem 10.14 and Proposition 10.17 in \cite{PZ92}.
\par
First,  since
 $(\dot{\widetilde{W}}^n_t, t\in [0,T])$ is $H_0$-valued
and $({\mathcal F}_t)$-adapted, the argument used at the end of section \ref{s2}
 with $\gamma=1$ proves that the measure with density
$L^1_T=\exp\Big(\int_0^T \dot{\widetilde{W}}^n(s) dW(s)-\frac{1}{2}
\int_0^T |\dot{\widetilde{W}}^n(s)|^2_0\, ds\Big)$ with respect to
${\mathbb P}$ is a probability ${\mathbb Q}_1<< {\mathbb P}$, such
that the process
 $\big( W_1(s):= W(s)-{\widetilde{W}}^n(s), 0\leq s\leq t\big)$
 is a ${\mathbb Q}_1$ Brownian motion with values in $H$, and the same covariance operator $Q$.
\par

Then using once more these two results, since $h\in L^2([0,T],H_0)$,
the measure   with  density $L^2_T=\exp\Big(- \int_0^T h(s) dW_1(s)
- \frac{1}{2} \int_0^T |h(s)|^2 ds\Big)$ with respect to ${\mathbb
Q}_1$ is a probability ${\mathbb Q}_2<<{\mathbb P}$, such that the
process
\[
W_2(t)=W_1(t)+\int_0^t h(s)ds=W(t)-\widetilde{ W}^n(t) +\int_0^t h(s)ds
\]
is a $H$-valued Brownian motion under ${\mathbb Q}_2$, with
covariance operator $Q$. Clearly ${\mathbb Q}_n^h= {\mathbb Q}_2$.
\par
Let  $U$ denote the solution to
\eqref{u-supp};  then, since %the transformation
 $T_n^h$ can be seen as a transformation of the standard Wiener space
 with Brownian motion $W(t)$, we deduce that
$U(\cdot)(T_n^h(\omega))= u^n(\cdot)(\omega)$ in distribution on $[0,T]$.
 Thus Theorem \ref{th:WZ-appr}
implies that for every $\varepsilon  >0$,
\[ \limsup_n \PP(\{ \omega \, :\, \|U(T^n_h(\omega)) - v_h\|_X < \varepsilon \})  >0.\]
Let $n_0\geq 1$ be an integer such that
\[ {\mathbb Q}_{n_0}^h(\{ \omega \, :\, \|U(\omega) - v_h\|_X < \varepsilon \})
\equiv\PP(\{ \omega \, :\, \|U(T^{n_0}_h(\omega)) - v_h\|_X < \varepsilon \})  >0.
\]
Since ${\mathbb Q}_{n_0}^h << \PP$,    this
 implies   $  \PP(\{ \omega \, :\, \|U(\omega) - v_h\|_X < \varepsilon ) >0$
which yields:
\[
\overline{\mathcal L}\subset  \mbox{\rm Support} (\PP \circ U^{-1}).
\]
This completes the proof of Theorem \ref{th:support}.
\section{Preliminary step in the proof of Theorem
\ref{th:WZ-appr}}\label{sect:prelim}
 Let $M>0$ be such that $h\in {\mathcal A}_M$.
Without loss of generality we may and do assume in the sequel that  $0<\la\leq 1$.
Fix $N\geq 1$, $m\geq 1$ and $\lambda \in ]0,1]$. Let us
 introduce the following stopping times which will enable us to bound several norms for $u$ and
$u^n$:
\begin{align*} % \label{tau-1}
\tau^{(1)}_N&= \inf\Big\{ t>0 : \sup_{s\in [0,t]}|u(s)|^2 +
\int_0^t \|u(s)\|^2 ds
\geq  N \Big\}\wedge T, \\
%\end{equation}
%\begin{equation}
% \label{tau-2}
\tau^{(2)}_n&= \inf\Big\{ t>0 : \sup_{s\in [0,t]}|u(s)-u^n(s)|^2 +
\int_0^t \|u(s)-u^n(s)\|^2 ds
\geq  \la\Big\}\wedge T, \\
%\end{equation}
%\begin{equation}
% \label{tau-3}
\tau^{(3)}_n&= \inf\Big\{ t>0 : \Big[\sup_{j\le n}\sup_{s\in [0, t]}
\Big|\dot{\tilde{\beta}}_j^n(s)\Big|\Big]\vee\Big[
n^{-\frac{1}{2}} \sup_{s\in [0, t]}\Big|\dot{\widetilde{W}}^n(s)\Big|_{H_0}\Big]
\geq  \alpha n^{1/2} 2^{n/2}
\Big\}\wedge T,
\end{align*}
and
\[ %begin{equation} \label{tau-4}
\tau^{(4)}_m= \inf\Big\{ t>0 : \sup_{s\in [0,t]}\|u(s)\|_\HH \geq  m\Big\}\wedge T.
\] %end{equation}
In the sequel, the constants $N$ and $m$ will be chosen to make sure that, except on small
sets, $\tau_N^{(1)}$ and $\tau_m^{(4)}$ are equal to  $T$; once this is done,  only the dependence in $n$
will be relevant. Thus once $N$ and $m$ have been chosen in terms of the limit process $u$, we let
\begin{equation}\label{tau-f}
\tau_n =\tau^{(1)}_N \wedge \tau^{(2)}_n \wedge\tau^{(3)}_n\wedge\tau^{(4)}_m.
\end{equation}
One can see from the definition of $\tau_N^{(1)}$ and $\tau_n^{(2)}$ that
 %\eqref{tau-1} and \eqref{tau-2} that
\begin{equation} \label{bnd-n1}
 \sup_{s\in [0,\tau_n]}\left(|u(s)|^2\vee |u^n(s)|^2\right) +
\int_0^{\tau_n} \left(\|u(s)\|^2 \vee \|u^n(s)\|^2 \right)ds
\le 2(N+1);
\end{equation}
the definition of $\tau_n^{(3)}$ yields % equation  \eqref{tau-3} yields
\begin{equation} \label{bnd-bW}
\sup_{s\le \tau_n}\Big( \Big[\sup_{j\le n}
\big|\dot{\tilde{\beta}}_j^n(s)\big|\Big]\vee\Big[
%\frac{1}{\sqrt{n}}
n^{-\frac{1}{2}} \sup_{s\le t}\big|\dot{\widetilde{W}}^n(s)\big|_{H_0}\Big]
\Big) \le  \alpha n^{1/2} 2^{n/2}.
\end{equation}
Furthermore, the definition of $\tau_m^{(4)} $ implies %equation  \eqref{tau-4} implies
\begin{equation} \label{bnd-HH}
 \sup_{s\in [0,\tau_n]}\|u(s)\|_\HH \le m.
\end{equation}
We use the following  obvious properties; their standard proof is omitted.
\begin{lemma}\label{le:stop-t}
Let   $\Psi(t)\equiv\Psi(\om,t)$ be  a random, a.s. continuous,
 nondecreasing process on
the interval $[0,T]$.
Let $\tau_\la=\inf\{ t>0\,:\; \Psi(t)\geq  \lambda\}\wedge T$.
Then
\[ %begin{equation}\label{Psi-t-l}
\PX(  \Psi(T)\geq  \lambda) =\PX( \Psi(\tau_\la)\geq  \lambda).
\] %end{equation}
Let  $\tau_{*}$ be a stopping time such that
$0\le \tau_{*}\le T$ and
$\PX \big( \tau_{*}< T\big)\le \e$. % and $\PX\{\om\,:\; \tau_{**}< T\}\le\e$,
Then
\[ %begin{equation}\label{Psi-t-2}
\PX \big( \Psi(T)\geq \lambda\big)\le\PX \big( \Psi(\tau_{\la}\wedge \tau_{*})\geq \lambda\big)+ \e .
\] %end{equation}
\end{lemma}
\par
Apply Lemma~\ref{le:stop-t} with
$\tau_{*}=\tau^{(1)}_N \wedge \tau^{(3)}_n \wedge\tau^{(4)}_m$ and
\[
\Psi(t)=\sup_{s\in [0,t]}|u(s)-u^n(s)|^2 +
\int_0^t \|u(s)-u^n(s)\|^2 ds.
\]
 Since a.s. $u\in {\mathcal C}([0,T],H)$  and $\int_0^T \|u(s)\|^2\, ds <+\infty$, the map
 $\Psi$ is a.s. continuous and
\begin{eqnarray*}
\lefteqn{
\{ \tau_{*}<T\} \subset \{\tau^{(1)}_N <T\}\cup  \{\tau^{(3)}_n <T\} \cup
\{ \tau^{(4)}_m< \tau^{(1)}_N \}
}
\\
&&\subset \Bigg\{ \sup_{s\in [0,\tau^{(1)}_N]}|u(s)|^2 +
\int_0^{\tau^{(1)}_N} \|u(s)\|^2 ds\ge N\Bigg\}\cup
\Bigg\{ \sup_{s\in [0,\tau^{(1)}_N]}\|u(s)\|_{\mathcal H} \ge m \Bigg\}
\cup \Omega_n(T)^c,
\end{eqnarray*}
where $\Omega_n(t)$ is given by \eqref{om-n-t}.
Therefore,  by Chebyshov's inequality,  from \eqref{eq3.1} and \eqref{crit-bound}
we deduce that
\[
\PX\big(  \tau_{*}<T\big)\le {C_1}{N}^{-1} +{C_2(N)}{m^{-q}} + \PX(\Omega_n(T)^c)\;.
\]
Hence, given $\epsilon >0$, one may choose $N$ and  then  $m$ large enough
 to have ${C_1}{N}^{-1}+{C_2(N)}m^{-q} <\frac{\epsilon}{2}$. Using Lemma \ref{le:grwth-noise}
 we deduce that there exists $n_0\geq 1$  such that for all integers
$n\geq n_0$, $\PX(\Omega_n(T)^c)<\frac{\e}{2}$.  Thus Lemma \ref{le:stop-t}
 shows that in order to prove
 \eqref{WZ-cnvrg} in Theorem~\ref{th:WZ-appr}, we  only need to prove
the following: Fix $N,m>0$; for every $\la >0$,
\begin{equation} \label{WZ-cnvrg-2}
\lim_{n\to\infty}\PP \Big(  \sup_{t\in [0,\tau_n]}|u(t)-u^n(t)|^2 +
\int_0^{\tau_n} \|u(s)-u^n(s)\|^2 ds
\geq \la\Big)=0,
\end{equation}
where  $\tau_n$ is defined by \eqref{tau-f}.  To check this convergence,  it is sufficient to prove
\begin{equation} \label{WZ-cnvrg-3}
\lim_{n\to\infty} \Big[\EX \Big(\sup_{t\in [0,\tau_n]}|u(t)-u^n(t)|^2\Big) +
\EX \int_0^{\tau_n} \|u(s)-u^n(s)\|^2 ds\Big] =0.
\end{equation}
The proof of this last convergence result is given in Section \ref{pr-conv}. It relies on some
precise  control of
times increments which is proven in the next section.

\section{Time increments}\label{time-in}
%For every $s\in [0,T]$, let
%\begin{eqnarray}\label{phi-psi}
%& \phi^{(1)}_n(s)= \left( s-T2^{1-n}\right)^+,  &  \psi^{(1)}_n(s)=s; \nonumber \\
%&\phi^{(2)}_n(s)= \left( s-T2^{-n}\right)^+,  &  \psi^{(2)}_n(s)=s; \nonumber \\
%&\phi^{(3)}_n(s)= s,  &   \psi^{(3)}_n(s) =\bar{s}_n.
%\end{eqnarray}
%\begin{lemma}\label{le:increm}
Let  $h\in\cA_M$,  $\xi$ be an $\cF_0$-measurable
$H$-valued random variable such that $\EX |\xi|^4 < \infty$ and let $u$ be the solution
to \eqref{main-eq}. For any integer $N\geq 1$ and $t\in[0,T]$  set
\begin{equation}\label{g-N-set}
\widetilde{G}_N(t)=\Big\{  \sup_{0\leq s\leq t}  |u(s)|\le N\Big\} .
%  ,\quad N=1,2,\ldots
\end{equation}
The following lemma refines the estimates proved in   \cite{ch-mi},
 Lemma 4.3 and the ideas are similar;
see also \cite{DM}, Lemma 4.2.
\begin{prop}\label{pr:increm}
 Let $\phi_n,\psi_n\, :\; [0,T]\mapsto [0,T]$ be non-decreasing
 piecewise continuous functions
such that
\begin{eqnarray}\label{phi-psi-1}
0\vee\left( s-k_0T2^{-n}\right) \le \phi_n(s)\le  \psi_n(s) \le
\left( s+k_1T2^{-n}\right)\wedge T
\end{eqnarray}
for some  integers $k_0,k_1\ge0$. Assume that $h\in\cA_M$ and $\xi$
 is a $\cF_0$-measurable, $H$-valued  random variable  such that $\EX |\xi|^4 <
\infty$. Let $\widetilde{G}_N(t)$ be given by \eqref{g-N-set} and
$u$ be the  solution to \eqref{main-eq}. There exists a constant
$C(N,M,T)$ such that
\begin{equation}\label{eq-increm-1}
I_n=\EX \int_0^T 1_{\widetilde{G}_N(\psi_n(s))} \left|
u(\psi_n(s))-u(\phi_n(s))\right|^2 ds\le  C(N,M,T) 2^{-3n/4}
\end{equation}
for every  $n=1,2,\ldots$
\end{prop}
\begin{proof} %The proof relies on the same idea as   Lemma 4.3 in \cite{ch-mi}.
%\par
We at first consider the case  $\phi_n(s)=0\vee\left( s-k_0T2^{-n}\right)$
for some $k_0\ge 0$; then
\[
I_n= \EX \int_0^{t_{k_0}} 1_{\widetilde{G}_N(\psi_n(s))} \left|
u(\psi_n(s))-\xi\right|^2 ds + I'_n
\]
where $t_{k_0}=k_0 T2^{-n}$ and
\[ %begin{equation}\label{In-prime}
I'_n=\EX \int_{t_{k_0}}^T 1_{\widetilde{G}_N(\psi_n(s))} \left|
u(\psi_n(s))-u(\phi_n(s))\right|^2 ds.
\]  %end{equation}
Therefore, using the definition \eqref{g-N-set} one can see that
\begin{equation}\label{In-prime1}
I_n\le C_{N,T} 2^{-n}+ I'_n.
\end{equation}
Furthermore,  It\^o's formula yields
\[
 |u(\psi_n(s))-u(\phi_n(s))|^2 =2\int_{\phi_n(s)}^{\psi_n(s)} (u(r)-u(\phi_n(s)), d
 u(r))+
\int_{\phi_n(s)}^{\psi_n(s)}|(\s+ \ts) (u(r))|^2_{L_Q}d r,
\]
so that $I'_n=\sum_{1\leq i\leq 6} I_{n,i}$, where
\begin{eqnarray*}
I_{n,1}&=&2  \EX\Big(  \int_{t_{k_0}}^T \!\!
ds1_{\widetilde{G}_N(\psi_n(s))} \int_{\phi_n(s)}^{\psi_n(s)}\!
\big( u(r)-u(\phi_n(s)), (\sigma  + \ts)(u(r)) dW(r)  \big)\Big),
\\
I_{n,2}&=&  \EX \Big( \int_{t_{k_0}}^T  \!\!ds1_{\widetilde{G}_N(\psi_n(s))}
\int_{\phi_n(s)}^{\psi_n(s)}\! \!\!
|(\sigma + \ts) (u(r))|_{L_Q}^2 \, dr\Big) , \\
I_{n,3}&=&2 \,  \EX \Big(  \int_{t_{k_0}}^T \!\! ds 1_{\widetilde{G}_N(\psi_n(s))}
\int_{\phi_n(s)}^{\psi_n(s)} \!\! \big(
G(u(r)) \, h(r)\,  , \, u(r)-u(\phi_n(s)) \big)\, dr\Big), \\
I_{n,4}&=&- 2 \,  \EX \Big(  \int_{t_{k_0}}^T  \!\! ds  1_{\widetilde{G}_N(\psi_n(s))}
\int_{\phi_n(s)}^{\psi_n(s)} \!\!
 \big\langle A \, u(r)\, , \, u(r)-u(\phi_n(s))\big\rangle  \, dr\Big) ,  \\
I_{n,5}&=&-2 \, \EX \Big( \int_{t_{k_0}}^T  \!\! ds  1_{\widetilde{G}_N(\psi_n(s))}
\int_{\phi_n(s)}^{\psi_n(s)}\!\!
  \big\langle B( u(r))\, , \,   u(r)-u(\phi_n(s))\big\rangle \, dr\Big) ,  \\
I_{n,6}&=&- 2  \, \EX \Big(  \int_{t_{k_0}}^T \!\!  ds  1_{\widetilde{G}_N(\psi_n(s))}
\int_{\phi_n(s)}^{\psi_n(s)}\!\!
  \big( R (u(r))\, , \,  u(r)-u(\phi_n(s))\big)\, dr\Big) .
\end{eqnarray*}
Clearly  $\widetilde{G}_N(\psi_n(s))\subset \widetilde{G}_N(r)$
for $r\le \psi_n(s)$. This
means that  $|u(r)|\vee |u(\phi_n(s))|\le N$
in the  above integrals.
 We use this observation in the considerations
below.
\par
The Burkholder-Davis-Gundy inequality and \eqref{s-bnd} yield
\begin{eqnarray}\label{i1-bdg}
|I_{n,1}|&\leq &
 6 \int_{t_{k_0}}^T ds \; \EX \Big(
 \int_{\phi_n(s)}^{\psi_n(s)} |(\s + \ts) (u(r))|_{L_Q}^2
 1_{\widetilde{G}_N(r)}\, | u(r)- u(\phi_n(s))|^2 \;  dr \Big)^{\frac{1}{2}} \nonumber
 \\
&\leq &
 6\sqrt{2(K_0+K_1 N^2)} \int_{t_{k_0}}^T\!  ds \, \EX \Big(
 \int_{\phi_n(s)}^{\psi_n(s)}
 1_{\widetilde{G}_N(r)}\, | u(r)- u(\phi_n(s))|^2 \;  dr \Big)^{\frac{1}{2}} .
\end{eqnarray}
This implies  that
\begin{equation} \label{In1}
|I_{n,1}|
\leq  C_N  \int_0^T   |\psi_n(s)- \phi_n(s)|^{1/2} \, ds
 \le C_{N}T \sqrt{k_0+k_1} 2^{-\frac{n}{2}}.
\end{equation}
In a similar way using \eqref{s-bnd} again we deduce  that
\begin{equation} \label{In2}
|I_{n,2}|
\leq  C_N  \int_0^T   |\psi_n(s)- \phi_n(s)| \, ds
 \le C_{N}T(k_0+k_1) 2^{-n}.
\end{equation}
The growth condition  \eqref{G-bnd-lip} yields
\[ |I_{n,3}|
\leq  C_N  \int_0^T ds \int_{\phi_n(s)}^{\psi_n(s)}|h(r)|_0dr
\le  C_N  \int_0^T ds
\int_{0\vee\left( s-k_0T2^{-n}\right) }^{\left( s+k_1T2^{-n}\right)\wedge T}|h(r)|_0 dr,
\]
and Fubini's theorem  implies
\begin{equation} \label{In3}
|I_{n,3}|
\le  C_N  \int_0^T|h(r)|_0dr 2^{-n}\le C(N,T,M) 2^{-n}.
\end{equation}
 Using Schwarz's inequality we deduce that
\begin{eqnarray} \label{In4-0}
I_{n,4} &\leq &  2  \EX \Big(  \int_{t_{k_0}}^T  \!\! \! ds
 1_{\widetilde{G}_N(\psi_n(s))}
\int_{\phi_n(s)}^{\psi_n(s)}\!\!
dr  \big[ -  \|u(r)\|^2  +  \|u(r)\| \|u(\phi_n(s))\|\big]\Big).
\end{eqnarray}
The antisymmetry relation \eqref{as} and inequality \eqref{boundB1}
yield
\begin{eqnarray*}
 \left|\big\langle B( u(r)),   u(r)-u(\phi_n(s))\big\rangle\right| & = &
  \left|\big\langle B( u(r)),  u(\phi_n(s))\big\rangle\right| \\
& \le &  \frac12 \|u(r)\|^2 +C |u(r)|^2 \|u(\phi_n(s))\|^4_\HH .
\end{eqnarray*}
Therefore,
\begin{align*} %\label{In5}
|I_{n,5}| &\leq  \EX \Big(  \int_{t_{k_0}}^T  \!\! \! ds
 1_{\widetilde{G}_N(\psi_n(s))}
\int_{\phi_n(s)}^{\psi_n(s)}\!\!  \|u(r)\|^2\, dr  \Big) \\
%\nonumber \\
& {}\quad +\, 2C \EX \Big(  \int_0^T  \!\! \! ds
 1_{\widetilde{G}_N(\psi_n(s))} \|
u(\phi_n(s))\|^4_\HH\int_{\phi_n(s)}^{\psi_n(s)} \!\!\!   |u(r)|^2\, dr  \Big).
\end{align*}
Using this inequality,  \eqref{In4-0} and  \eqref{interpol},  we deduce:
\begin{align*} %\label{In4-01}
I_{n,4} &+I_{n,5} \leq   \EX  \int_{t_{k_0}}^T \! \!
 1_{\widetilde{G}_N(\psi_n(s))} \|u(\phi_n(s))\|^2\; |\psi_n(s)-\phi_n(s)| ds
  \\
& {}\qquad\qquad  +\, C_N \EX  \int_{t_{k_0}}^T \! \!
 1_{\widetilde{G}_N(\psi_n(s))} \|
u(\phi_n(s))\|^4_\HH \; |\psi_n(s)-\phi_n(s)| ds  %\nonumber
\\
 &\leq   C(N,T)  2^{-n} \EX  \int_{t_{k_0}}^T \! \!
 1_{\widetilde{G}_N(\psi_n(s))}  \|u(\phi_n(s))\|^2 ds
 =  C(N,T)  2^{-n} \EX  \int_{t_{k_0}}^T \! \!
  \|u(s -t_{k_0})\|^2 ds.  %\nonumber
\end{align*}
Hence, this last inequality and  \eqref{eq3.1}  imply
\begin{equation}\label{In4.5}
   I_{n,4} +I_{n,5}\le  C(N,T)\,  2^{-n}.
\end{equation}
A similar easier computation based on the growth condition \eqref{R-bnd-lip} on $R$ yields
\begin{equation} \label{In6}
|I_{n,6}|
\leq  4  R_0 N (1+N) \int_0^T |\psi_n(s)- \phi_n(s)|\, ds
 \le C(T,N)\, (k_0+k_1)\,  2^{-n}.
\end{equation}
Thus by \eqref{In-prime1}, \eqref{In1}--\eqref{In3}, \eqref{In4.5} and \eqref{In6},
 %and also \eqref{In4.5}, \eqref{In3}
when   $\phi_n(s)=0\vee(s-k_0T2^{-n})$ we obtain:  %from \eqref{In-prime1} that
%\[
%I_n=\EX \int_0^T 1_{\widetilde{G}_N(\psi_n(s))} \left|
%u(\psi_n(s))-u(\phi_n(s))\right|^2 ds\le  C(N,M,T) 2^{-n/2}
%\]
\begin{equation}\label{inc-prl}
I_n=\EX \int_0^T 1_{\widetilde{G}_N(\psi_n(s))} \left|
u(\psi_n(s))-u(\phi_n(s))\right|^2 ds\le  C(N,M,T) 2^{-n/2}.
\end{equation}
%under the conditions of Proposition~\ref{pr:increm}.
\par
In order to obtain \eqref{eq-increm-1},  we need to improve
the bound for $I_{n,1}$ (see \eqref{In1}). We make it using \eqref{inc-prl} and
 we again  assume at first that $\phi_n(s)=0\vee(s-k_0T2^{-n})$.
Let us  denote by  $\chi_{i,n}(s)=\underline{\left(s+(i+1)T 2^{-n}\right)}_{\, n}$
 the step function defined with the help of relations  \eqref{s-funct} and set
\begin{align*}
 \cI_n^{(i,-)}&=\int_{t_{k_0}}^{T-t_{k_1}} ds \; \EX
  \int_{s+iT 2^{-n}}^{\chi_{i,n}(s)}
 1_{\widetilde{G}_N(r)}\, | u(r)- u(\phi_n(s))|^2 \;  dr, \\
 \cI_n^{(i,+)}&=\int_{t_{k_0}}^{T-t_{k_1}} ds \; \EX
  \int_{\chi_{i,n}(s)}^{s+(i+1)T 2^{-n}}
 1_{\widetilde{G}_N(r)}\, | u(r)- u(\phi_n(s))|^2 \;  dr.
\end{align*}
The inequality  \eqref{i1-bdg} implies that
\begin{eqnarray}\label{i-pl-min}
|I_{n,1}|&\leq &
 C_N\int_{t_{k_0}}^T ds \; \EX \Big(
 \int_{\phi_n(s)}^{\psi_n(s)}
 1_{\widetilde{G}_N(r)}\, | u(r)- u(\phi_n(s))|^2 \;  dr \Big)^{\frac{1}{2}}
\nonumber
\\
& \le & C_{N,T} 2^{-n} +
C_{N,T}\Big[\sum_{-k_0\leq i< k_1}\left( \cI_n^{(i,-)}+\cI_n^{(i,+)}\right) \Big]^{1/2}.
\end{eqnarray}
%where
%\[
% \cI_n^{(i,-)}=\int_{t_{k_0}}^{T-t_{k_1}} ds \; \EX
%  \int_{s+iT 2^{-n}}^{\chi_{i,n}(s)}
% 1_{\widetilde{G}_N(r)}\, | u(r)- u(\phi_n(s))|^2 \;  dr.
%\]
%and
%\[
% \cI_n^{(i,+)}=\int_{t_{k_0}}^{T-t_{k_1}} ds \; \EX
%  \int_{\chi_{i,n}(s)}^{s+(i+1)T 2^{-n}}
% 1_{\widetilde{G}_N(r)}\, | u(r)- u(\phi_n(s))|^2 \;  dr.
%\]
For any $r$ from the interval $[s+iT 2^{-n}, \chi_{i,n}(s))[$
we have   $\bar{r}_n= \chi_{i,n}(s))$; therefore
\begin{eqnarray*}
 \cI_n^{(i,-)} &\le & 2\int_{t_{k_0}}^{T-t_{k_1}} ds \; \EX
 \Big[ \int_{s+iT 2^{-n}}^{\chi_{i,n}(s)}
 1_{\widetilde{G}_N(r)}\, | u(r)- u(\bar{r}_n)|^2 \;  dr
\\ & &\qquad
 + T2^{-n} 1_{\widetilde{G}_N(\chi_{i,n}(s))}
 | u(\chi_{i,n}(s))- u(\phi_n(s))|^2\Big].
\end{eqnarray*}
Thus,  using Fubini's theorem and  \eqref{inc-prl} we can conclude
that
\[
\cI_n^{(i,-)}\le C(N,M,T) 2^{-3n/2}.
\]
Similarly
\begin{eqnarray*}
 \cI_n^{(i,+)} &\le & 2\int_{t_{k_0}}^{T-t_{k_1}} ds \; \EX
 \Big[ \int_{\chi_{i,n}(s)}^{s+(i+1)T 2^{-n}}
 1_{\widetilde{G}_N(r)}\, | u(r)- u(\underline{r}_n)|^2 \;  dr
\\ & &\qquad
 + T2^{-n} 1_{\widetilde{G}_N(\chi_{i,n}(s))}
 | u(\chi_{i,n}(s))- u(\phi_n(s))|^2\Big]
 \le C(N,M,T) 2^{-3n/2}.
\end{eqnarray*}
Hence  \eqref{i-pl-min} implies
$I_{n,1}\le C(N,M,T) 2^{-3n/4}$.  This inequality and
 the above upper estimates    for $I_{n,i}$ with $i\neq1$
  prove  \eqref{eq-increm-1}  in the case
 $\phi_n(s)= \phi^{*}_n(s):=  0\vee\left( s-k_0T2^{-n}\right)$.
In the general case,  we can write
\begin{eqnarray*}
\lefteqn{
 1_{\widetilde{G}_N(\psi_n(s))}\left|u(\psi_n(s))-u(\phi_n(s))\right|^2
}
\\ & &
\le 2 \,\Big( 1_{\widetilde{G}_N(\psi_n(s))}\left|
u(\psi_n(s))-u(\phi^*_n(s))\right|^2
+  1_{\widetilde{G}_N(\phi_n(s))}\left|
u(\phi_n(s))-u(\phi^*_n(s))\right|^2\Big);
\end{eqnarray*}
this concludes the proof of  Proposition~\ref{pr:increm} for functions $\phi_n$ and $\psi_n$
 which satisfy \eqref{phi-psi-1}.
\end{proof}
\smallskip

We also need a similar  for the time increments of the
approximate solutions $u^n$.
%%%% ????????
%It is more restrictive
%since the function $\phi_n$ has to be piece-wise constant.
%%%%%%%% ????????
\par

\begin{prop}\label{pr:increm-n}
 Let $\phi_n,\psi_n\, :\; [0,T]\mapsto [0,T]$ be  non-decreasing piecewise continuous functions
 such that condition \eqref{phi-psi-1} is satisfied for some positive
 integers $k_0$ and $k_1$.
%$\phi_n$ is piece-wise constant  taking values in $\{ kT2^{-n}\,
%:\, k=0, 1, \cdots, 2^n-1\}$, $\psi_n$ is
 %piece-wise continuous and
%the following  property }
%\begin{eqnarray}\label{phi-psi-1a}
%0\vee\left( s-k_0T2^{-n}\right) \le \phi_n(s)\le  \psi_n(s) \le s\wedge T
%\end{eqnarray}
%holds for some  integers $k_0>0$.
Fix $M>0$, let  $h\in\cA_M$ and $\xi$ be a $\cF_0$-measurable, $H$-valued  random variable
 such that $\EX |\xi|^4 <\infty$.
Let $u^n$ be the  solution to \eqref{apprx-eq};   for $N>0$ set
\begin{equation}\label{g-N-set2}
G^n_N(t)=\left\{\sup_{0\leq s\leq t}  |u^n(s)|\le N\right\}\cap
\left\{\int_0^t \|u^n(s)\|^2ds\le N\right\}\cap\,\Om_n(t),
\end{equation}
where $\Om_n(t)$ is defined in \eqref{om-n-t} and  let $\tau_n$ be the
stopping time defined in \eqref{tau-f}.
 There exists a constant
$C(N,M,T)$ such that
\begin{equation}\label{eq-increm-1-n}
\tilde I_n=\EX \int_0^{\tau_n} 1_{G^n_N(\psi_n(s))} \left|
u^n(\psi_n(s))-u^n(\phi_n(s))\right|^2 ds\le  C(N,M,T) n^{3/2} 2^{-3n/4}
\end{equation}
for every  $n=1,2,\ldots$
\end{prop}
\begin{proof} We use the same idea as in the proof of Proposition~\ref{pr:increm}
and at first  suppose that  $\phi_n(s)=0\vee\left( s-k_0T2^{-n}\right)$
for some $k_0\ge 0$. Let $t_{k_0}=k_0 T2^{-n}$; then we have
% and as in Proposition~\ref{pr:increm}
%we can see that
\begin{equation}\label{In-prime1-n}
\tilde I_n\le  C_{N,T}  2^{-n}+ \tilde I'_n,
\end{equation}
where
\begin{equation}\label{In-prime-n}
\tilde I'_n=\EX \int_{\tau_n\wedge t_{k_0}}^{\tau_n} 1_{G^n_N(\psi_n(s))} \left|
u^n(\psi_n(s))-u^n(\phi_n(s))\right|^2 ds . %,\quad t_{k_0}=k_0 T2^{-n}.
\end{equation}
It\^o's formula applied to   $\left|
u^n(.)-u^n(\phi_n(s))\right|^2$
 implies that   $\tilde I'_n=\sum_{1\leq i\leq 6} \tilde{I}_{n,i}$, where
\begin{eqnarray*}
\tilde I_{n,1} & = &
 2 \EX \int_{\tau_n\wedge t_{k_0}}^{\tau_n} ds 1_{G^n_N(\psi_n(s))}
\int_{\phi_n(s)}^{\psi_n(s)} \big( u^n(r)-u^n(\phi_n(s))\, ,\,  \s(u^n(r))dW(r)\big),
\\
\tilde I_{n,2}
&= &
\EX \int_{\tau_n\wedge t_{k_0}}^{\tau_n} ds
1_{G^n_N(\psi_n(s))} \int_{\phi_n(s)}^{\psi_n(s)}|\s (u^n(r))|^2_{L_Q}dr ,
\\
\tilde I_{n,3}&=&
2 \, \EX \int_{\tau_n\wedge t_{k_0}}^{\tau_n} ds 1_{G^n_N(\psi_n(s))}
\int_{\phi_n(s)}^{\psi_n(s)} \Big(u^n(r)-u^n(\phi_n(s))\, ,\,  \ts(u^n(r))\dot{\widetilde{W}}^n(r)\Big)\, dr ,
\\
\tilde I_{n,4}&=&
2 \, \EX \int_{\tau_n\wedge t_{k_0}}^{\tau_n} ds 1_{G^n_N(\psi_n(s))}
\int_{\phi_n(s)}^{\psi_n(s)} \big( u^n(r)-u^n(\phi_n(s))\, ,\,  G(u^n(r))h(r)\big)\, dr ,
\\
\tilde I_{n,5}&= &
2 \, \EX \int_{\tau_n\wedge t_{k_0}}^{\tau_n} ds  1_{G^n_N(\psi_n(s))}
\int_{\phi_n(s)}^{\psi_n(s)} \Big(u^n(r)-u^n(\phi_n(s)),
(\ro+\hf\tro-R)(u^n(r))\Big)\, dr ,
\\
\tilde I_{n,6}&= &
-2 \, \EX \int_{\tau_n\wedge t_{k_0}}^{\tau_n} ds  1_{G^n_N(\psi_n(s))}
\int_{\phi_n(s)}^{\psi_n(s)} \big( u^n(r)-u^n(\phi_n(s))\, ,\, Au^n(r)+B(u^n(r))\big)\, dr.
\end{eqnarray*}
Estimates for $\tilde I_{n,2}$,  $\tilde I_{n,4}$, $ \tilde I_{n,5}$ are obvious.  Indeed,  we can first
extend  outward integration to the time interval $[t_{k_0}, T]$ and  then use  the growth conditions
\eqref{s-bnd}, \eqref{rn-1} and \eqref{rn-3}. This yields  the estimate
\begin{equation}\label{I-2-4-5}
|\tilde I_{n,i}|\le C(N,T)\,  2^{-n},\quad i=2,4,5.
\end{equation}
Note that \eqref{I-2-4-5} holds as soon as $0\leq \psi_n(s)-\phi_n(s)\leq C 2^{-n}$ for some constant
$C>0$, and does not require the specific form of $\phi_n$.   Schwarz's inequality and Condition {\bf (B)} imply
\begin{eqnarray*}
 - (u^n(r)-u^n(\phi_n(s))\, ,\, Au^n(r)+B(u^n(r)))
 \le   C_1 \|u(\phi_n(s))\|^2 +C_2 |u(r)|^2 \|u(\phi_n(s))\|^4_\HH \\
 \le  C_0\|u(\phi_n(s))\|^2  \left[1 + |u(r)|^2 |u(\phi_n(s))|^2\right].
\end{eqnarray*}
Therefore, if  $\phi_n(s)=(s-{t_{k_0}})\vee 0$ we deduce
\begin{align} \label{In6-n}
|\tilde I_{n,6}|& \le
C(N,T) \EX \int_{\tau_n\wedge t_{k_0}}^{\tau_n} ds 1_{G^n_N(\psi_n(s))}
\int_{\phi_n(s)}^{\psi_n(s)}\|u^n(\phi_n(s))\|^2dr
\nonumber \\
%& \le
%C(N,T)  2^{-n} \EX \int_{\tau_n\wedge t_{k_0}}^{\tau_n} 1_{G^n_N(\psi_n(s))}
%\|u^n(\phi_n(s))\|^2 \, ds
%\nonumber \\
&\le
C(N,T) 2^{-n} \EX \int_{\tau_n\wedge t_{k_0}}^{\tau_n} 1_{G^n_N(\psi_n(s))}
\|u^n(s-t_{k_0})\|^2 \, ds
\nonumber \\
&\le
C(N,T) 2^{-n} \EX \int_{(\tau_n\wedge t_{k_0}-t_{k_0})^+}^{(\tau_n-t_{k_0})^+}
\|u^n(s)\|^2 ds\nonumber \\
& \le
C(N,T) 2^{-n} \EX \int_{0}^{\tau_n}
\|u^n(s)\|^2 ds\le C(N,T) 2^{-n}.
\end{align}
Using \eqref{tau-f} and the upper bound of $\psi_n(s)-\phi_n(s)$ (and not the specific form of $\phi_n$), we deduce
\begin{eqnarray}\label{In3-n}
|\tilde I_{n,3}| \le
 C_N  n 2^{-\frac{n}{2}} \EX \int_{\tau_n\wedge t_{k_0}}^{\tau_n}\!\! \!\!\!  ds 1_{G^n_N(\psi_n(s))}
\int_{\phi_n(s)}^{\psi_n(s)} \!\!\!
|u^n(r)-u^n(\phi_n(s))|dr \leq C(T,N)  n 2^{-\frac{n}{2}}.
\end{eqnarray}
%which implies
%\[
%|I_{n,3}|\le C_N  n 2^{-n/2}.
%\]
Since for $s\leq t$ we have $G_N^n(t)\subset G_N^n(s)$, the
local property of stochastic integrals,
the linear growth condition \eqref{s-bnd}
and Schwarz's inequality imply  that
\begin{align}  %\label{i1-bdg-n}
|&\tilde I_{n,1}|\leq 2 \sqrt{T} \Big( \EX \int_{t_{k_0}\wedge
\tau_n}^{\tau_n}\!\! ds 1_{G_N^n(\psi_n(s))} \nonumber \\
& \quad \quad \times \Big| \int_{\phi_n(s)\wedge\tau_n}^{\psi_n(s)\wedge \tau_n} 1_{G_N^n(r)} \, \big( \sigma(u^n(r))\, ,\, u^n(r)-u^n(\phi_n(s))\big)\, dW(r)
\Big|^2 \Big)^{\frac{1}{2}}\nonumber \\
&\quad \leq 2 \sqrt{T} \Big( \int_{0}^T \!\! ds \EX
\int_{\phi_n(s)}^{\psi_n(s)}  %\wedge \tau_n}^{\psi_n(s)\wedge \tau_n} \!\!
1_{G_N^n(r)}\, 1_{[0,\tau_n]}(r)\,  |\s(u^n(r))|_{L_Q}^2
 | u^n(r)- u^n(\phi_n(s))|^2 \,  dr \Big)^{\frac{1}{2}} \label{interm}
 \\
&\quad \leq
 C(N,T) \Big(   \int_{0}^T ds \;
 \int_{\phi_n(s)}^{\psi_n(s)}
   dr \Big)^{\frac{1}{2}}
=  C(N,T)\,  2^{-n/2}. \label{i1-bdg-n}
\end{align}
The inequalities    \eqref{In-prime1-n} %\eqref{I-2-4-5}
- \eqref{i1-bdg-n}  yield  for $\phi_n(s)=(s-t_{k_0})\vee 0$:
\begin{equation}\label{tildeI1}
\tilde I_n=  \EX \int_0^{\tau_n} 1_{G^n_N(\psi_n(s))} \left|
u^n(\psi_n(s))-u^n(\phi_n(s))\right|^2 ds\le  C(N,M,T) n 2^{-n/2}.
\end{equation}
  In order to obtain the final estimate in \eqref{eq-increm-1-n} we
need to improve the upper estimates of $\tilde I_{n,1}$ and $\tilde I_{n,3}$.
This can be done in a way similar to that used in the proof of previous Proposition.
One can easily  see  that \eqref{tildeI1}
holds when $\phi_n \leq \psi_n$ satisfy the assumptions of the
Proposition and $\phi_n$ is piece-wise constant.
Then   \eqref{In3-n} and Schwarz's inequality  obviously imply  that
\[
|\tilde I_{n,3}|\le C_N  n^{3/2} 2^{-3n/4}.
\]
Thus,  to conclude the proof we need to deal with the improvement of
$I_{n,1}$. Let  the function $\phi_n$ be  piece-wise constant; then
 given $r\in [\phi_n(s),\psi_n(s)]$, we have $\phi_n(s) \in \{
\underline{r-i2^{-n}} : 0\leq i\leq k_0\}$.
Therefore,  using the inequality
 \eqref{interm}, Fubini's theorem and \eqref{tildeI1} applied with
the functions $\phi_{n,i}(r)= \underline{r-iT2^{-n}}$ and $\psi_n(r)=r$ we deduce
\begin{align*}
 |\tilde I_{n,1}|&\leq C(N,T) \Big(\sum_{0\leq i\leq k_0} \EX \int_0^{\tau_n}\!\! \!  dr 1_{G_N^n(r)}
|u^n(\psi_n(r))-u^n(\phi_{n,i}(r)) )|^2
\int_r^{(r+k_0T2^{-n})\wedge T}\!\!\! ds \Big)^{\frac{1}{2}} \\
& \leq (k_0+1)^{\frac{1}{2}} \,  C(N,M,T) n^{\frac{1}{2}} 2^{-3n/4};
\end{align*}
this
concludes  the proof of \eqref{eq-increm-1-n} when $\phi_n$ is piece-wise constant.
To deduce that this inequality holds for arbitrary functions $\phi_n$ and $\psi_n$
satisfying \eqref{phi-psi-1}, apply
\eqref{eq-increm-1-n} for $\tilde \phi_n(s)=(s-t_{k_0})_n$ and either $\tilde{\psi}_n=\phi_n$
or   $\tilde{\psi}_n=\psi_n$; this concludes the proof.
\end{proof}
\par
Proposition~\ref{pr:increm-n} implies that
\begin{equation}\label{eq-increm2}
\EX \int_0^{\tau_n} \!\! 1_{G^n_N(s)}\left(
\left| u^n(s)-u^n(s_n)\right|^2
+ \left| u^n(s)-u^n(\underline{s}_n)\right|^2\right) ds
\le  C(T,N,M) n^{3/2} 2^{-3n/4}
\end{equation}
where $G^n_N(t)$ is defined  by \eqref{g-N-set2}.
This is precisely what we need below.

\section{Proof of convergence result}\label{pr-conv}
 The aim of this section is to prove Theorem \ref{th:WZ-appr}.  For every integer $n\geq 1$, $\tau_n$ is the stopping time
defined by \eqref{tau-f} and we prove \eqref{WZ-cnvrg-3}.
In the estimates below,   constants may change from line to line,
but we indicate their dependence on parameters when it
becomes important.
\par
From equations \eqref{apprx-eq} and \eqref{main-eq} we deduce:
\begin{eqnarray}\label{difrnce-eq}
u^n(t)-u(t)& = &  - \int_0^t\Big[ A[u^n(s)-u(s)]+  B(u^n(s))- B(u(s))+
R(u^n(s))-R(u(s))\Big]\, ds  \nonumber \\
&& +\int_0^t \big[G(u^n(s))-G(u(s))\big] h(s)\,  ds+
\int_0^t \big[ \s((u^n(s))-\s(u(s))\big]\,  dW(s)
\nonumber \\
&&
+  \int_0^t\Big[ \ts(u^n(s_n)) \dot{\widetilde{W}}^n(s)\, ds -\ts(u(s))\, d W(s)\Big]
- \int_0^t \Big(\ro+\hf\tro\Big) (u^n(s))\, ds\nonumber \\
&& +
 \int_0^t\big[ \ts(u^n(s)) - \ts(u^n(s_n)) \big]\dot{\widetilde{W}}^n(s)\, ds.
\end{eqnarray}
Let $\underline{t}_n$ and $\bar{t}_n$ be defined by \eqref{s-funct},  let
 $\cE_n$ denote projector in $L^2(0,T)$  on
the subspace  of step functions defined by
\[
(\cE_n f)(t)= \left( T^{-1} 2^n \int_{\underline{t}_n}^{\bar{t}_n} f(s)ds\right)\cdot
1_{[\underline{t}_n, \bar{t}_n[}(t)
\]
and let  $\de_n\, :\; L^2(0,T)\mapsto L^2(0,T) $ denote the  shift operator defined  by
\[
(\de_n f)(t)=f\left( (t +  T 2^{-n})\wedge T\right)\quad\mbox{for}\quad t\in[0,T].
\]
Using \eqref{beta-appr} and \eqref{W-appr} we deduce
\begin{equation}\label{ts-1}
 \int_0^t \ts(u^n(s_n)) \dot{\widetilde{W}}^n(s)ds
 =  \int_0^t\cE_n\left[ \left(\de_n [ 1_{[0,t]}\right)(s)
 \ts(u^n(\underline{s}_n))\circ \Pi_n\right] dW(s).
 \end{equation}
Hence
\begin{align*}
& u^n(t)-u(t) =    - \int_0^t\Big[A\big( u^n(s)-u(s)\big)+  B(u^n(s))- B(u(s))+
R(u^n(s))-R(u(s))\Big]\, ds
\nonumber \\
&  +\int_0^t \big[ G(u^n(s))-G(u(s))\big] h(s)\,  ds
 + \int_0^t\Big(  \big[\s+\tilde{\s}\big](u^n(s)) -
[\s+\tilde{\s}\big](u(s)) \Big)\, dW(s)\nonumber \\
&+  \int_0^t\big[ \ts(u^n(s)) - \ts(u^n(s_n)) \big]\dot{\widetilde{W}}^n(s)\, ds
 - \int_0^t \Big( \ro+\hf\tro\Big) (u^n(s))\, ds\ %nonumber \\
% & + \int_0^t\Big(  \big[\s+\tilde{\s}\big](u^n(s)) -
%[\s+\tilde{\s}\big](u(s)) \Big)\, dW(s) +
+ \int_0^t \widetilde{\Sigma}_n(s ) \, dW(s),
\end{align*}
where
\begin{equation}\label{Sig-def}
\widetilde{\Sigma}_n(s)=
% \s((u^n(s))-\s(u(s))+
\cE_n\left[ \left(\de_n  1_{[0,t]}\right)(s)
 \ts(u^n(\underline{s}_n))\circ \Pi_n\right]-\ts(u^n(s)).
 \end{equation}
 It\^o's formula implies that
\begin{align*}
%\lefteqn{
& |u^n(t)-u(t)|^2+ 2\int_0^t \!\! \|u^n(s)-u(s)\|^2 ds  %} &&
% \\&&
=
- 2\int_0^t \!\! \left\langle B(u^n(s))- B(u(s))\,,\, u^n(s)-u(s)\right\rangle ds
\nonumber \\
 &+ 2\int_0^t\!\! \Big( \left[G(u^n(s))-G(u(s))\right] h(s)
-\left[R(u^n(s))-R(u(s))\right]\, ,\,  u^n(s)-u(s)\Big)\, ds
\\ &
+
 2\int_0^t\!\! \Big(\big[ \ts(u^n(s)) - \ts(u^n(s_n)) \big]\dot{\widetilde{W}}^n(s)
 -  \Big(\ro+\hf\tro\Big) (u^n(s)), u^n(s)-u(s)\Big) ds
\nonumber \\ & + \int_0^t\!\! \left| \Sigma_n(s)\right|^2_{L_Q}ds
 + 2\int_0^t \Big( \Sigma_n(s) dW(s) \, ,\, u^n(s)-u(s)\Big),
\end{align*}
where
\begin{equation}\label{sig-def2}
\Sigma_n(s)=\widetilde{\Sigma}_n(s)+ \big(\s+\tilde{\s}\big)(u^n(s)) -
\big(\s+\tilde{\s}\big)(u(s)) .
\end{equation}

Using \eqref{bw-cut} and \eqref{tau-f},
 we have $\dot{\widetilde{W}}_n(s)=\dot{W}_n(s)$ on the set $\{s\leq \tau_n\}$;
let
\begin{align}\label{Z-def}
  Z^{(0)}_n(t) &= \int_0^t \!\!\! \big( \Sigma_n(s)  dW(s)\, ,\,  u^n(s)-u(s)\big), \\
 Z^{(1)}_n(t) &= \int_0^t \left|
\cE_n\left[ \left(\de_n  1_{[0,t]}\right)(s)
 \ts(u^n(\underline{s}_n))\circ \Pi_n\right]-\ts(u^n(s))
   \right|^2_{L_Q} ds,  \nonumber\\
   Z^{(2)}_n(t)&=
   \int_0^t\left(\left[ \ts(u^n(s)) - \ts(u^n(s_n)) \right]\dot{W}^n(s)
 -  \Big(\ro+\hf\tro\Big) (u^n(s))\, ,\,  u^n(s)-u(s)\right) ds. \nonumber
\end{align}
The  equation \eqref{diffB1} with $\eta=1/2$  and condition {\bf (GR)}
%relations  \eqref{s-lip}, \eqref{G-bnd-lip} and \eqref{R-bnd-lip}
yield
\begin{eqnarray}\label{ito-rel1}
|u^n(t\wedge \tau_n)-u(t\wedge\tau_n)|^2+ \int_0^{t\wedge\tau_n}  \|u^n(s)-u(s)\|^2 ds
\le  2 \sum_{0\leq i\leq 2} Z^{(i)}_n(t\wedge\tau_n) % +2Z^{(1)}_n(t\wedge\tau_n) +2 Z^{(2)}_n(t\wedge\tau_n)
\\
+ 2\int_0^{t\wedge\tau_n} \left( 2L+ \sqrt{L} |h(s)|_0 +R_1 +C_{1/2} \|u(s)\|^4_\HH \right)\,
 |u^n(s)-u(s)|^2 \, ds.
\nonumber
\end{eqnarray}
For every integer $n\geq 1$ and every $t\in [0,T]$, set
\begin{eqnarray}\label{Tn-def}
T_n(t)&=& \sup_{0\leq s\le t\wedge\tau_n}|u^n(s)-u(s)|^2+
\int_0^{t\wedge\tau_n} \|u^n(s)-u(s)\|^2 ds.
\end{eqnarray}
Using \eqref{bnd-HH} and Gronwall's lemma we conclude that  for all $t\in [0,T]$
\begin{equation}\label{ito-rel4}
\EX T_n(t)
\le  C  \sum_{0\leq i\leq 2}
\EX\Big( \sup_{s \le t\wedge\tau_n}\left|Z^{(i)}_n(s)\right|\Big).
\end{equation}
\subsection{Estimate for $Z^{(0)}_n$}
The Burkholder-Davies-Gundy inequality,   equations  \eqref{Sig-def},
\eqref{sig-def2} and \eqref{s-lip} %condition {\bf (S)} (i)
imply that for any $\eta >0$ there exists $C_\eta>0$ such that
\begin{align*}
 \EX\Big(& \sup_{s \le t\wedge\tau_n} \left|Z^{(0)}_n(s)\right|\Big) \le
3\, \EX\left\{
 \int_0^{t\wedge\tau_n}\left|u^n(s)-u(s)\right|^2
 \left| \Sigma_n(s) \right|^2_{L_Q} ds\right\}^{1/2} \\
&\le 3\, \EX\Big\{ \sup_{s \le t\wedge\tau_n}\left|u^n(s)-u(s)\right|
\Big[ \int_0^{t\wedge\tau_n}
 \left| \Sigma_n(s) \right|^2_{L_Q} ds\Big]^{1/2}\Big\} \\
 &\le  \eta\,  \EX T_n(t) + C_\eta\,  \EX
\Big( \int_0^{t\wedge\tau_n}
 \left| \Sigma_n(s) \right|^2_{L_Q} ds\Big)
\\ &
\le  \eta\,  \EX T_n(t)+ 2 \, C_\eta \, \EX Z^{(1)}_n(t\wedge\tau_n)
+ 4\, L\,   C_\eta \int_0^{t}\EX\left|u^n(s\wedge\tau_n)-u(s\wedge\tau_n)\right|^2ds.
\end{align*}
Thus if  $\eta = \frac{1}{2} $,
 \eqref{ito-rel4} and Gronwall's lemma
imply that for some constant $C$ which does not depend on $n$,
\begin{equation}\label{ito-rel4a}
\EX T_n(t)
\le
C\, \Big( \EX \sup_{s \le t\wedge\tau_n}Z^{(1)}_n(s) +
\EX \sup_{s\le t\wedge\tau_n} \big|Z^{(2)}_n(s)\big|\Big).
\end{equation}
\subsection{Estimate of $Z^{(1)}_n$}
The convergence of $Z^{(1)}_n$ is stated in the following
assertion.
\begin{lemma}
Fix $M>0$,  $h\in {\mathcal A}_M$ and let   $Z^{(1)}_n(t)$ be defined by  \eqref{Z-def}; then
for fixed $N$ and $m$  we have:
\begin{equation}\label{Z1-conv}
 \lim_{n\to\infty}\EX\Big( \sup_{t\in [0,T]}Z^{(1)}_n(t\wedge\tau_n)\Big)=0.
\end{equation}
\end{lemma}
\begin{proof}
For $k=0, \ldots, 2^n-1$ let $\Om_{n,k}=\{ t_k<t\wedge\tau_n\le t_{k+1}\}$,
where as above  we set $t_k=kT2^{-n}$.
We consider   $Z^{(1)}_n(t\wedge\tau_n)$ separately on each set $\Om_{n,k}$.
\par
We start with the case  $\om\in \Om_{n,k}$ for $k\ge 2$.  Then
$\left(\de_n  1_{[0,t\wedge\tau_n]}\right)(s)=1$
for $s\le  t_{k-1}$ and
\begin{eqnarray*}
 Z^{(1)}_n(t\wedge\tau_n) &=& \int_0^{t\wedge \tau_n} \left|
\cE_n\left[ \left(\de_n [ 1_{[0,t\wedge\tau_n]}\right)(s)
 \ts(u^n(\underline{s}_n))\circ \Pi_n\right]-\ts(u^n(s))
   \right|^2_{L_Q} ds  \nonumber\\
&\le& \sum_{0\leq i\leq k-2} \int_{t_i\wedge\tau_n}^{t_{i+1}\wedge\tau_n} \left|
\cE_n\left[
 \ts(u^n(\underline{s}_n))\circ \Pi_n\right]-\ts(u^n(s))
   \right|^2_{L_Q} ds
\\
& & + 2\int_{t_{k-1}\wedge\tau_n}^{t_{k+1}\wedge t\wedge\tau_n}\left( \left|
 \ts(u^n(\underline{s}_n))\circ \Pi_n\right|^2_{L_Q}+ \left|\ts(u^n(s))
   \right|^2_{L_Q}\right) ds.
\end{eqnarray*}
 Thus using \eqref{s-bnd} and   \eqref{bnd-n1} % and \eqref{s-lip}
 we deduce
that for some constant  $C=C(K_0,K_1, N,T) $ which does not depend on $n$:
\begin{align*}
 Z^{(1)}_n & (t\wedge\tau_n)
\leq
C 2^{-n} +\sum_{0\leq i\leq k-2} \int_{t_i\wedge \tau_n}^{t_{i+1}\wedge \tau_n} \left|
 \ts(u^n(\underline{s}_n))\circ \Pi_n -\ts(u^n(s))
   \right|^2_{L_Q} ds \\
& \le  C\, 2^{-n} + 2T\sup_{|u|\le 2(N+1)} \left|
 \ts(u)\circ \Pi_n -\ts(u))\right|^2_{L_Q}
 \\  & \qquad \quad  +  2 \sum_{0\leq i\leq k-2} \int_{t_i\wedge \tau_n}^{t_{i+1}\wedge \tau_n} \left|
 \ts(u^n(t_i)) -\ts(u^n(s))
   \right|^2_{L_Q} ds\\
& \le  C\, 2^{-n} + 2T\sup_{|u|\le 2(N+1)}  \left|  \ts(u)\circ \Pi_n -\ts(u))\right|^2_{L_Q}
 % \\  &
+  2 L \int_{0}^{t\wedge\tau_n}
  \left|  u^n(\underline{s}_n) -u^n(s) \right|^2 ds  %_{L_Q} ds }
\\
 & \le  C 2^{-n} + 2T\!\! \sup_{|u|\le 2(N+1)} \left|
 \ts(u)\circ \Pi_n -\ts(u))\right|^2_{L_Q}
 % \\  &
+  2 L \int_{0}^{\tau_n}\!\! 1_{G^n_N(s)} \left|
 u^n(\underline{s}_n) -u^n(s)
   \right|^2 ds,
\end{align*}
where $G^n_N(s)$ is defined  by \eqref{g-N-set2}.
Furthermore, given
 $\om\in \Om_{n,1}\cup \Om_{n,2} $ we have:
\begin{eqnarray*}
 Z^{(1)}_n(t\wedge\tau_n) &\leq &  2\int_{0}^{t_{2}\wedge t\wedge\tau_n}\left( \left|
 \ts(u^n(\underline{s}_n))\circ \Pi_n\right|^2_{L_Q}+ \left|\ts(u^n(s))
   \right|^2_{L_Q}\right) ds\le C 2^{-n},
\end{eqnarray*}
where $C=C(K_0,K_1, N,T) $ does not depend on $n$. This yields
\begin{eqnarray*}
\EX\Big( \sup_{t\in [0,T]}  Z^{(1)}_n(t\wedge\tau_n)\Big)
& \le & C\, 2^{-n} + 2T\sup_{|u|\le 2(N+1)} \left|
 \ts(u)\circ \Pi_n -\ts(u))\right|^2_{L_Q}
 \\ & & +  2 L \EX \int_{0}^{\tau_n}1_{G^n_N(s)} \left|
 u^n(\underline{s}_n) -u^n(s)
   \right|^2 ds;
\end{eqnarray*}
therefore,  \eqref{Z1-conv} follows from \eqref{sn-conv} and  \eqref{eq-increm2}.
\end{proof}
% \newpage

\subsection{Estimate of $Z^{(2)}_n$}
\subsubsection{\bf Main splitting}
The identities  \eqref{sj-def}, \eqref{W-appr} and \eqref{bw-cut} yield
\begin{eqnarray}\label{Z2-1}
   Z^{(2)}_n(t\wedge\tau_n)&=&
   \sum_{1\leq j\leq n} \int_0^{t\wedge\tau_n}\Big(\left[ \ts_j(u^n(s)) - \ts_j(u^n(s_n))
   \right]\dot{\beta_j}^n(s)\, ,\, u^n(s)-u(s)\Big) ds \nonumber
   \\ && -
   \int_0^{t\wedge\tau_n}\Big( \big(\ro+\frac{1}{2}\tro\big) (u^n(s))\, ,\,  u^n(s)-u(s)\Big) ds.
\end{eqnarray}
For every $j=1, \cdots, n$ Taylor's formula implies that
\begin{eqnarray*}
&& \ts_j(u^n(s)) - \ts_j(u^n(s_n)) = D \ts_j(u^n(s_n)) [ u^n(s) - u^n(s_n)]
\\
 &&  +
\int_0^1 (1-\mu)
d\mu \Big\langle D^2 \ts_j\big(u^n(s_n)+ \mu [u^n(s) - u^n(s_n)]\big)
   ; u^n(s) - u^n(s_n),  u^n(s) - u^n(s_n)\Big\rangle ,
\end{eqnarray*}
where $\langle D^2 \ts_j(v); v_1, v_2\rangle$ denotes
the value of the second Fr\'echet derivative  $D^2 \ts_j(v)$   on elements
$v_1$ and $v_2$.
Therefore
condition \eqref{D2s-bnd} and  the bound  \eqref{bnd-n1}
imply that for every $t\in [0,T]$,
\begin{equation}\label{Z2-2}
\left| Z^{(2)}_n(t\wedge\tau_n)\right|\le  T_n(t,1) +
 \left|\widetilde{Z}^{(2)}_n(t)\right|,
\end{equation}
where
\[
T_n(t,1)= C_2(2N+1) \sum_{1\le j\le n}\int_0^{t\wedge\tau_n}\left| u^n(s) - u^n(s_n)\right|^2
|\dot{\beta_j}^n(s)| \left|u^n(s)-u(s)\right| ds,
\]
and
\begin{eqnarray}\label{Z2-3}
   \widetilde{Z}^{(2)}_n(t) &=&
   \sum_{1\leq j\leq n} \int_0^{t\wedge\tau_n}\Big( D \ts_j(u^n(s_n))
    [u^n(s) - u^n(s_n)]\, ,\,  u^n(s)-u(s)\Big)\, \dot{\beta_j}^n(s) ds\nonumber
   \\ && -
   \int_0^{t\wedge\tau_n}\Big( \big(\ro+\hf\tro\big) (u^n(s))\, , \, u^n(s)-u(s)\Big) ds.
\end{eqnarray}
For $G_N^n(t)$ defined  by \eqref{g-N-set2}, one has
\[
T_n(t,1)\le C_N\sum_{1\leq j\leq n}\int_0^{t\wedge\tau_n}1_{G_N^n(s)} \left| u^n(s) - u^n(s_n)\right|^2
|\dot{\beta_j}^n(s)| \left|u^n(s)-u(s)\right| ds .
\]
Therefore,  \eqref{tau-f},   the inequalities  \eqref{bnd-bW} and \eqref{eq-increm2} yield
for some constant $C:=C(N,M,T)$
\begin{eqnarray}\label{Z2-t1}
\EX\Big(\!\sup_{t\in [0,T]} T_n(t,1)\Big)  \leq
 \tilde{C}_N n^{\frac{3}{2}} 2^{\frac{n}{2}}
\EX
\int_0^{\tau_n}\!\! 1_{G_N^n(s)} \left| u^n(s) - u^n(s_n)\right|^2 ds
\le C n^{3} 2^{-\frac{n}{4}}.
\end{eqnarray}
To bound $ \widetilde{Z}^{(2)}_n$,
 rewrite   $u^n(s) - u^n(s_n)$   in \eqref{Z2-3} using  the evolution equation \eqref{apprx-eq}.
This yields the following decomposition:
\begin{equation}\label{Z2-tilde}
 \widetilde{Z}^{(2)}_n(t)=\sum_{2\leq i\leq 6}  T_n(t,i),
\end{equation}
where
\begin{eqnarray}\label{T2}
    T_n(t,2) &=&
   \sum_{1\leq j\leq n}\int_0^{t\wedge\tau_n}\Big( D \ts_j(u^n(s_n))\cI_n(s, s_n)
 \dot{\beta_j}^n(s)\, ,\,  u^n(s)-u(s)\Big) ds
\nonumber    \\ && -
   \int_0^{t\wedge\tau_n}\Big( \big(\ro+\hf\tro\big) (u^n(s))\, ,\,  u^n(s)-u(s)\Big) ds
\end{eqnarray}
with
\begin{equation}\label{T2-i}
\cI_n(s, s_n) := %\equiv
   \int_{s_n}^s\s(u^n(r))dW(r) +   \int_{s_n}^s\ts(u^n(r))\dot{\widetilde{W}}^n(r) dr,
\end{equation}
\begin{eqnarray*}
 && T_n(t,3) =-
   \sum_{1\leq j\leq n}\int_0^{t\wedge\tau_n}\Big( D \ts_j(u^n(s_n))
   \Big[\int_{s_n}^s A u^n(r) dr\Big] \dot{\beta_j}^n(s)\, ,\,  u^n(s)-u(s)\Big) ds,\nonumber\\
 && T_n(t,4) =-
   \sum_{1\leq j\leq n}\int_0^{t\wedge\tau_n}\Big( D \ts_j(u^n(s_n))
   \Big[\int_{s_n}^s B(u^n(r))dr\Big]
\dot{\beta_j}^n(s)\, ,\,  u^n(s)-u(s)\Big) ds,\nonumber\\
  && T_n(t,5) =
   \sum_{1\leq j\leq n}\int_0^{t\wedge\tau_n}\Big( D \ts_j(u^n(s_n))
   \Big[\int_{s_n}^s G(u^n(r))h(r)dr\Big]
 \dot{\beta_j}^n(s)\, ,\,  u^n(s)-u(s)\Big) ds,\nonumber\\
  && T_n(t,6) = -
   \sum_{1\leq j\leq n}\int_0^{t\wedge\tau_n}\Big( D \ts_j(u^n(s_n))
   \Big[\int_{s_n}^s \widetilde{R}(u^n(r))dr\Big]
 \dot{\beta_j}^n(s)\, ,\,  u^n(s)-u(s)\Big) ds,\nonumber
\end{eqnarray*}
with $\widetilde{R}(u)=R(u)+ \ro(u)+\hf\tro(u)$.
The most difficult term to deal with  is $T_n(t,2)$ and therefore we devote several
separate subsections below to upper estimate it. Let us start with the easier case $3\leq i\leq 6$.
\subsubsection{ \bf Bound for $T_n(t,i)$, $3\le i\le6$}
By duality we obtain
\[
  |T_n(t,3)| =
   \sum_{1\leq j\leq n}\!\Big|\!\int_0^{t\wedge\tau_n}\!  \dot{\beta_j}^n(s) \Big(
 \int_{s_n}^s A^{1/2} u^n(r) dr\, ,\,
  A^{1/2}  \big[D\ts_j(u^n(s_n))\big]^*[u^n(s)-u(s)]\Big) ds\Big|.
\]
Therefore, using  \eqref{Ds*-bnd}, \eqref{bnd-n1} and \eqref{bnd-bW}
we deduce that  for every $\tilde t\in [0,T]$
\[
 \sup_{t\in[0,\tilde t]} |T_n(t,3)| \le
  C_3(2N+2)\, \alpha\,  n^{3/2}\, 2^{n/2}\int_0^{\tilde t\wedge\tau_n}\Big( \int_{s_n}^s \| u^n(r)\| dr\Big)  \,
\|u^n(s)-u(s)\| ds.
\]
For any $\eta >0$,  Schwarz's inequality yields  % the existence of a constant $C:=C(N,T,\eta)$,
\[
 \sup_{t\in[0,\tilde t]} |T_n(t,3)| \le\eta \int_0^{\tilde t\wedge\tau_n} \|u^n(s)-u(s)\|^2 ds
 +  C n^{3}\int_0^{\tau_n}\int_{s_n}^s \| u^n(r)\|^2 dr ds .
\]
 for some constant  $C:=C(N,T,\eta)$.
Finally,   Fubini's theorem and \eqref{bnd-n1}  imply that   %for every $\eta >0$,
\begin{equation}\label{T3}
\EX\Big( \sup_{t\in[0,T]} |T_n(t,3)|\Big) \le
\eta T_n(T)
 +  C(N,T,\eta) \, n^{3}\,  2^{-n}.
%\quad \mbox{for every $\eta>0$ and $\tilde t\in [0,T]$.}
\end{equation}
Similarly, using \eqref{preB} and \eqref{bnd-bW}  we obtain
\begin{eqnarray*}
 |T_n(t,4)| =
   \sum_{1\leq j\leq n} \Big|\int_0^{t\wedge\tau_n}  \dot{\beta_j}^n(s)   \Big(
   \Big[\int_{s_n}^s B(u^n(r))dr\Big]
, [D\ts_j(u^n(s_n))]^*[u^n(s)-u(s)]\Big) ds\Big|  \\ %\nonumber\\
\le
C \alpha\,  n^{3/2} 2^{n/2}\int_0^{t\wedge\tau_n}  ds
   \int_{s_n}^s \left\|u^n(r)\right\|_\HH^2 dr
 \sup_{1\leq j\leq n}\left\|[D\ts_j(u^n(s_n))]^*[u^n(s)-u(s)]\right\| . %\nonumber
\end{eqnarray*}
Thus the inequalities \eqref{Ds*-bnd}, \eqref{bnd-HH} and \eqref{bnd-n1} yield
that for some constant $C:=C(N,m,T)$:
\begin{equation}\label{T4}
\EX \Big(\sup_{t\in[0,T]} |T_n(t,4)|\Big) \leq
C_3(2N+2)\alpha m^2\,  n^{\frac{3}{2}} 2^{-\frac{n}{2}}\int_0^{\tau_n}\!\!
 \left\|u^n(s)-u(s)\right\| ds\le C  n^{\frac{3}{2}}  2^{-\frac{n}{2}}.
\end{equation}
Using  \eqref{Ds-bnd}, \eqref{bnd-n1}  and \eqref{bnd-bW} we deduce
\[
| T_n(t,5)| \le C_1(2N+2) \, \alpha\,  n^{3/2}\,  2^{n/2} \int_0^{t\wedge\tau_n}
 \Big( \int_{s_n}^s  \left|G(u^n(r))h(r)\right|dr\Big) \,
| u^n(s)-u(s)| ds.
\]
Therefore,  \eqref{G-bnd-lip},  \eqref{bnd-n1} and Fubini's theorem yield for some constant $C:=C(N,M,T)$
\begin{equation}\label{T5}
\EX\Big( \sup_{t\in[0,T]} | T_n(t,5)|\Big)  \le C_N\, n^{3/2}\,  2^{n/2} \EX \int_0^{\tau_n}
  \int_{s_n}^s  \left|h(r)\right|dr ds \le  C \, n^{3/2}\,  2^{-n/2}.
\end{equation}
Similarly,  relying on \eqref{rn-1}, \eqref{R-bnd-lip}, \eqref{bnd-bW} and \eqref{bnd-n1}
we deduce
\begin{equation}\label{T6}
\EX\Big(\sup_{t\in[0,T]} | T_n(t,6)|\Big) \le  C_{K, R_0,N} n^{3/2} 2^{-n/2}.
\end{equation}
Thus, collecting the relations in  \eqref{ito-rel4a}--\eqref{T6},   and choosing $\eta>0$
small enough in \eqref{T3},  we obtain the following assertion:
\begin{prop}\label{pr:t-n2}
Let the assumptions of Theorem \ref{th:WZ-appr} be satisfied,
$T_n(t)$ be  defined by \eqref{Tn-def}; then we have:
\[
\EX T_n(T)\le \gamma_n(N,M,m,T) +C\EX \Big(\sup_{t\in[0,T]} | T_n(t,2)|\Big),
\]
%for every $t\in [0,T]$,
 where $\lim_{n\to\infty} \gamma_n(N,M,m,T)=0$ and
$T_n(t,2)$ is defined by \eqref{T2}.
\end{prop}
\subsubsection{\bf Splitting of $T_n(t,2)$.}
 Let    $T_n(t,2)$ be defined   by \eqref{T2}; then we have the following
decomposition:
\begin{equation}\label{T2-si}
 T_n(t,2)=\sum_{1\leq i\leq 7} S_n(t,i),
\end{equation}
where
\begin{align*}
 & \!  S_n(t,1) =
   \sum_{j=1}^n\!\!\int_0^{t\wedge\tau_n}\!\!\!   \dot{\beta_j}^n(s)\Big( D \ts_j(u^n(s_n))\cI_n(s, s_n)
, [u^n(s)-u(s)]- [u^n(s_n)-u(s_n)]\Big) ds,
    \\
& S_n(t,2) =-
   \int_0^{t\wedge\tau_n}\! \Big( \big(\ro+\hf\tro\big)(u^n(s))-
   \big(\ro_n+\hf\tro_n\big)(u^n(s))\, ,\,
   u^n(s)-u(s)\Big) ds,
    \\
&  S_n(t,3) =-
   \int_0^{t\wedge\tau_n}\Big( \big(\ro_n+\hf\tro_n\big)(u^n(s))-
   \big(\ro_n+\hf\tro_n\big)(u^n(s_n)),
   u^n(s)-u(s)\Big) ds\, ,\,
      \\
& S_n(t,4) =-
   \int_0^{t\wedge\tau_n}\Big(
   \big(\ro_n+\hf\tro_n\big)(u^n(s_n)),
  [u^n(s)-u(s)]- [u^n(s_n)-u(s_n)]\Big) ds\, ,\,
\\
 &   S_n(t,5) =
   \sum_{1\leq j\leq n}\!\int_0^{t\wedge\tau_n}  \!\!  \dot{\beta_j}^n(s)
\Big( D \ts_j(u^n(s_n))
\Big[ \int_{s_n}^s\left[ \s(u^n(r))- \s(u^n(s_n))\right]dW(r)
\\
  &  {}\qquad  {}\qquad\qquad\qquad
  + \int_{s_n}^s\left[ \ts(u^n(r))- \ts(u^n(s_n))\right]\dot{\widetilde{W}}^n(r) dr
\Big]
\, ,\,  u^n(s_n)-u(s_n)\Big) ds,
 \end{align*}
\begin{align}\label{sn6-def}
 &   S_n(t,6) =
  \int_0^{t\wedge\tau_n}\!\!\Big(   \sum_{1\leq j\leq n} \dot{\beta_j}^n(s) D \ts_j(u^n(s_n))
\big[ \s(u^n(s_n)) \big( W(s)-W(s_n)\big)\big]   -\ro_n(u^n(s_n))\, ,
\nonumber
\\&  {}\qquad  {}\qquad\qquad\qquad\qquad  % -\ro_n(u^n(s_n)),
u^n(s_n)-u(s_n)\Big) ds,
 \end{align}
\begin{align}\label{sn7-def}
 &   S_n(t,7) =
   \int_0^{t\wedge\tau_n}\!\!\Big(   \sum_{1\leq j\leq n}   \dot{\beta_j}^n(s)  D \ts_j(u^n(s_n))
\Big[ \ts(u^n(s_n))\Big( \int_{s_n}^s\dot{\widetilde{W}}^n(r)dr\Big) \Big]
    -\hf\tro_n(u^n(s_n))\, ,
\nonumber \\
  &  {}\qquad  {}\qquad\qquad\qquad\qquad
% -\hf\tro_n(u^n(s_n)),
u^n(s_n)-u(s_n)\Big) ds.
\end{align}
The most difficult terms to deal with are $S_n(t,6)$ and $S_n(t,7)$. We start with the simpler ones
$S_n(t,i)$, $i=1, ...,5$.
\subsubsection{\bf Bound for  $S_n(t,1)$.}
Let $\cI_n(s, s_n)$ be defined in \eqref{T2-i}; using \eqref{Ds-bnd},  \eqref{bnd-n1} and \eqref{bnd-bW}
we obtain
\[
   |S_n(t,1)| \!\le \! C_1(2N+2) \, \alpha\,
 n^{3/2}\, 2^{n/2}\!\int_0^{t\wedge\tau_n}\!\!\!\left|\cI_n(s, s_n)\right|
 \left( |u^n(s)-u^n(s_n)| +|u(s)- u(s_n)|\right) ds.
\]
 %   \\
 %& \le &   C_N
 %n^{3/2}2^{n/2}\left[\cN_n
 %\right]^{1/2} \left[\int_0^{\tau_n}\left|\cI_n(s, s_n)\right|^2ds\right]^{1/2},
 Thus for $\cN_n= \int_0^{\tau_n}\!\! \left( |u^n(s)-u^n(s_n)|^2 +|u(s)- u(s_n)|^2\right) ds$,
Schwarz's inequality yields
\begin{equation}\label{s1-1}
\EX\Big(\sup_{t\in [0,T]} |S_n(t,1)|\Big)\le
C_N\,
 n^{3/2}\, 2^{n/2}\, \left[\EX\cN_n
 \right]^{1/2} \, \Big[\EX\int_0^{\tau_n}\left|\cI_n(s, s_n)\right|^2ds\Big]^{1/2}.
\end{equation}
Since
\[
\EX \cN_n\le
 \EX\int_0^{\tau_n} 1_{G_N^n(s)}\left( |u^n(s)-u^n(s_n)|^2 +|u(s)- u(s_n)|^2\right) ds
\]
with $G_N^n(s)$ defined by \eqref{g-N-set2}, using  %\eqref{eq-increm} and
\eqref{eq-increm-1} with $\phi_n(s)=s_n$ and $\psi_n(s)=s$ we deduce:
\begin{equation}\label{s1-2}
\EX \cN_n\le C(N,M,T) \, n^{3/2} \, 2^{-{3n}/{4}}.
\end{equation}
Furthermore, the local property of the stochastic integral and \eqref{bnd-bW} yield
\begin{align*}
\EX &\int_0^{\tau_n}\left|\cI_n(s, s_n)\right|^2ds
\\ &
\le  C \int_0^T\left[\EX\Big|
\int_{s_n}^s\! 1_{G_N^n(r)}\s(u^n(r))dW(r)\Big|^2 +
 \EX\Big| \int_{s_n}^s\! 1_{G_N^n(r)}\ts(u^n(r))\dot{W}^n(r) dr\Big|^2\right]ds
 \\
 &\le  C \int_0^T\left[\EX
\int_{s_n}^s\! 1_{G_N^n(r)}\left|\s(u^n(r))\right|^2_{L_Q}dr +
\alpha^2\, n^2 \,  \EX\int_{s_n}^s\! 1_{G_N^n(r)}
\left| \ts(u^n(r))\right|^2_{L_Q} dr\right]ds.
\end{align*}
Thus by \eqref{s-bnd} and the definition  of the set $G_N^n(s)$ given in \eqref{g-N-set2},
we deduce:
\begin{equation}\label{s1-3}
\EX\int_0^{\tau_n}\left|\cI_n(s, s_n)\right|^2ds\le C(N,M,T) \, n^2\, 2^{-n}.
\end{equation}
Consequently the inequalities \eqref{s1-1} -- \eqref{s1-3} yield
\begin{equation}\label{s1-fin}
\EX\Big(\sup_{t\in [0,T]} |S_n(t,1)|\Big)\le
C(N,M,T) \, n^{13/4}\, 2^{-3n/8}.
\end{equation}
\subsubsection{\bf Bound for  $S_n(t,2)$.}
The inequality \eqref{bnd-n1} implies that
\[
\sup_{t\in [0,T]} |S_n(t,2)|\le
C(N)T
\sup_{|u|\le 2(N+1)}\left\{ |\varrho_n(u)-\ro(u)| +|\tro_n(u)-\tro(u)|\right\}.
\]
Therefore, the locally uniform convergence   \eqref{rn-3}
of $\rho_n$ to $\rho$ and $\tilde{\rho}_n$ to $\tilde \rho$ respectively  yields
\begin{equation}\label{t-s2}
\lim_{n\to\infty}\EX\Big(\sup_{t\in [0,T]} |S_n(t,2)|\Big)=0.
\end{equation}
\subsubsection{\bf Bound for  $S_n(t,3)$.}
The local Lipschitz property \eqref{rn-2}  and \eqref{bnd-n1} imply
\begin{align*}
 |S_n(t,3)| & \le 2\,   {\bar C}_{2N+2}    \sqrt{N+1}
   \int_0^{\tau_n}\left|u^n(s)-u^n(s_n)\right| ds
  \\ & \le C(N,T)\Big[ \int_0^{\tau_n}\!\!
   1_{G_N^n(s)}\left|u^n(s)-u^n(s_n)\right|^2 ds
   \Big]^{1/2},
\end{align*}
where  $G_N^n(s)$ is defined by \eqref{g-N-set2}. Thus
\eqref{eq-increm2} yields
\begin{equation}\label{t-s3}
\EX\Big(\sup_{t\in [0,T]} |S_n(t,3)|\Big)\le C(N,M,T)\, n^{3/4}\,  2^{-3n/8} .
\end{equation}
\subsubsection{\bf Bound for  $S_n(t,4)$.}
The local growth condition \eqref{rn-1}, relations \eqref{eq-increm-1} and
 \eqref{eq-increm2}, and also  Schwarz's inequality
imply
\begin{align}\label{t-s4}
\EX &\Big(\sup_{t\in [0,T]} |S_n(t,4)|\Big)
\le 2   {\bar  K} _{2N+2}    \EX \int_0^{\tau_n}\left(
     |u^n(s)-u^n(s_n)|+ |u(s)-u(s_n)|\right) ds
    \nonumber
     \\
    & \le
2    {\bar K} _{2N+2}    \, \sqrt{T}\,  \Big[ \EX \int_0^{\tau_n}\left( 1_{G_N^n(s)}
     |u^n(s)-u^n(s_n)|^2+ 1_{\widetilde{G}_N(s)}|u(s)-u(s_n)|^2\right) ds\Big]^{1/2}
\nonumber
     \\
      & \le C(N,M,T) \, n^{3/4}\,  2^{-3n/8}.
\end{align}
\subsubsection{\bf Bound for  $S_n(t,5)$.}
The local bound  \eqref{Ds-bnd} together with the inequalities \eqref{bnd-n1}  and  \eqref{bnd-bW} yield
\begin{eqnarray*}
  | S_n(t,5)| &\le & C_1(2N+2) \, \alpha\,  n^{3/2} 2^{n/2}
   \int_0^{\tau_n}\Big\{
\Big| \int_{s_n}^s\left[ \s(u^n(r))- \s(u^n(s_n))\right]dW(r)
\Big|
\\
  & &  +\Big|
  \int_{s_n}^s\left[ \ts(u^n(r))- \ts(u^n(s_n))\right]\dot{\widetilde W}^n(r) dr
\Big| \Big\}\, ds.
\end{eqnarray*}
Using Schwarz's inequality, \eqref{s-lip} and \eqref{bnd-bW}, we deduce
\begin{align*}
 \EX \Big(\! \sup_{t\in [0,T]} &  |S_n(t,5)|\Big)
% \\
% &
 \le  C_N\, n^{3/2}\, 2^{n/2}\,
\Bigg\{ \! \int_0^{T} \!\! \EX
\Big| \int_{s_n\wedge\tau_n}^{s\wedge \tau_n}  \!\!\!
  1_{G_N^n(r)}\left[ \s(u^n(r))- \s(u^n(s_n))\right]dW(r)
\Big|^2 ds
\\
 & \quad \qquad  + \alpha^2\,  n^2\,  2^n \, L \, \EX \int_0^{\tau_n} \Big|
  \int_{s_n}^s\left| u^n(r)- u^n(s_n)\right| dr
\Big|^2 ds\Bigg\}^{1/2}
\\
 & \le  C_N\, n^{3/2}\,  2^{n/2}\, \sqrt{L} \,
\Bigg\{  \int_0^{T} ds\, \EX
 \int_{s_n\wedge\tau_n}^{s\wedge\tau_n}   1_{G_N^n(r)}\, \left|u^n(r)- u^n(s_n)\right|^2\,  dr
\\
  & \quad  \qquad   + \alpha^2\,  n^2\,  T \,  \EX \int_0^{\tau_n} ds
  \int_{s_n}^s\left| u^n(r)- u^n(s_n)\right|^2 dr
\Bigg\}^{1/2}
\\
 & \le  C_N n^{3/2} 2^{n/2}\sqrt{L(1+ \alpha^2 n^2 T)}
\left\{  \int_0^{T} ds\, \EX
 \int_{s_n\wedge \tau_n}^{s\wedge\tau_n}   1_{G_N^n(r)}\left|u^n(r)- u^n(s_n)\right|^2 dr
\right\}^{1/2} .
%\\
% & \le  C_{N,T,L} n^{5/2}
%\left\{ 2^{-n} + 2^{n} \EX \int_0^{\tau_n} ds\,
% \int_{s_n}^{s}   1_{G_N^n(r)}\left|u^n(r)- u^n(s_n)\right|^2 dr
%\right\}^{1/2}.
\end{align*}
Fubini's theorem and \eqref{eq-increm2} imply that
\begin{align}\label{t-s5}
 \EX \Big(\sup_{t\in [0,T]}& |S_n(t,5)|\Big) \le \sqrt{L}  C(N,T)\, n^{5/2}\, 2^{n/2}
% \Big[  2^{-n}
\nonumber\\
&
\quad \times   \Big( \EX \int_0^{\tau_n} 1_{G_N^n(r)}\big[ \left|u^n(r)- u^n(r_n)\right|^2
+  \left|u^n(r)- u^n(\underline{r}_n)\right|^2 \big]\, 2T2^{-n} \, dr\Big)^{\frac{1}{2}}
%\Big]
\nonumber\\&
 \le C(N,M,T)\, n^{13/4}\,  2^{-3n/8}.
\end{align}
Proposition~\ref{pr:t-n2}  and the relations in \eqref{s1-fin}--\eqref{t-s5}
imply  the following assertion:
\begin{prop}\label{pr:t-s67}
Let the assumptions of Theorem \ref{th:WZ-appr} be satisfied and let
 $T_n(t)$ be  defined  by \eqref{Tn-def}; then we have:
\[
\EX T_n(T)\le \gamma^*_n(N,M,m,T) +C\EX \Big(
\sup_{t\in[0,T]} | S_n(t,6)|+ \sup_{t\in[0,T]} | S_n(t,7)|\Big) ,
\]
 where $\lim_{n\to\infty} \gamma^*_n(N,M,m,T)=0$,
$S_n(t,6)$ and $S_n(t,7)$ are defined  by \eqref{sn6-def} and \eqref{sn7-def}.
\end{prop}
The upper estimates of $S_n(t,6)$ and $S_n(t,7)$ are the key ingredients  of the proof; they justify
the drift correction term in the definition of $u^n$.

\subsubsection{\bf Bound for  $S_n(t,6)$.}
\begin{lemma}\label{S6}
Let the assumptions of Theorem \ref{th:WZ-appr} be satisfied
and $S_n(t,6)$ be given   by \eqref{sn6-def}. Then there exists a constant $C(N,T)$
such that
\begin{equation}\label{t-s6}
\EX\Big(\sup_{t\in [0,T]} |S_n(t,6)|\Big) \leq C(N,T)\,  n\, 2^{-\frac{n}{2}}.
\end{equation}
\end{lemma}
\begin{proof}
For $t\in [0,T]$ set
\begin{align*} %\label{t-s6-1}
 &U^n_j(s) = \dot{\beta_j}^n(s)\,  D \ts_j(u^n(s_n))
\big( \s(u^n(s_n)) \left[ W(s)-W(s_n)\right] \big)\, ,
\nonumber \\
%\end{equation}
%\begin{equation}
\nonumber %\label{t-s6-3a}
&\Delta_n(s)=
\Big( \sum_{1\leq j\leq n} U_j^n(s) -\ro_n(u^n(s_n))\, ,\,
   u^n(s_n)-u(s_n)\!\Big).
\end{align*}
We also have an obvious decomposition
\[
\sum_{1\leq j\leq n} U_j^n(s) -\ro_n(u^n(s_n))= \sum_{1\leq i\leq 3} V_n^{(i)}(s),
\]
where  \eqref{W-n}  yields
\begin{align*}
& V_n^{(1)}(s)= \sum_{1\leq j\leq n} D \ts_j(u^n(s_n))
\s(u^n(s_n)) \left[ W(s)-W(\underline{s}_n)\right] \dot{\beta_j}^n(s),
\\
& V_n^{(2)}(s)=\sum_{1\leq j\leq n} \sum_{l\neq j} D \ts_j(u^n(s_n))
\s_l(u^n(s_n)) \left[\beta_l(\underline{s}_n)-\beta_l(s_n) \right]\ {2^n} T^{-1}
\left[\beta_j(\underline{s}_n)-\beta_j(s_n) \right],
\\
& V_n^{(3)}(s)= \sum_{1\leq j\leq n} D \ts_j(u^n(s_n))
\s_j(u^n(s_n)) \left[ {2^n} T^{-1}
\Big( \beta_j(\underline{s}_n)-\beta_j(s_n) \Big)^2 -1\right].
\end{align*}
The obvious identity
\begin{equation}\label{1-split}
1_{\{s\le\tau_n\}}=1_{\{s_n\le\tau_n\}}- 1_{\{s_n\le\tau_n< s\}}
\end{equation}
yields the  following decomposition, where $G_N^n(t)$ is defined by \eqref{g-N-set2}:
\begin{equation}\label{t-s6-2}
%\lefteqn{
    S_n(t,6) =
   \int_0^{t\wedge\tau_n} \Delta_n(s) \, 1_{G_N^n(s_n)}\, ds
    =\sum_{1\leq i\leq 3} S^{(i)}_n(t)-S^{(4)}_n(t),
\end{equation}  %] % end{eqnarray*}
with
\begin{eqnarray*}% \label{t-s6-2}
S^{(i)}_n(t)&=& \int_0^t  1_{\{s_n\le\tau_n\}}1_{G^n_N(s_n)}
\big( V_n^{(i)}(s)\, ,\, u^n(s_n)-u(s_n)\big)\,  ds, \qquad i=1,2,3, \\
S^{(4)}_n(t)&=& \int_0^t  1_{\{s_n\le\tau_n<s\}} 1_{G^n_N(s_n)}
\Delta_n(s) ds.
\end{eqnarray*}
We note that $S^{(i)}_n(t)=0$ for  every $i=1,2,3$  and $t\le t_2$.
\smallskip

\noindent {\it Bound  for  $S^{(4)}_n$.}\quad  Set $t_{-1}=t_0=0$; using twice Schwarz's inequality, we deduce
\begin{align*}
\EX \Big(\sup_{t\in [0,T]} &|S^{(4)}_n(t)|\Big) \le
%\EX \int_0^T  1_{\{s_n\le\tau_n\le s\}} 1_{G^n_N(s_n)}\left|\Delta_n(s)\right| ds
%\\ & \le
\sum_{0\leq k<2^n}
 \EX  \int_{t_k}^{t_{k+1}}
 1_{\{t_{k-1}\le\tau_n\le t_{k+1}\}} 1_{G^n_N(s_n)}\,\left|\Delta_n(s)\right| ds
\\
&  \le \Big\{2\, \sum_{0\leq k<2^n} \EX 1_{\{t_k\le\tau_n\le t_{k+1}\}}  \Big\}^{1/2}
\Big\{\sum_{0\leq k<2^n}\EX  \Big(\int_{t_k}^{t_{k+1}}
  1_{G^n_N(s_n)}\left|\Delta_n(s)\right| ds\Big)^2   \Big\}^{1/2}
  \\
%&  \le
%\left\{\sum_{k=0}^{2^n-1}\EX  \left[\int_{t_k}^{t_{k+1}}
%  1_{G^n_N(s_n)}\left|\Delta_n(s)\right| ds\right]^2   \right\}^{1/2}
%    \\
&  \le
\sqrt{2}\, \Big\{T 2^{-n} %\sum_{k=0}^{2^n-1}
\EX \int_0^T %{t_k}^{t_{k+1}}
  1_{G^n_N(s_n)} \left|\Delta_n(s)\right|^2 ds   \Big\}^{1/2}.
\end{align*}
 Schwarz's inequality,  \eqref{s-bnd}, \eqref{Ds-bnd}, \eqref{rn-1} and
the definition \eqref{g-N-set2} of the set $G^n_N(s_n)$  yield
\begin{eqnarray*}
 1_{G^n_N(s_n)} \left|\Delta_n(s)\right|^2& \le &  { C}(N)   \Big(1+\sum_{1\leq j\leq n}
 \left| W(s)-W(s_n)\right|_0 |\dot{\beta_j}^n(s)| \Big)^2 \\
 &\le &   {C}(N)   \,
  \Big(1+ n  \left| W(s)-W(s_n)\right|_0^2 \sum_{1\leq j\leq n}|\dot{\beta_j}^n(s)|^2 \Big).
 \end{eqnarray*}
Therefore, Schwarz's inequality implies that for some constants $C(N,T)$, $c_1,c_2$, one has:
\begin{align}\label{t-s6-s2}
\EX & \Big(\sup_{t\in [0,T]} |S^{(4)}_n(t)|\Big)
 \nonumber \\
 & \le C(N,T) \, 2^{-n/2}\,
\Big\{1+   n\sum_{1\leq j\leq n}  \int_{0}^{T} \Big[\EX
 \left| W(s)-W(s_n)\right|_0^4\Big]^{1/2}
 \Big[\EX  |\dot{\beta_j}^n(s)|^4 \Big]^{1/2} ds   \Big\}^{1/2}
 \nonumber  \\
 & \le C(N,T) 2^{-n/2}
\Big\{1+   n^2
\Big[c_1  \frac{T^2}{ 2^{2n}} \Big]^{1/2}
 \Big[ T^{-4} 2^{4n}c_2 \frac{ T^2}{ 2^{2n}} \Big]^{1/2} %ds
 \Big\}^{1/2}
 \le C(N,T) \, n\,  2^{-\frac{n}{2}}.
\end{align}

%\subsubsection{Decomposition of  $S^{(1)}_n$.}
%\label{6.3.11}
%Note that by the definition of  $U_j^n$ % defined by \eqref{t-s6-1},
% we have
%Thus we have that
%$S^{(1)}_n(t)=\sum_{1\leq i\leq 3} S^{(1,i)}_n(t)$,  % +S^{(1,2)}_n(t)+S^{(1,3)}_n(t)$,
% where for $i=1,2,3$,
%\begin{equation}\label{s6-s1-i}
%S^{(1,i)}_n(t)= \int_{t_2}^t  1_{\{s_n\le\tau_n\}}1_{G^n_N(s_n)}
%(V_n^{(i)} (s), u^n(s_n)-u(s_n)) ds. % \quad i=1,2,3.
%\end{equation}

\noindent %subsubsection
{\it Bound for $S^{(1)}_n$.}\quad
Using duality and Fubini's theorem,  we can write
\begin{align*}
S^{(1)}_n(t) =& \sum_{1\leq j\leq n} \int_{0}^t  1_{\{s_n\le\tau_n\}}1_{G^n_N(s_n)}  \, \dot{\beta_j}^n(s)\\
& \quad \times
\int_{\underline{s}_n}^s\big(
\left[ D\ts_j(u^n(s_n))
\s(u^n(s_n))\right]^* ( u^n(s_n)-u(s_n)) \,,\, d W(r)\big)ds
\\
=& \sum_{1\leq j\leq n} \int_{0}^t   \Big( \int_r^{\bar{r}_n}
 1_{\{s_n\le\tau_n\}}1_{G^n_N(s_n)}  \,  \dot{\beta_j}^n(s)  \\
 & \quad \times
\left[ D\ts_j(u^n(s_n))
\s(u^n(s_n))\right]^* ( u^n(s_n)-u(s_n))\, ds\,  ,\,  d W(r)\Big).
\end{align*}
Since $\dot{\beta}_j^n(s)$ is ${\mathcal F}_{\underline{s}_n} = {\mathcal F}_{\underline{r}_n} $
adapted for $r\leq s\leq \bar{r}_n$,  the process $S^{(1)}_n$ is a  martingale. Therefore, the
 Burkholder-Davies-Gundy and Schwarz inequalities, \eqref{s-bnd}, \eqref{Ds-bnd} and \eqref{g-N-set2}
 imply
\begin{align} \label{s6-s11}
\EX & \Big( \sup_{t\in [0,T]} |S^{(1)}_n(t)|\Big)
\le c_0 \, \EX\Big\{\int_{0}^T \Big|  \sum_{1\leq j\leq n}  \int_r^{\bar{r}_n}
 1_{\{s_n\le\tau_n\}}1_{G^n_N(s_n)}\,  \dot{\beta_j}^n(s)
\nonumber  \\
&  \qquad \qquad \times
\left[ D\ts_j(u^n(s_n))
\s(u^n(s_n))\right]^* ( u^n(s_n)-u(s_n)) ds\Big|^2 dr\Big\}^{1/2}
\nonumber \\
\leq &
 \frac{C(N,T) \sqrt{n}}{2^{n/2}}
\EX\Big\{\int_{0}^T \!\!dr \!\! \sum_{1\leq j\leq n}  \int_r^{\bar{r}_n} \!\!
 1_{\{s_n\le\tau_n\}}1_{G^n_N(s_n)}
\left| D\ts_j(u^n(s_n))
\s(u^n(s_n))\right|^2 |\dot{\beta_j}^n(s)|^2 ds\Big\}^{1/2}
\nonumber \\
% \le & \e \EX \left(\sup_{t\in [0,T]} | u^n(t\wedge\tau_n)-u(t\wedge\tau_n))|^2\right)+
\le & % \frac{C_{N,T}}{\e}
C(N,T)\,  \sqrt{n} 2^{-n/2}\Big\{
\int_{0}^T \!\!dr \sum_{1\leq j\leq n} \EX \int_r^{\bar{r}_n}
|\dot{\beta_j}^n(s)|^2 ds\Big\}^{1/2} \leq
{C(N,T)}\, n\,   2^{-n/2}.
\end{align}

%\subsubsection
\noindent{\it Bound for $S^{(2)}_n$. }\quad
For $j=1, \cdots,n$, $l\neq j$, $i=1, \cdots, 2^n-1$, set
\[ \Phi_{j,l}(i)=\Big( D\ts_j(u^n(t_i)) \s_l(u^n(t_i))\, ,\, u^n(t_i)) - u(t_i)\Big)\, 1_{\{t_i\leq \tau_n\}}\,
1_{G_N^n(t_i)},\]
and for $k=3, \cdots, 2^n$, let
\[ M_k=\sum_{2\leq i\leq k} \sum_{1\leq j\leq n}\sum_{l\neq j} \Phi_{j,l}(i-1) \,
\big( \beta_l(t_i)-\beta_l(t_{i-1} )\big)   \, \big(
\beta_j(t_i)-\beta_j(t_{i-1} )\big).\]
 Then the
random variable  $   \Phi_{j,l}(i-1)$ is ${\mathcal F}_{t_{i-1}}$
measurable, and since for $l\neq j$ the sigma-field ${\mathcal
F}_{t_{i-1}}$ and the random variables
 $\beta_j(t_i)-\beta_j(t_{i-1}) $  and $\beta_l(t_i)-\beta_l(t_{i-1})
$ are independent, the process $(M_k,{\mathcal F}_{t_k}, 2\leq k
<2^n)$ is a discrete martingale. Furthermore, for the cases (a) $i<i'$ and $l'\neq
j'$, (b) $i'<i$ and $l\neq j$ or (c) $i=i'$ and $(\min(j,l),\max(j,l))\neq
(\min(j',l'),\max(j',l'))$,
 one has
\begin{align*} \EX \big[ \Phi_{j,l}(i-1)\, & \Phi_{j',l'}(i'-1)\, \big( \beta_l(t_i)-\beta_l(t_{i-1} )
\big)   \, \big( \beta_j(t_i)-\beta_j(t_{i-1} )\big)\\
& \times  \big( \beta_{l'}(t_{i'})-\beta_{l'}(t_{i'-1} )\big)   \,
\big( \beta_{j'}(t_{i'})-\beta_{j'}(t_{i'-1} )\big)\big]=0.
\end{align*}
Therefore, Schwarz's  and Doob's inequalities
 yield
\begin{align*}
\EX \Big(& \max_{2\leq k<2^n} |M_k|\Big)^2 \leq \EX\Big(  \max_{2\leq k<2^n} M_k^2\Big)
\leq 4 \EX \big( M_{2^n-1}^2\big) \\
& \leq 12 \sum_{2\leq i<2^n} \sum_{1\leq j\leq n} \sum_{l\neq j} \EX \big( \Phi_{j,l}(i-1)^2\big)\,
\EX \big( \big| \beta_l(t_i)-\beta_l(t_{i-1} )\big|^2\big)  \,
\EX \big( \big| \beta_j(t_i)-\beta_j(t_{i-1} )\big|^2\big).
\end{align*}
Furthermore, using \eqref{Ds-bnd}, \eqref{s-bnd} and \eqref{g-N-set2}  we deduce that for every $i,j,l$
\[ \EX\big( \Phi_{j,l}(i-1)^2\big) \leq q_l\, C_1(N)^2\, (K_0+K_1N^2)\, (2N)^2,\]
which implies
\begin{equation} \label{majoM1}
\EX \Big( \max_{2\leq k<2^n} |M_k|\Big) \leq C(N,T)\, n \, 2^{-{n}/{2}}.
\end{equation}
A similar easier computation shows that

\begin{align}\label{majoM2}
\EX & \Big( \max_{2\leq k<2^n} \sup_{t_k\leq t\leq t_{k+1}}\! \Big|\!
\sum_{1\leq j\leq n}\sum_{l\neq j}
 \Phi_{j,l}(k-1)\frac{2^n(t-t_k)}{T}
 \big( \beta_l(t_k)-\beta_l(t_{k-1} ) \big)   \big( \beta_j(t_k)-\beta_j(t_{k-1} )\Big|\Big)
\nonumber \\
& \leq \EX  \Big( \max_{2\leq k<2^n} \! \Big|\!
\sum_{1\leq j\leq n}\sum_{l\neq j}
 \Phi_{j,l}(k-1)
 \big( \beta_l(t_k)-\beta_l(t_{k-1} ) \big)   \big( \beta_j(t_k)-\beta_j(t_{k-1} )\Big|\Big)
\nonumber \\
& \leq \Big\{\EX  \Big( \max_{2\leq k<2^n} \! \Big|\!
\sum_{1\leq j\leq n}\sum_{l\neq j}
 \Phi_{j,l}(k-1)
 \big( \beta_l(t_k)-\beta_l(t_{k-1} ) \big)
   \big( \beta_j(t_k)-\beta_j(t_{k-1} )\Big|^2\Big)\Big\}^{1/2}
\nonumber \\
& \leq \Big\{  \sum_{2\leq k<2^n}\EX  \Big|\!
\sum_{1\leq j\leq n}\sum_{l\neq j}
 \Phi_{j,l}(k-1)
 \big( \beta_l(t_k)-\beta_l(t_{k-1} ) \big)
   \big( \beta_j(t_k)-\beta_j(t_{k-1} )\Big|^2\Big)\Big\}^{1/2}
\nonumber \\
& \leq \!\Big\{ \sum_{2\leq k<2^n} \!  \sum_{1\leq j\leq n} \sum_{l\neq j}
\EX \big( \Phi_{j,l}(k-1)^2\big)
\EX \big( \big| \beta_l(t_k)-\beta_l(t_{k-1} )\big|^2\big)
\EX \big( \big| \beta_j(t_k)-\beta_j(t_{k-1} )\big|^2\big)
\Big\}^{\frac{1}{2}}
\nonumber \\
&\leq  C(N,T)\, n \, 2^{-n/2}.
\end{align}
  Furthermore,
\begin{eqnarray*}
\lefteqn{
\EX \left(\sup_{t\in [0,T]} |S^{(2)}_n(t)|\right)
\le
\EX \left(\sup_{k}\sup_{t\in [t_k,t_{k+1}]} |S^{(2)}_n(t)|\right)
}  \\ &\le &
\EX \sup_{k\ge 3}\Big|  \sum_{2\leq i<k }  \int_{t_i}^{t_{i+1}}
1_{\{s_n\le\tau_n\}}1_{G^n_N(s_n)}
(V_n^{(2)} (s), u^n(s_n)-u(s_n)) ds\Big|
\\ &&
+ \EX \sup_{k\ge 2}\Big[ \sup_{t\in [t_k,t_{k+1}]}\Big|  \int_{t_k}^{t}
1_{\{s_n\le\tau_n\}}1_{G^n_N(s_n)}
(V_n^{(2)} (s), u^n(s_n)-u(s_n)) ds\Big|\Big].
\end{eqnarray*}
This inequality,  \eqref{majoM1} and \eqref{majoM2} immediately yield
\begin{equation}\label{s6-s12}
\EX\Big(\sup_{t\in [0,T]} |S_n^{(2)}| \Big) \leq  C(N,T)  \, n \, 2^{-\frac{n}{2}}.
\end{equation}

%\subsubsection
\noindent{\it Bound of $S^{(3)}_n$. } \quad
The argument is similar to the previous  one, based on a different discrete martingale.
For $i=1, \cdots, 2^n-1$, $j=1, \cdots,n$, set
\[ \Phi_j(i)= T 2^{-n}\, \Big( D\ts_j(u^n(t_i)) \s_j(u^n(t_i)) \, ,\, u^n(t_i)-u(t_i)\big)
\, 1_{\{t_i \leq \tau_n\}}\,
1_{G_N^n(t_i)}.\]
Then $\Phi_j(i-1)$ is ${\mathcal F}_{t_{i-1}}$-measurable and independent of the centered random variable
$Y_{ij}=2^n\, T^{-1} \big| \beta_j(t_i) - \beta_j(t_{i-1})\big|^2 -1$. Furthermore,  for $(i,j)\neq (i',j')$ one has
\[ \EX\big( \Phi_j(i-1)\, Y_{ij}\, \Phi_{j'}(i'-1)\, Y_{i'j'}\big)=0. \]
Using \eqref{Ds-bnd}, \eqref{s-bnd} and \eqref{g-N-set2}, we deduce   that for all $i,j$,
$\EX\big(|\Phi_j(i-1)|^2\big) \leq C_{N,T}\, 2^{-2n}$.
For $k=2, \cdots, 2^n$, set
\[N_k=\sum_{2\leq i\leq k} \sum_{1\leq j\leq n} \Phi_j(i-1)\, Y_{ij}.\]
The process $(N_k, {\mathcal F}_{t_k})$ is a discrete martingale;
 thus Schwarz's and Doob's  inequality  yield
\begin{align}\label{majoN1}
\EX  \Big( \max_{2\leq k\leq 2^n}   |N_k|\Big)&  \leq 2\big\{ \EX \big( |N_{2^n}|^2\big) \big\}^{\frac{1}{2}}
\nonumber \\
& \leq 2  \Big\{ \sum_{2\leq i\leq 2^n} \sum_{1\leq j\leq n} \EX\big( \Phi_j(i-1)^2\big)
\EX \big( |Y_{ij}|^2\big) \Big\}^{\frac{1}{2}}
\leq C_{T,N}\, n^{\frac{1}{2}}\, 2^{-\frac{n}{2}}.
\end{align}
Finally, a similar argument shows that
\begin{align}\label{majoN2}
 \EX &\Big(  \max_{1\leq k < 2^n}\sup_{t_k\leq t\leq t_{k+1}}\Big| 2^n T^{-1} (t-t_k)\,
 \sum_{1\leq j\leq n}
\Phi_j(k-1) Y_{kj}\Big|\Big) \nonumber \\
&\leq \Big( \EX  \sum_{1\leq k<  2^n}\Big|   \sum_{1\leq j\leq n}
\Phi_j(k-1)  Y_{kj}    \Big|^2\Big)^{\frac{1}{2}}\nonumber  \\
& \leq  \Big(  \sum_{1\leq k<  2^n} \!\! \EX \sum_{1\leq j\leq n} \big|
\Phi_j(k-1) Y_{kj}\big|^2\Big)^{\frac{1}{2}}
 \leq
 C(N,T)  n^{\frac{1}{2}} 2^{-\frac{n}{2}}.
\end{align}
The inequalities \eqref{majoN1} and \eqref{majoN2} imply that
\begin{equation}\label{s6-s13}
\EX\Big(\sup_{t\in [0,T]} |S_n^{(3)}| \Big) \leq  C(N,T) \, n^{\frac{1}{2}} \, 2^{-\frac{n}{2}}.
\end{equation}
Using \eqref{t-s6-2} and
collecting the upper estimates in \eqref{t-s6-2},
\eqref{s6-s11}, \eqref{s6-s12}
and \eqref{s6-s13}, we conclude  the proof of Lemma \ref{S6}.
\end{proof}

\subsubsection{\bf Bound for    $S_n(t,7)$.}
\begin{lemma}\label{S7}
Let the assumptions of Theorem \ref{th:WZ-appr} be satisfied and $S_n(t,7)$ be defined by
\eqref{sn7-def}. There exists a constant $C(N,T)$
such that
\begin{equation}\label{t-s7}
\EX\Big(\sup_{t\in [0,T]} |S_n(t,7)|\Big) \leq C(N,T)\,  n^2\, 2^{-\frac{n}{2}}.
\end{equation}
\end{lemma}
\begin{proof}
For $s\in [0,T]$, $j=1, \cdots, n$, set
\begin{align*}
&\widetilde{U}^n_j(s) = D \ts_j(u^n(s_n)) \Big[
\ts(u^n(s_n)) \Big(\int_{s_n}^s
 \dot{\widetilde{W}}^n(r) dr\Big) \Big]\,  \dot{\beta_j}^n(s),\\
&\widetilde{\Delta}_n(s)=
\Big(  \sum_{1\leq j\leq n}\widetilde{U}_j^n(s) -\hf\tro_n(u^n(s_n))\, , \,
   u^n(s_n)-u(s_n)\!\Big).
\end{align*}
We obviously have that
\[
\sum_{1\leq j\leq n}\widetilde{U}_j^n(s) -\hf\tro_n(u^n(s_n))=
\sum_{1\leq i\leq 3}  \widetilde{V}_n^{(i)}(s),
\]
where
\begin{align*}
&  \widetilde{V}_n^{(1)}(s)= \sum_{1\leq j\leq n} D \ts_j(u^n(s_n))
\ts(u^n(s_n))
\big[ W_n(s_n) - W_n((s_n-T2^{-n})\vee 0)\big]\,  \dot{\beta_j}^n(s),
\\
&  \widetilde{V}_n^{(2)}(s)=\sum_{1\leq j\leq n}  \sum_{l\neq j} D \ts_j(u^n(s_n))
\ts_l(u^n(s_n))
% \\  & {} \qquad {} \qquad \times
 (s-\underline{s}_n) \left[\beta_l(\underline{s}_n)-\beta_l(s_n) \right]\frac{2^{2n}}{T^2}
\left[\beta_j(\underline{s}_n)-\beta_j(s_n) \right],
\\
&  \widetilde{V}_n^{(3)}(s)= \sum_{1\leq j\leq n} \ts_j(u^n(s_n))
\ts_j(u^n(s_n)) \left[ \frac{2^{2n}}{T^2}(s-\underline{s}_n)
\left[\beta_j(\underline{s}_n)-\beta_j(s_n) \right]^2 -\hf\right].
\end{align*}
Using  \eqref{1-split}   we deduce
 the following decomposition of $S_n(t,7)$:
\begin{equation}\label{decom-s7}
    S_n(t,7)  =    \int_0^{t\wedge\tau_n}
   \tilde{\Delta}_n(s) ds
   = \sum_{1\leq i\leq 3} \tilde S^{(i)}_n(t)-\tilde S^{(4)}_n(t),
\end{equation}
where
\begin{eqnarray*} %\label{t-s7-2}
\tilde S^{(i)}_n(t)&= &\int_{t_2}^t  1_{\{s_n\le\tau_n\}}1_{G^n_N(s_n)}
\big( \widetilde{V}_n^{(i)} (s), u^n(s_n)-u(s_n)\big)  ds,   \quad i=1,2,3, \\
\tilde S^{(4)}_n(t) &= & \int_0^t  1_{\{s_n\le\tau_n<s\}} 1_{G^n_N(s_n)}
\widetilde{\Delta}_n(s) ds.
%\end{equation}
\end{eqnarray*}
We note that  $\tilde S^{(i)}_n(t)=0$ for $i=1,2,3$ and $t\le t_2$.

%\subsubsection
\noindent {\it Bound  for  $\tilde S^{(4)}_n$. }\quad
The proof is similar to that of the upper estimate of $S^{(4)}_n$.
Schwarz's inequality implies
\begin{eqnarray*}
\EX \Big(\sup_{t\in [0,T]} |\tilde S^{(4)}_n(t)|\Big)
 \le % \sqrt{2}
\Big\{ 2 T 2^{-n}\sum_{0\leq k<2^{n-1}} \EX \int_{t_k}^{t_{k+1}}
  1_{G^n_N(t_k)} \left|\widetilde{\Delta}_n(s)\right|^2 ds   \Big\}^{1/2}.
\end{eqnarray*}
The inequalities  \eqref{s-bnd}, \eqref{Ds-bnd}, \eqref{rn-1}, the definition \eqref{g-N-set2}
of the set $G^n_N(s)$ and Schwarz's  inequality yield for $ t_k \le s <t_{k+2}$:
\begin{eqnarray*}
 1_{G^n_N(t_k)}  \left|\widetilde{\Delta}_n(s)\right|^2& \le
&  {C}(N)   \Big(1+ n \Big| \int_{s_n}^s
 \dot{\widetilde{W}}^n(r)dr\Big|^2 \sum_{1\leq j\leq n}|\dot{\beta_j}^n(s)|^2 \Big).
% \Big),\quad t_k \le s <t_{k+1}.
 \end{eqnarray*}
Therefore,  Fubini's theorem and Schwarz's inequality imply
\begin{align*}
\EX & \Big(\sup_{t\in [0,T]} |\tilde S^{(4)}_n(t)|\Big)
\\
 & \le C_{N,T} 2^{-n/2}
\Big\{T+   n\sum_{1\leq j\leq n}\sum_{0\leq k<2^n} \EX \int_{t_{k-1}}^{t_{k+1}}
 \left| \int_{s_n}^s
 \dot{\widetilde{W}}^n(r)dr\right|_0^2 |\dot{\beta_j}^n(s)|^2 ds   \Big\}^{1/2}
 \nonumber \\
 & \le C_{N,T} 2^{-n/2}
\Big\{1+   {n} 2^{-n}\sum_{1\leq j\leq n}\sum_{0\leq k<2^n}\EX \Big[
 \int_{t_{k-1}\vee 0}^{t_{k+1}} \Big|\dot{\widetilde{W}}^n(r)\Big|_0^2dr
  \int_{t_k}^{t_{k+1}} |\dot{\beta_j}^n(s)|^2 ds\Big] \Big\}^{1/2}
 \nonumber \\
 & \le C_{N,T} 2^{-\frac{n}2}
\Big\{1+   n 2^{-2n}\sum_{1\leq j\leq n}\sum_{0\leq k<2^n} \Big[
 \int_{t_{k-1}\vee 0}^{t_{k+1}}\EX  \big|\dot{\widetilde{W}}^n(r)\big|_0^4dr\Big]^{\hf}
 \Big[ \int_{t_k}^{t_{k+1}}\EX |\dot{\beta_j}^n(s)|^4 ds\Big]^{\hf} \Big\}^{\hf}.
 \nonumber
\end{align*}
Since for every $s\in [0,T] $ we have
$\EX \big| \dot{\widetilde{W}}^n(s)\big|_0^4 \leq C(T)\, n^4\, 2^{2n}$ and
$\EX |\dot{\beta}_j^n(s)|^4 \leq C(T) 2^{2n}$, we deduce the existence of some constant $C(N,T)$ such that
\begin{equation}\label{t-s7-s4}
\EX \Big(\sup_{t\in [0,T]} |\tilde S^{(4)}_n(t)|\Big)
 \le C(N,T)\,  n^2 \,  2^{-n/2}.
 \end{equation}

%\subsubsection
\noindent {\it Bound for $\tilde S^{(1)}_n$. }\quad
For $j=1, ..., n$ let
\[ \varphi_j(s)= 1_{\{ \underline{s}_n \leq \tau_n\}} 1_{G_N^n(\underline{s}_n)}
 \Big( D\ts_j(u^n(\underline{s}_n))
\big[ \tilde{\s}(u^n(\underline{s}_n)) \big( W_n(\underline{s}_n) - W_n(s_n)\big)\big] \, ,\, u^n(\underline{s}_n)
- u(\underline{s}_n)\Big).\]
Then $ \varphi_j(s)$ is ${\mathcal F}_{\underline{s}_n}$ measurable and for $t\geq t_2$,
\[ \tilde S^{(1)}_n(t) = \sum_{1\leq j\leq n} \int_{t_1}^{t_n} \varphi_j(s) d\beta_j(s)
+  \sum_{1\leq j\leq n} \varphi_j(t-T2^{-n})  2^n T^{-1}
(t-\underline{t}_n) \big[ \beta_j(\underline{t}_n) -\beta_j(t_n)\big].
\]
 For fixed $j$ the process $\big( \varphi_j(t_k) (\beta_j(t_{k+1})
- \beta_j(t_k))\, ,\,  0\leq k<2^n\big)$  is a martingale
increments. Therefore,  the Burkholder and Schwarz inequalities,
\eqref{Ds-bnd}, \eqref{s-bnd} and \eqref{g-N-set2}, yield
\begin{align}\label{s7-s31}
\EX & \Big( \sup_{t\in [0,T]} |\tilde S^{(1)}_n(t)|\Big)
\nonumber \\ & \le C \Big\{
\EX \int_0^T \sum_{1\leq j\leq n} \varphi_j (s)^2 ds \Big\}^{\frac{1}{2}}
+ C  \EX\Big( \sum_{1\leq j\leq n} \max_{1\leq k<2^n}   |\varphi_j(t_k)| \,
|\beta_j(t_{k+1} ) -\beta_j(t_k)|\Big)
\nonumber  \\
&\leq C_{N,T} \Big\{n  \EX \int_0^T |W_n(\underline{s}_n) - W_n( s_n)|_0^2\,  ds
\Big\}^{\frac{1}{2}}
\nonumber \\
&\qquad
+ C_{N,T}  \EX\Big\{ n\sum_{1\leq k<2^n} \sum_{1\leq j\leq n} |W_n(t_k) - W_n(t_{k-1})|_0^2
 |\beta_j(t_{k+1}) - \beta_j(t_k)|^2\Big\}^{\frac{1}{2}}
\nonumber \\
&\leq C_{N,T} \sqrt{n}\Big[ 2^{-\frac{n}{2}} + \Big\{ \sum_{1\leq j\leq n} \sum_{1\leq k<2^n}
\EX |W_n(t_k)-W_n(t_{k-1 })|_0^2 \EX |\beta_j(t_{k+1})-\beta_j(t_k)|^2\Big\}^{\frac{1}{2}}
\nonumber \\
 & \le C(N,T)\,  n\,  2^{-{n}/{2}}.
\end{align}
%\subsubsection
\noindent {\it Bound for $ \tilde S^{(2)}_n$.}\quad
 For $i=1, \cdots, 2^n-1$, $j=1, \cdots, n$ and $l\neq j$ set
\[\tilde{\Phi}_{j,l}(i) = 2^{2n} T^{-2} 1_{\{ t_i\leq \tau_n\}} 1_{G_N^n(t_i)}
\Big( D\ts_j(u^n(t_i)) \ts_l(u^n(t_i)) \, , \, u^n(t_i)-u(t_i) \Big).\]
Then $\tilde{\Phi}_{j,l}(i)$ is ${\mathcal F}_{t_i}$ measurable and   since for $l\neq j$, ${\mathcal F}_{t_{i-1}}$,
$\beta_j(t_i)-\beta_j(t_{i-1})$ and $\beta_l(t_i)-\beta_l(t_{i-1})$ are independent,
if one sets
$Z_{j,l}(i)= \big( \beta_l(t_i)-\beta_l(t_{i-1}) \big) \big(  \beta_j(t_i)-\beta_j(t_{i-1})$,
the following process $(\tilde{M}_k, 2\leq k\leq 2^n)$  is a $({\mathcal F}_{t_k}) $
 centered martingale:
\begin{align*}  \tilde{M}_k&= \sum_{2\leq i\leq   k} \sum_{1\leq j\leq n} \sum_{l\neq j}
 \int_{t_i}^{t_{i+1}} \tilde{\Phi}_{j,l}(i-1) (s-t_i) Z_{j,l}(i) ds \\
% \big( \beta_l(t_i)-\beta_l(i-1) \big) \big(  \beta_j(t_i)-\beta_j(i-1) \big) ds \\
&=T^2 2^{-(1+2n)}  \sum_{2\leq i\leq   k} \sum_{1\leq j\leq n} \sum_{l\neq j}
\tilde{\Phi}_{j,l}(i-1)  Z_{j,l}(i).
%\big( \beta_l(t_i)-\beta_l(i-1) \big) \big(  \beta_j(t_i)-\beta_j(i-1) \big) .
\end{align*}
Furthermore, if $i<i'$ and $l'\neq j'$, or $i'<i$ and $l\neq j$, or $i=i'$ and
$\big(\min(j,l),\max(j,l)\big)\neq \big(\min(j',l'),\max(j',l')\big)$,  one has
$
\EX\big[ \tilde{\Phi}_{j,l}(i-1)
 Z_{j,l}(i) \tilde{\Phi}_{j',l'}(i'-1)  Z_{j',l'}(i')\big] =0$.
Hence Doob's, Schwarz's inequalities together with \eqref{Ds-bnd}, \eqref{s-bnd} and \eqref{g-N-set2} yield
\begin{align}\label{majoMt1}
& \EX  \Big( \max_{2\leq k<2^n} |\tilde{M}_k|\Big)^2 \leq \EX \Big(  \max_{2\leq k<2^n} |\tilde{M}_k|^2\Big)
\leq 4 \EX\big(\tilde M_{2^n-1}^2\big)
\nonumber \\
& \leq C_T 2^{-4n} \!\! \sum_{2\leq i<2^n} \sum_{1\leq j\leq n} \sum_{l\neq j}
 \EX\big( |\tilde{\Phi}_{j,l}(i-1)|^2\big)
%\nonumber \\
%& \qquad \times
  \EX  \big( | \beta_l(t_i)-\beta_l(t_{i-1})|^2\big)
 \EX  \big( | \beta_j(t_i)-\beta_j(t_{i-1})|^2\big)
\nonumber \\
& \leq  C(N,T)\, n\, 2^{-n}.
\end{align}
A   computation similar to that  performed in \eqref{majoM2} proves that
\begin{align} \label{majoMt2}
\EX & \Big( \max_{2\leq k<2^n} \sup_{t_k\leq t\leq t_{k+1}} \Big| \sum_{1\leq j\leq n} \sum_{l\neq j}
\int_{t_k}^{t}  \tilde{\Phi}_{j,l}(k-1) (s-t_k) Z_{j,l}(k) ds\Big| \Big)
 \\
&\leq  T^2 2^{-2n} \Big\{  \sum_{2\leq k<2^n}  \sum_{1\leq j\leq n} \sum_{l\neq j} 2^{-4n}
 \EX\big( |\tilde{\Phi}_{j,l}(k-1)|^2\big)
\nonumber \\
& \qquad\qquad \qquad\times
 \EX  \big( | \beta_l(t_k)-\beta_l(t_{k-1})|^2\big)
 \EX  \big( | \beta_j(t_k)-\beta_j(t_{k-1})|^2\big)\Big\}^{\frac{1}{2}} \leq  C(N,T)\,  n \, 2^{-{n}/{2}}.
 \nonumber
%\nonumber \\
%& \leq  C(N,T)\,  n \, 2^{-{n}/{2}}.
\end{align}
The inequalities \eqref{majoMt1} and \eqref{majoMt2} yield
\begin{equation} \label{s7-s32}
\EX\Big(\sup_{t\in [0,T]} |\tilde S^{(2)}_n(t)|\Big) \leq C(N,T) \, n \, 2^{-{n}/{2}}.
\end{equation}

%\subsubsection
\noindent {\it Bound for $\tilde S^{(3)}_n$.}\quad
Finally, for $i=1, \cdots, 2^n-1$ and $j=1, \ldots, n$, set
\begin{eqnarray*}
 \tilde{\Phi}_j(i)& =& 1_{\{ t_i\leq \tau_n\}} 1_{G_N^n(t_i)}
\Big( D\ts_j(u^n(t_i)) \ts_j(u^n(t_i)) \, ,\, u^n(t_i)-u(t_i)\Big), \\
Z_j(i)&=& \int_{t_i}^{t_{i+1}} \Big[ \frac{2^{2n}}{T^2} (s- t_i) \big( \beta_j(t_{i+1}) - \beta_j(t_i)\big)^2
- \frac{1}{2} \Big]\, ds.
\end{eqnarray*}
Then the random variables $Z_j(i)$ and $\tilde{\Phi}_j(i)$ are independent, $\EX (Z_j(i))=0$ and
$\EX (Z_j(i)^2) \leq
C_T 2^{-2n}$. Furthermore, for $(i,j)\neq (i',j')$, $\EX \big (\tilde{\Phi}_j(i) Z_j(i) \tilde{\Phi}_{j'}(i')
Z_{j'}(i') \big) =0$. The process defined for $k=1, \cdots, 2^n-1$
by $\tilde{N}_k = \sum_{1\leq i\leq k}\sum_{1\leq j\leq n} \tilde{\Phi}_j(i) Z_j(i)$
is a discrete $({\mathcal F}_{t_{k+1}})$ martingale. Doob's and Schwarz's inequalities, \eqref{Ds-bnd},
\eqref{s-bnd}     and \eqref{g-N-set2} imply that
\begin{align}\label{s7-s33}
\EX & \Big( \sup_{t\in [0,T]} |\tilde S^{(3)}_n(t)|\Big) \leq \EX \Big( \max_{1\leq k<2^n}
 \big|\tilde{N}_k\big| \Big)
\nonumber \\
&\quad + \EX \Big( \max_{1\leq k<2^n} \sup_{t_k\leq t\leq t_{k+1}}\Big|  \sum_{1\leq j\leq n}
\tilde{\Phi}_j(k)
 \int_{t_k}^{t} \Big[ \frac{2^{2n}}{T^2} (s- t_k) \big( \beta_j(t_{k+1}) - \beta_j(t_k)\big)^2
- \frac{1}{2} \Big]\, ds \Big| \Big)
\nonumber \\
&\leq C \EX\big(\big|\tilde{N}_{2^n-1}\big|^2\big)^{\frac{1}{2}}
+ \Big( 2^n n \max_{1\leq k<2^n} \max_{1\leq j\leq n} \EX \big(  \tilde{\Phi}_j(k)^2 \big)
\EX\big(Z_j(k)^2\big) \Big)^{\frac{1}{2}}
\nonumber \\
& \leq C \Big( n 2^n \max_{2\leq k<2^n} \max_{1\leq j\leq n}  \EX \big(  \tilde{\Phi}_j(k)^2 \big)
\EX\big(Z_j(k)^2\big) \Big)^{\frac{1}{2}}
\leq C(N,T) \, n^{\frac{1}{2}} \, 2^{-\frac{n}{2}}.
\end{align}
The relations in \eqref{decom-s7} -- \eqref{s7-s33} conclude the proof of Lemma \ref{S7}.
%\begin{equation}\label{t-s7}
%\EX\Big( \sup_{t\in [0,T]} |S_n(t,7)|\Big) \leq C_{N,T}\, n^2\, 2^{-\frac{n}{2}}.
%\end{equation}
\end{proof}

Now using  Proposition~\ref{pr:t-s67}, Lemmas \ref{S6} and \ref{S7},
we obtain
 \eqref{WZ-cnvrg-3};
this completes  the proof of Theorem \ref{th:WZ-appr}.

\section{Appendix} \label{appendix}
We consider  some additional properties of the solution to \eqref{main-eq}. The aim
of this section is to introduce some more properties on the coefficients
$\sigma$, $\tilde{\sigma}$, $G$ and $R$ which will ensure that the property \eqref{crit-bound}
holds. Let $\bar{C}$ denote a constant such that
\begin{equation} \label{comparnorm}
|u|\leq \bar{C} \|u\|\, , \forall  u\in V.
\end{equation}
\subsection{Exponential moments}
\begin{prop}\label{pr:exp-m}
Let $h(t)\in {S}_M$ be deterministic, suppose that the  operators $G$ and $\s+\ts$
are uniformly bounded and that the linear growth of $R$ is small enough,  i.e.,  there exist
positive constants $K_0$, $R_0$ and $  \tilde{R}_0$ such that:
\begin{equation}\label{CApp}
|G(u)|^2_{L (H_0, H)} \leq K_0,\;
|(\s + \ts)(u)|^2_{L_Q}  \leq K_0,\;  |R(u)|\leq R_0 + \tilde{R}_0|u|\;  \mbox{\rm with }
\tilde{R}_0<\bar{C}^{-2}
\end{equation}
for every $u\in H$. Let $u(t)$ be the  solution to \eqref{main-eq} such that
the initial condition has some exponential moment, i.e., $\EX\exp(\alpha_0 |\xi|^2)<\infty$
for some $\alpha_0>0$. Then there exist constants $\alpha_1 \in ]0,\alpha_0]$, $\beta(\alpha)>0$
and $c_i>0, i=1,2$ such that for $0<\alpha<\alpha_1$ and $t\in [0,T]$:
\begin{equation}\label{exp-1}
\EX\exp\left(  \a |u(t)|^2+\b(\a)\int_0^t\|u(s)\|^2 ds\right)\le
 e^{c_1t + c_2M}\EX\exp(\alpha |\xi|^2).
\end{equation}
 The same estimate holds
for Galerkin approximations $u_n$ of $u$ with constants $c_1,c_2$ which do not depend on $n$.
\end{prop}
\begin{proof}
Let $\sigma_0=\sigma + \tilde{\sigma}$,
$\Phi_0(t)=\exp\left( \a |u(t)|^2\right)$ and $\Phi(t)=
 \Phi_0(t)\exp\left( \b\int_0^t\|u(s)\|^2 ds\right)$. By It\^o's formula
we have for every $t\in [0,T]$:
\[
d \Phi(t)=\left[ \b\|u(t)\|^2 \Phi_0(t) dt+ d\Phi_0(t)\right]
\exp\Big( \b\int_0^t\|u(s)\|^2 ds \Big)
\]
and
\[
d \Phi_0(t)= \a\Phi_0(t) \left[ 2(u(t),du(t)) + |\s_0(u(t))|^2_{L_Q}  dt +
2\a |\s_0^*(u(t)) u(t)|^2_{H_0} dt\right].
\]
Therefore, if $I(t)=2\alpha\int_0^t \Phi(s) \big( u(s)\, ,\, \s_0(u(s)) dW(s)\big)$,
\begin{eqnarray*}
d \Phi(t) &=  &\Phi(t) \Big[ -( 2\a-\b) \|u(t)\|^2 +2\a\big(-R(u(t))+G(u(t))h(t),u(t))
+ \a|\s_0(u(t))|^2_{L_Q}
\\ & & \qquad\quad  +
2\a^2 |\s_0^*(u(t)) u(t)|^2_{H_0}\Big]dt + I(t).
\end{eqnarray*}
For any integer $n\geq 1$, let $\tau_n=\inf\{t :
 \sup_{0\leq s\leq t} |u(s)|^2 + \int_0^t \|u(s)\|^2 ds \geq n\}\wedge T$. Then
we have $\EX(I(t\wedge \tau_n)=0$ for $t\in [0,T]$.
Since $|u(t)|\leq \bar{C} \|u(t)\|$, if  $\tilde{R}_0$ from \eqref{CApp}
is such that  $\tilde{R}_0<\bar{C}^{-2}$, for $\alpha_1\leq \alpha_0$ small
enough and $0<\alpha<\alpha_0$, we have $1-(\tilde{R}_0 +2^{-1}\alpha K_0)\bar{C}^2>0$. For $0<\beta<
\beta(\alpha)$ with $\beta(\alpha)$ small enough, and for $\epsilon$ small enough,
Fubini's theorem implies:
%\[ d\EX \Phi(t\wedge\tau_n) \leq \EX \Phi(t\wedge\tau_n)
%\Big[ {R_0}{\epsilon}^{-1} + {K_0}{\epsilon}^{-1} |h(t\wedge\tau_n)|_0^2 +
%\alpha K_0\Big]\, dt.
%\]
\begin{align*}
 \sup_{0\leq s\leq t} \EX \Phi(s\wedge \tau_n)&\leq \exp(\alpha |\xi|^2)  + \EX \int_0^{t\wedge \tau_n} \Phi(s) \big[
R_0 \epsilon^{-1} + K_0  \epsilon^{-1} |h(s)|_0^2 + \alpha K_0\big] \, ds\\
&\leq \exp(\alpha |\xi|^2) +  \int_0^t  \EX \Phi(s\wedge \tau_n) \big[
R_0 \epsilon^{-1} + K_0  \epsilon^{-1} |h(s)|_0^2 + \alpha K_0\big] \, ds.
\end{align*}
Since $ \Phi(.\wedge\tau_n)$ is bounded, Gronwall's lemma implies that
 there exist constants $c_1, \, c_2$ depending on $K_0$, $R_0$ and $\alpha$ such that for every $t\in [0,T]$,
\[
\sup_n \sup_{0\leq t\leq T} \EX  \Phi(t\wedge \tau_n) \leq \exp(\alpha |\xi|^2) \, \exp(c_1 T+c_2M).
\]
Using \eqref{eq3.1} and the monotone convergence theorem, we conclude the proof by letting $n\to \infty$.
\end{proof}

\subsection{Properties in $\HH$.}
Now we are in position to state the conditions which
guarantee the validity of conditions (i) and (ii) in Theorems~\ref{th:WZ-appr}  and \ref{th:support}.
\smallskip\par
\noindent {\bf Condition (BS+)} Let condition  {\bf (B)} hold with $\HH= Dom(A^{1/4})$
and suppose that there exists a constant $K>0$ such that for $u\in \HH$:
\begin{equation} \label{s-sm1-4}
|A^{\frac{1}{4}}
\sigma(t,u)|^2_{L_Q(H_0,H)}+ |A^{\frac{1}{4}}
\ts(t,u)|^2_{L_Q(H_0,H)} \leq K(1+\|u\|^2_{\mathcal  H}).
\end{equation}

\noindent {\bf Condition (GR1)} There exist constants $\bar{K}_0$ and $\bar{R}_0$ such that
for every  $u\in \HH$:
\begin{equation}\label{growth1}
|A^{\frac{1}{4}} G(u)|^2_{L(H_0,H)} \leq \bar{K}_0(1+\|u\|^2_{\HH}) \, ,\;
 |A^{\frac{1}{4}} R(u)|\leq  \bar{R}_0(1+\|u\|_{\HH})\, .
%,
%\;  \forall u\in \HH.
\end{equation}

\begin{prop}\label{pr:add-est}
Assume that
 conditions {\bf (BS+)}, {\bf (GR1)},  as well as  \eqref{s-bnd} and \eqref{s-lip} from
condition {\bf (S)} are satisfied.
Let the hypotheses of Proposition \ref{pr:exp-m} be in force
and let $u$ be the solution to \eqref{main-eq}.
Assume in addition that $\EX\|\xi\|^2_\HH<\infty$.
Then there exist $q>0$ and $q_*>0$ such that
\begin{equation} \label{crit-bound-0}
\EX \Big( {\rm ess}\sup_{[0,T]}\| u(t)\|^q_\HH \Big) +
\EX \Big( \Big| \int_0^T | A^{3/4} u(\tau)|^2d \tau\Big|^{q_*}\Big)<\infty .
\end{equation}
%and
%\begin{equation} \label{crit-bound-0i}
%\EX \left( \int_0^T | A^{3/4} u(\tau)|^2d \tau\right)^{q_*}<\infty.
%\end{equation}
\end{prop}
\begin{proof} We consider the  Galerkin approximations $u_n$ and, to ease notations, we skip the  index $n$.
Let $\s_0=\s+\ts$ and  for  $t \in [0,T]$ set
\[ I(t):=\sup_{0\leq s\leq t} 2\left|\int_0^{t} \big(A^{\frac{1}{4}}
 \sigma_0(u(r))  dW(r)\, ,\,  A^{\frac{1}{4}} u(r)\big)\right| .
\]
Using   It\^o's formula for $\|u(t)\|^2_{\HH}=|A^{1/4}u(t)|^2$
and usual upper estimates, we  deduce that
\begin{align*}
 &\sup_{s\leq t}\|u(s)\|_{\mathcal H}^2  +
2  \int_0^{t} |A^{\frac{3}{4}} u(s)|^2\,  ds
 \leq \|\xi\|_{\mathcal H}^2 + 2  \int_0^{t} |\langle B(u(s),u(s)) ,
  A^{\frac{1}{2}} u(s)\rangle|\, ds
\\ &
+  I(t) +
\int_0^{t}\!\! 4K (1+\|u(s)\|_{\mathcal H}^2 ) ds
+ 2\int_0^{t}\!\!|(-R(u(s))+ G(u(s))h(s), A^{1/2}u(s))|ds.
\end{align*}
The inequality \eqref{preB}  and condition {\bf (GR1)} imply
\begin{align*}
|\langle B(u,u),   A^{\frac{1}{2}} u\rangle|
\le C_0\|u\|_\HH\|u\| |A^{3/4}u| &\le |A^{3/4}u|^2+ C_0^2\, 2^{-2}\,  \|u\|^2_\HH\|u\|^2\, ,\\
%\]
%The inequalities \eqref{G-bnd-lip} and \eqref{R-bnd-lip} from condition {\bf (GR)}  yield
%\[
|(-R(u)+ G(u)h, A^{1/2}u)|&\le c_0(1+|h|_0)(1+\|u\|_{\mathcal H}^2),
\end{align*}
where $c_0$ depends on $\bar{K}_0$ and $\bar{R}_0$.
Hence, for
$ X(t)=  \sup\{ \|u(s)\|_{\mathcal H}^2 \, : \, 0\leq s \leq t \}$,
we deduce
\begin{equation} \label{crit-bound-0X}
 X(t) +  \int_0^{t} |A^{\frac{3}{4}} u(s)|^2\,  ds\leq  \|\xi\|_{\mathcal H}^2 + I(t)+ c_1 +c_2\int_0^t
\left[1+ |h(s)|_0+ \| u(s)\|^2\right]\, X(s) \,  ds,
\end{equation}
where
the constant  $c_1$ depends on $K, \bar{K}_0,\bar{R}_0,T,M$
and $c_2$ depends on $\bar{K}_0$ and $\bar{R}_0$.
  Gronwall's lemma yields
\begin{eqnarray*}
X(t)& \leq & \left[ c_1 +\|\xi\|_{\mathcal H}^2 + I(t)\right]
\exp\Big( c_2\int_0^t
\left[1+ |h(s)|_0+ \| u(s)\|^2\right] ds\Big).
\end{eqnarray*}
This implies that for $\delta >0$:
\begin{eqnarray*}
\EX |X(t)|^\delta & \leq & C(M,T)\;
 \Big[ \EX \left(c_1+ \|\xi\|_{\mathcal H}^2+ I(t)\right)^{2\delta}\Big]^{1/2}
\Big[\EX \exp\Big( 2 c_2\delta \int_0^t \| u(s)\|^2ds\Big) \Big]^{1/2}.
\end{eqnarray*}
Thus  Proposition ~\ref{pr:exp-m} implies that for $\delta$ small enough
we have:
\[
\EX |X(t)|^\delta
 \leq   C(M,T) \EX\big[\exp(2c_2 \delta |\xi|^2) )\big]^{\frac{1}{2}} \,
%  \Big[\EX \left| c +\|\xi\|_{\mathcal H}^2+ I(t)\right|^{2\delta}\Big]^{1/2}
 \left[ 1 +\EX\|\xi\|_{\mathcal H}^2 + \EX I(t)\right]^{\frac{1}{2}}.
\]
%%%%%%%%%%%%
% Reference for the BDG inequality with constant 2\sqrt{2} instead  of 3 in our context
% Revuz-Yor Exercise 4.17  page 160
%D. Revuz and M. Yor, Continuous Martingales
%and Brownian Motion, 3rd edition, Springer, 1999
%%%%%%%%%%%%%%%%%%%%
The Burkholder-Davies-Gundy inequality,
 relations \eqref{s-sm1-4} and \eqref{eq3.1} yield
\begin{eqnarray*}
\EX I(t) &\leq   &
 6 \,
\EX \Big\{ \int_0^{t}
|A^{1/4}u(r)|^{2} \; |A^{1/4} [ \s+\ts](u(r))|^2_{L_Q}\;
 dr \Big\}^\frac12    \nonumber  \\
 &\leq \; & 6 \,
\EX \Big\{4 K \int_0^{t}
  \|u(r)\|_\HH^{2} \left(1+ \|u(r)\|_\HH^{2}\right)\;
 dr \Big\}^\frac12 \le  c_4(T,K,C)\, .
\end{eqnarray*}
Thus there exists constants $q>0$ and $c:=c(K,T,M,C)$ such that
\begin{equation} \label{majoN}
 \sup_{n\geq 1} \, \EX\Big(  \sup_{0\leq s\leq T } \|u_n(s)\|_{\mathcal H}^q \Big)= c <+\infty
\end{equation}
for the Galerkin approximations $u_n$. As $n\to +\infty$, after limit transition we
deduce that the first term in the left hand-side of \eqref{crit-bound-0} is finite.
\par
To  prove that the second term is finite as well,   note that \eqref{crit-bound-0X}
implies that for every $n$:
\[
 \int_0^{t} |A^{\frac{3}{4}} u_n(s)|^2\,  ds\leq  C+\|\xi\|_{\mathcal H}^2 + I(t)
  +c_2 {\rm ess}\sup_{0\leq s\leq T }\|u_n(s)\|_{\mathcal H}^2 \int_0^t
\left[1+ |h(s)|^2_0+ \| u_n(s)\|^2\right] \,  ds.
\]
Thus we can use \eqref{majoN} and  complete  the  proof of   \eqref{crit-bound-0}
by a similar argument.
\end{proof}

We prove that the process $u$ solving \eqref{main-eq} belongs to  $ {\mathcal C}([0,T],\HH)$ a.s.
\begin{prop}
Let the  conditions of Proposition~\ref{pr:add-est} be satisfied and let $u$ be the solution
to \eqref{main-eq}. Then
the process $u$ belongs to  $ {\mathcal C}([0,T],\HH)$ a.s.
\end{prop}
\begin{proof}
Let $\s_0=\s+\ts$; then for fixed $\delta >0$, we have
 $e^{-\delta A} u \in C([0,T],{\mathcal H})$.
 Indeed, \eqref{s-sm1-4} and \eqref{eq3.1}  imply that
 $
 \EX \int_0^T |A^{\frac{1}{4}} e^{-\delta A}  \sigma_0(u(s))|^2_{L_Q}\, ds
 <+\infty,
 $
 so that $\int_0^{t}  e^{-\delta A}  \sigma_0(u(s))
 \, dW(s) \in {\mathcal C}([0,T],{\mathcal H})$.
Since for $\delta >0$ the operator $e^{-\delta A}$ maps $H$ to $V$ and $V^\prime$
to ${\mathcal H}$,
we deduce that almost surely  the maps
 $A^{\frac{1}{4}} \, e^{-\delta A}\int_0^{t} [B((u(s)) + R(u(s))] \, ds $
and   $A^{\frac{1}{4}} \, e^{-\delta A}\int_0^{t} G((u(s)) \, h(s)\, ds $
belong to ${\mathcal C}([0,T],{\mathcal H})$.
 Therefore it  is sufficient to prove that
\begin{equation}\label{conv-cont}
\lim_{\delta\to0} \EX\Big( \sup_{0\leq t\leq T}
 \|u(t)-e^{-\de A}u(t)\|_{\mathcal H}^{2p}\Big)=0
\end{equation}
for some $p>0$.
Let $T_\de= Id-e^{-\de A}$ and apply It\^o's   formula to
$\|T_\de u(t)\|_{\mathcal H}^2$. This yields
\begin{align} \label{ito-g-de}
 \|T_\de u(t)\|_{\mathcal H}^2  =&
 \|T_\de \xi\|_{\mathcal H}^2  -2\int_0^{t}\!\! |A^{\frac{3}{4}}
 u(s)|^2 ds   +   2  I(t) +    \int_0^{t} \!\! |A^{\frac{1}{4}}
 T_\de \sigma_0(u(s))|_{L_Q}^2 ds  \nonumber \\
& \; -2\int_0^{t}\!\! \big\langle  B(u(s))+   R (u(s)) -
  G(u(s)) h(s),  A^{\frac{1}{2}} T^2_\de u(s)\big\rangle \, ds ,
\end{align}
where $I(t)=  \int_0^{t} \big( A^{\frac{1}{4}} T_\de \sigma_0(u(s))  dW(s),
A^{\frac{1}{4}} T_\de u(s)\big)$.
The Burkholder-Davies-Gundy and Schwarz inequalities together with
  \eqref{s-sm1-4} imply that  for any $p>0$:
\begin{eqnarray*}
\EX\sup_{0\leq t\leq T} |I(t)|^p &\le&
C_p \EX\left( \int_0^{T}\!\! \|T_\de u(s)\|_{\mathcal H}^2
 |A^{\frac{1}{4}} T_\delta \, \sigma_0(u(s))|^2_{L_Q} ds\right)^{p/2}  \\
&\leq  &  \frac{1}{2}\;  \EX\sup_{0\leq t\leq T}  \|T_\de u(t)\|_{\mathcal H}^{2p} +
\frac{C_p^2}{2} \EX \left(
\int_0^{T} \!\! |A^{\frac{1}{4}} T_\de \s_0( u(s))|_{L_Q}^2\, ds\right)^{p}.
 \end{eqnarray*}
Hence   \eqref{ito-g-de} yields for $0<p<1$ the existence of a constant $c_p$ such that
\begin{align*}
\EX\sup_{0\le t\le T } \| & T_\de u(t )\|_{\mathcal H} ^{2p}
\le  c_p\Big[ \|T_\de \xi\|_{\mathcal H}^{2p}+
  \EX \Big| \int_0^{T}\!\! | A^{\frac{1}{4}} T_\de \sigma_0 (u(s))|_{L_Q}^2 ds
 \Big|^{p}   \\
& \; +\,  \EX \Big(\int_0^{T}\!\!  \left|\big\langle  B(u(s))+  R (u(s)) -
  G(u(s)) h(s),  A^{\frac{1}{2}} T^2_\de u(s)\big\rangle \right|  \, ds\Big)^{p}
\Big].
\end{align*}
Since for every $u\in {\mathcal H}$, $\|T_\de u \|_{\mathcal H}  \to 0$   as $\delta\to 0$
and $\sup_{\delta>0} |T_\delta|_{L({\mathcal H},{\mathcal H})}\leq 1$,
 we deduce that if $\{\varphi_k\}$
denotes an orthonormal basis in $H$, then
$ | A^{\frac{1}{4}} T_\de \sigma_0 (u(s))Q^{1/2} \varphi_k|^2 \to 0$
for every $k$  and almost every $(\omega,s)\in \Omega\times [0,T]$.
Since $\sup_{\delta >0} \|e^{-\delta A}\|_{L({\mathcal H})}<+\infty$, \eqref{s-sm1-4}
implies
\[
{\displaystyle \sup_{\delta>0 }
|A^{\frac{1}{4}} T_\delta \sigma_0 (u)|_{L_Q}^2
%  \sum_k  \sup_{\delta >0}  |A^{\frac{1}{4}} T_\de \sigma_0 (u)Q^{1/2} \varphi_k|^2
%\leq C \|u\|_{{\mathcal H}}^2\in L^1(\Omega\times [0,T])},
\leq C (1+ \|u\|_{{\mathcal H}}^2) \in L^1(\Omega\times [0,T])} .
\]
Therefore, the Lebesgue  dominated convergence theorem yields
$$
\EX \int_0^{T}\!\! | A^{\frac{1}{4}} T_\de \sigma_0 (u(s))|_{L_Q}^2
ds \to 0 \quad\mbox{as }~~\delta \to 0.$$  Furthermore, using
\eqref{preB} we deduce
\begin{align*}
 \int_0^{T}  \Big|\big\langle  B(u(s)),
 A^{\frac{1}{2}} T^2_\de & u(s)\big\rangle \Big|  \, ds
 \le
C  \int_0^{T}  \|u(s)\|_\HH  \|u(s)\| | A^{\frac{3}{4}} T^2_\de u(s)|  \, ds \\
& \le
C\,  {\rm ess}\sup_{[0,T]} \|u(s)\|_\HH  \Big[
 \int_0^{T}    \|u(s)\|^2  \, ds\Big]^{1/2}
\Big[
 \int_0^{T} | A^{\frac{3}{4}} T^2_\de u(s)|^2  \, ds\Big]^{1/2}.
\end{align*}
Thus, using Proposition \ref{pr:add-est}  for  $p>0$ small enough and   H\"older's inequality, we obtain
\begin{align*}
 \EX \Big[\int_0^{T}  \left|\big\langle  B(u(s)),
 A^{\frac{1}{2}} T^2_\de u(s)\big\rangle \right|  \, ds\Big]^{p}
 \le
C \left[ \EX \,  {\rm ess}\sup_{[0,T]} \|u(s)\|^{2p}_\HH \right]^{1/2}
 \Big[
\EX  \Big(\int_0^{T} \!\!   \|u(s)\|^2  \, ds\Big)^{2p} \Big]^{1/4}\\
\times
\Big[ \EX  \Big(
 \int_0^{T} \!\!| A^{\frac{3}{4}} T^2_\de u(s)|^2  \, ds\Big)^{2p}\Big]^{1/4}
 %\\
 \le
 C \Big[ \EX  \Big(
 \int_0^{T} \!\! | A^{\frac{3}{4}} T^2_\de u(s)|^2  \, ds\Big)^{2p}\Big]^{1/4}.
\end{align*}
Given  $u\in Dom(A^{\frac{3}{4}})$ we have
$|A^{\frac{3}{4}} T_\de^2 u| \to 0$ as $\delta\to 0$
while $|A^{\frac{3}{4}} T_\de^2 u| \le 2 |A^{\frac{3}{4}} u|$.
 Hence  the dominated convergence theorem yields
$ % \[
\EX \Big[\int_0^{T} \! \Big|\big\langle  B(u(s)),
 A^{\frac{1}{2}} T^2_\de u(s)\big\rangle \Big|  \, ds\Big]^{p} \to 0 $ as $\delta \to 0$.
% \quad\mbox{\rm as} \quad \delta\to 0.
% \]
A similar argument can be applied to the term
%\begin{align*}
$ \int_0^{T} \! \left|\big\langle  R (u(s)) -
  G(u(s)) h(s),  A^{\frac{1}{2}} T^2_\de u(s)\big\rangle \right|  \, ds $.
%\end{align*}
Thus we obtain that \eqref{conv-cont} holds with $p>0$ small enough.
\end{proof}

\subsection{Examples of models}
In Remark \ref{re:appl} we have already shown that 
Theorems~\ref{th:WZ-appr} and \ref{th:support}
can be applied  to  periodic stochastic 2D Navier-Stokes equations  
and also to some  shell models of turbulence. 
The corresponding arguments involve either 
the additional symmetry  of the bilinear 
operator $B$  (see \eqref{Ns-per}) or some 
additional  regularity 
provided by the  discrete structure of shell type models.
These properties  are not true for other 2D hydrodynamical problems 
which we have in mind (see Section 2.1 in \cite{ch-mi}).
 However  the properties stated in (\ref{CApp}) and also in Conditions 
{\bf (BS+)} and {\bf (GR1)} provide us with another set of sufficient hypotheses
on the operators in (\ref{main-eq}) which guarantee the requirements
(i) and (ii) concerning solutions in Theorem~\ref{th:WZ-appr}.
 They allow us  to cover   several important
cases which include: 
\begin{itemize}
\item  2D Navier-Stokes equations   with
 Dirichlet boundary conditions,
\item 2D Boussinesq model for the B\'enard convection,
\item 2D  MHD equations and 2D magnetic B\'enard problem in bounded domains.
\end{itemize}
For more details concerning the models mentioned in this section
we refer to \cite{ch-mi} and to the references therein.
In all these cases
 a direct analysis
based on results of   interpolation of intersections (\cite{Tribel78})
makes it possible to prove that  $Dom(A^{1/4})$
is embedded  into $L_4$ type spaces 
and thus  (due to the considerations in 
\cite{ch-mi}) the basic hypotheses 
in
Condition  {\bf (B)} holds with $\HH=Dom(A^{1/4})$.
Thus we can apply 
Theorems~\ref{th:WZ-appr} and \ref{th:support} 
assuming the additional properties (\ref{CApp}), (\ref{s-sm1-4}) and 
(\ref{growth1}) concerning $R$, $G$, $\sigma$ and $\ts$.
\medskip

\noindent {\bf Acknowledgments:} We would like to thank 
anonymous referees for pointing
out references of related works on the Wong-Zakai approximation 
of  infinite dimensional stochastic evolution equations,  and for valuable remarks.


\begin{thebibliography}{99}
\bibitem{AKS} {\sc  S. Aida, S. Kusuoka, D. Stroock},  
\textit{ On the support of Wiener functionals},
Asymptotic problems in probability theory: Wiener functionals and
asymptotics, in: K.D. Elworthy and N. Ikeda (Eds.), Pitman Research
Notes in Math. Series 284, Longman Scient. \& Tech. 1993, 3--34.




\bibitem{BMSS} {\sc V. Bally, A. Millet, M. Sanz-Sol\'e},
\textit{Approximation and support theorem in H\"older norm for parabolic
stochastic partial differential equations},    Annals of Probability
{\bf 23} (1995), 178--222.


\bibitem{BaDP} {\sc V. Barbu, G. Da Prato},
\textit{ Existence and ergodicity for the two-dimensional stochastic magneto-hydrodynamics
 equations},  Appl. Math. Optim. {\bf 56(2)}  (2007),  145--168.

\bibitem{brzecapinfland} \textsc{Z. Brze\'{z}niak, M. Capi\'{n}ski,
F. Flandoli,}  \textit{A Convergence result for stochastic partial
differential equations}, Stochastics \textbf{\ 24} (1988), 423-445.

\bibitem{brzefland} \textsc{Z. Brze\'{z}niak, F. Flandoli},
\textit{Almost sure approximation of Wong-Zakai type for stochastic
partial differential equations}. Stochastic Process. Appl. {\bf 55} (1995),
329--358.

\bibitem{CG94}
 {\sc M. Capinsky, D. Gatarek},
\textit{Stochastic equations in Hilbert space with application
 to Navier-Stokes equations in any
dimension},  J. Funct. Anal. {\bf 126} (1994) 26--35. 


\bibitem{CW-M} {\sc C. Cardon-Weber, A. Millet},
\textit{A support theorem for a generalized Burgers
equation},  Potential Analysis  {\bf 15}  (2001),  361--408.

\bibitem{ChVui2004}
{\sc I. Chueshov, P. Vuillermot,}
{\it Non-random invariant sets for some
systems of parabolic stochastic partial differential equations},
Stoch. Anal. Appl. {\bf 22} (2004), 1421--1486.


\bibitem{ch-mi} {\sc I. Chueshov, A. Millet},
\textit{Stochastic 2D hydrodynamical type systems:  Well-posedness and large
deviations},
    Appl. Math. Optim. {\bf 61} (2010),  379--420.

\bibitem{Constantin}   {\sc P. Constantin, C. Foias},
Navier-Stokes Equations, U. of Chicago Press, Chicago,
1988. 



\bibitem{PZ92} {\sc G. Da Prato, J. Zabczyk},  Stochastic Equations
in Infinite Dimensions, Cambridge Univ. Press, 1992.



\bibitem{DM}  {\sc J. Duan, A. Millet}, \textit{ Large deviations for the Boussinesq equations
  under random influences},  Stoch. Proc. and Appl. {\bf 119} (2009),  2052--2081.

\bibitem{Ferrario} {\sc B. Ferrario},
\textit{The B\'{e}nard problem with random perturbations: Dissipativity
and invariant measures}, Nonlin. Diff. Equations and
Appl. {\bf  4} (1997), 101--121.

\bibitem{FG95}
 {\sc F. Flandoli, D. Gatarek},
\textit{Martingale and stationary solutions for stochastic Navier-Stokes equations},
Probab. Theory
Related Fields {\bf 102} (1995),  367--391. 


\bibitem{GreSch95}
{\sc W. Grecksch, B. Schmalfuss,}
\textit{Approximation of the stochastic Navier--Stokes
equation}. Comp. Appl. Math. {\bf 15} (1996), 227-239.

\bibitem{gyongy} \textsc{I. Gy\"{o}ngy}, \textit{\ On the Approximation of
Stochastic Partial Differential Equations}, Part I, Stochastics \textbf{25}
(1988), 59-85; Part II, Stochastics \textbf{26}
(1989) 129-164.


\bibitem{gyongy-1} {\sc I. Gy\"{o}ngy,}
\textit{\ The stability of stochastic partial differential
equations and applications},
 Part I, Stochastics and Stochastic Reports {\bf 27} (1989), 129-150; 
Part II,  Stochastics and Stochastic Reports {\bf 27} (1989), 189-233.

\bibitem{gyongy-2} \textsc{I. Gy\"{o}ngy, A. Shmatkov,}
{\it Rate of convergence of Wong-Zakai approximations
for stochastic partial differential equations}
Appl. Math. Optim. {\bf 54} (2006), 315--341.


\bibitem{IW}
\textsc{Ikeda, N., Watanabe, S.,} Stochastic
Differential Equations and Diffusion Processes, 
 North Holland, Amsterdam, 1981.


\bibitem{Ma} {\sc V. Mackevi\v cius}, 
\textit{ On the support of the solution of stochastic differential
equations}, Livetuvos Matematikow Rinkings  {\bf 36(1)} (1986),  91--98.



\bibitem{MS02}
 {\sc J.L. Menaldi, S.S. Sritharan}, \textit{Stochastic 2-D Navier-Stokes equation},
Appl. Math. Optim. {\bf 46} (2002) 31--53. 


\bibitem{M-SS1} {\sc A Millet, M. Sanz-Sol\'e},
  \textit{The support of the solution to a hyperbolic SPDE},
 Probability Theory and Related  Fields {\bf 98} (1994), 361--387

\bibitem{M-SS2}  {\sc A Millet,  M. Sanz-Sol\'e}, 
{\it A simple proof of the support theorem for
diffusion processes}, S\'eminaire de Pro\-ba\-bi\-li\-t\'es XXVIII,
 Lecture Notes in Mathematics  {\bf  1583} (1994),  36--48.


\bibitem{Nak} {\sc  T. Nakayama},
 \textit{Support theorem for mild solutions of SDE's in Hilbert spaces},
 J. Math. Sci. Univ. Tokyo {\bf 11} (2004) 245--311.



\bibitem{protter} \textsc{Protter, P.,} \textit{Approximations of solutions
of stochastic differential equations driven by semi-Martingales}, The Annals
of Probability \textbf{13} (3) (1985), 716-743.


\bibitem{SV} {\sc  D. W. Stroock,   S.R.S. Varadhan}, \textit{ On the support  of
diffusion processes with applications
to the strong maximum principle},
Proc. of Sixth Berkeley Sym. Math. Stat. Prob. III, Univ. California Press,
Berkeley, 333--359, 1972.


\bibitem{tessitorezab} \textsc{Tessitore, G., Zabczyk, J., }
\textit{Wong-Zakai approximations of
stochastic evolution equations,} 
Journal of Evolution Equations {\bf  6} (2006), 621-655.


\bibitem{Tribel78}
 {\sc H. Triebel},  Interpolation Theory, Functional Spaces and
Differential Operators, North Holland, Amsterdam, 1978.  


\bibitem{twardo} \textsc{K. Twardowska,} \textit{Wong-Zakai
approximations for stochastic differential equations}, Acta Applic. Math.
\textbf{43} (1996) 317-359.


\bibitem{Twa-padova}
{\sc K. Twardowska},  \textit{An approximation theorem of Wong-Zakai type for stochastic
Navier-Stokes equations}, Rend. Sem. Mat. Univ. Padova
{\bf  96} (1996), 15--36.

\bibitem{Twa-gyor}
{\sc K. Twardowska},
\textit{On support theorems for stochastic nonlinear partial differential equations},
in: I. Csiszar, Gy. Michaletzky (Eds.),
Stochastic differential and difference equations (Gy\"or, 1996),
Progr. Systems Control
Theory {\bf  23}, Birkh\"auser, Boston, 1997, 309--317.


\bibitem{VKF}
{\sc  M. I. Vishik, A. I. Komech, A. V. Fursikov},
\textit{Some mathematical problems of statistical hydromechanics},
Russ. Math. Surv. {\bf 34(5)} (1979),   149--234.



\bibitem{WZ}
 \textsc{E. Wong, M. Zakai,} \textit{%
Riemann-Stieltjes approximations of stochastic integrals,}  Z.
Wahrscheinlichkeitstheorie u. verw. Gebiete \textbf{12} (1969), 87-97.



\end{thebibliography}
\end{document}